\documentclass{amsart}
\usepackage{amsfonts}
\usepackage{amsmath}
\usepackage{amssymb}
\usepackage{amsthm}
\usepackage{array}
\usepackage{booktabs}
\usepackage{color}
\usepackage{enumerate}
\usepackage{esint}
\usepackage{lineno,hyperref}
\usepackage[left=1in,right=1in,top=1in,bottom=1in]{geometry}
\usepackage{graphicx}
\usepackage{relsize}
\usepackage{subfigure}
\usepackage{times}
\usepackage{color}

\newcommand{\fra}{\mathfrak{a}}
\newcommand{\frb}{\mathfrak{b}}
\newcommand{\frc}{\mathfrak{c}}
\newcommand{\frd}{\mathfrak{d}}
\newcommand{\frf}{\mathfrak{f}}
\newcommand{\frg}{\mathfrak{g}}
\newcommand{\frm}{\mathfrak{m}}
\newcommand{\frn}{\mathfrak{n}}
\newcommand{\frs}{\mathfrak{s}}
\newcommand{\frw}{\mathfrak{w}}
\newcommand{\frX}{\mathfrak{X}}
\newcommand{\frC}{\mathfrak{C}}

\newcommand{\pc}{P_{{\rm c}}}
\newcommand{\po}{P_{0}}

\newcommand{\1}{\mathbf{1}}

\newcommand{\ga}{\gamma}

\newcommand{\tang}{ \boldsymbol{\tau} }

\newcommand{\bveps}{\boldsymbol{\varepsilon}}
\newcommand{\bsigma}{\boldsymbol{\sigma}}

\newcommand{\nrml}{\mathfrak{n}}

\newcommand{\veps}{\varepsilon}
\newcommand{\vrho}{\varrho}

\newcommand{\R}{\mathbb{R}}
\newcommand{\Z}{\mathbb{Z}}
\newcommand{\cm}{\mathfrak{c}}

\newcommand{\vv}{\mathbf{v}}

\newcommand{\ww}{\mathbf{w}}

\newcommand{\rd}{\mathrm{d}}

\newcommand{\re}{\mathrm{e}}

\newcommand{\T}{\mathbb{T}^1}

\newcommand{\N}{\mathbb{N}}

\newcommand{\cH}{\mathcal{H}}
\newcommand{\vth}{\vartheta}
\newcommand{\cG}{\mathcal{G}}
\newcommand{\cF}{\mathcal{F}}
\newcommand{\cV}{\mathcal{V}}

\newcommand{\cT}{\mathcal{T}}
\newcommand{\cR}{\mathcal{R}}

\newtheorem{theorem}{Theorem}[section]
\newtheorem{lemma}[theorem]{Lemma}
\newtheorem{remark}[theorem]{Remark}
\newtheorem{proposition}[theorem]{Proposition}

\begin{document}
\title{Fronts Under Arrest II: Analytical Foundations}
\begin{abstract}{
We study a class of minimal geometric partial differential equations that serves as a framework to understand the evolution of boundaries between states in different pattern forming systems.  The framework combines normal growth, curvature flow and nonlocal interaction terms to track the motion of these interfaces.  This approach was first developed to understand arrested fronts in a bacterial system.  These are fronts that become stationary as they grow into each other. This paper establishes analytic foundations and geometric insight for studying this class of equations.  In so doing, an efficient numerical scheme is developed and employed to gain further insight into the dynamics of these general pattern forming systems.}
\end{abstract}
\author[von Brecht]{James H. von Brecht}
\address[von Brecht]{\newline
	Department of Mathematics and Statistics , California State University, Long Beach, CA 90840	}
	\email[vonBrecht]{\href{James.vonBrecht@csulb.edu}{James.vonBrecht@csulb.edu}}

\author[McCalla]{Scott G. McCalla}
\address[McCalla]{\newline
	Department of Mathematical Sciences, Montana State University, Bozeman, MT 59717 	}
	\email[McCalla]{\href{scott.mccalla@montana.edu}{scott.mccalla@montana.edu}}

\author[Kim]{Eun Heui Kim}\thanks{The work of Kim is supported by and done while serving at the National Science Foundation. Any opinion, findings, and conclusions or recommendations expressed in this material are those of the authors and do not necessarily reflect the views of the National Science Foundation.}

	\address[Kim]{\newline
	Department of Mathematics and Statistics , California State University, Long Beach, CA 90840, and National Science Foundation, Alexandria, Virginia 22314.
	}
	\email[Kim]{\href{EunHeui.Kim@csulb.edu}{EunHeui.Kim@csulb.edu}}
\maketitle

\section{Introduction}

\noindent 
We study the motion of non-degenerate planar interfaces undergoing geometric dynamics governed by a combination of curvature, constant normal growth and nonlocal forcing. To be precise, let $\Gamma = \left\{ \ga_1,\ldots,\ga_m \right\}$ denote a finite collection of closed, planar curves. We consider a family of models where each curve $\ga_i : \T \times [0,T] \to \R^2$ in this collection evolves in time
\begin{equation}\label{eq:introdyn}
\partial_{t} \ga_i  = v_i \nrml_i \qquad \text{for all} \qquad i \in [m] := \{1,\ldots,m\}
\end{equation}
according to a family of normal velocities $v_i : \T \times [0,T] \to \R$ that determine the flow. The normal velocities we consider take the form
$$
v_i = \kappa_i + c_i + f_i,
$$
where $\kappa_i$ denotes the signed curvature along $\ga_i$ and $c_i \in \R$ specifies a constant  growth rate or death rate in the normal direction. We then select a family $\cG = \{ g_{ij}(s) : (i,j) \in [m] \times [m] \}$ of kernel functions and set
\begin{equation}\label{eq:nonlocstructure}
f_i(x,t) := \sum^{m}_{j=1} \int_{\T} g_{ij} \left( \frac12| \ga_i(x,t) - \ga_j(y,t)|^2 \right) | \dot \ga_j(y,t) | \, \rd y
\end{equation}
for the nonlocal force acting on each curve. The kernels $g_{ij}(s)$ encode the interactions between interfaces in the system. 

Models similar to \eqref{eq:introdyn} describe a wide variety of natural systems characterized by interfaces that evolve in a predictable fashion \cite{rubinstein1989fast,petrich1994nonlocal,goldstein1996interface,BeerPNAS09,BeerPNAS10}. Prototypical examples occur in materials science and metallurgy, where the interfaces between different system states evolve under a mean curvature flow \cite{pego1989front,alfaro2008singular}. Such models also describe complex systems whose dynamics reduce to an interfacial evolution between the boundaries of different phases. Early examples of this type of reduction typically exhibited a local flow for the evolving interface. For example, a balance between a fast reaction term and a slow diffusion term in bi-stable equations, such as the Allen--Cahn equation or Fitzhugh--Nagumo equation, can lead to the formation of domain walls that will then evolve according to a mean curvature flow together with a constant normal forcing. The reduction of a reaction-diffusion system may also lead to effective nonlocal interactions between interfaces. As a simple one-dimensional illustration, consider a singularly perturbed system that exhibits a single fast variable in conjunction with some number of slow variables. On a long-enough time-scale, the full dynamics may reduce to a set of coupled ordinary differential equations 
\begin{eqnarray*}
\dot{x}_i=\alpha+\sum^{m}_{j=1} (-1)^j g \left( x_i-x_j \right) 
\end{eqnarray*}
for the transition locations $x_i(t)$ between the two states of the fast variable. This is a one-dimensional instance of \eqref{eq:introdyn}, where the slow variables are now implicitly included through the nonlocal interaction kernel. 
In the planar case, these transitions regions are fast fronts that evolve with respect to an approximately constant slow field. Their dynamics exhibit motion by mean curvature in conjunction with a normal motion, and nonlocal interactions represent a correction to the slow fields. Following \cite{goldstein1996interface}, these can be accounted for in the planar case with the additional assumption that the shape of the front is approximately constant for small curvatures of the boundary.  This assumption then leads to the nonlocal terms only acting in the normal direction at each point on the evolving interface. By contrast, the semistrong regime for singularly perturbed systems is characterized by different components exhibiting strongly different asymptotic decay rates.  Certain fronts might then have a fast component that is exponentially localized but whose evolution experiences algebraic corrections due to the slow components interacting strongly.  In this case, the dynamics of fast fronts can sometimes be reduced to normal growth plus a nonlocal interaction \cite{semistrong,van2010front}. In this way, the evolving interface models \eqref{eq:introdyn} can be seen as a reduction from a standard reaction-diffusion formalism.

\begin{figure}[t]
\centering
        \includegraphics[height=1.3in]{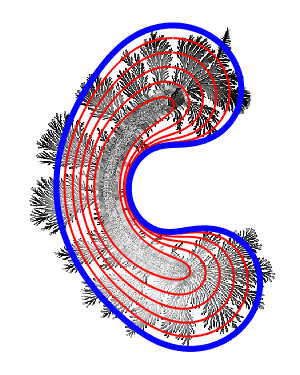}
         \includegraphics[height=1.3in]{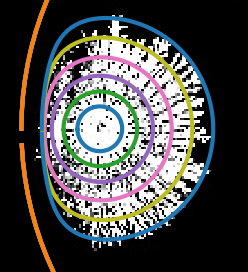}
        \includegraphics[height=1.3in]{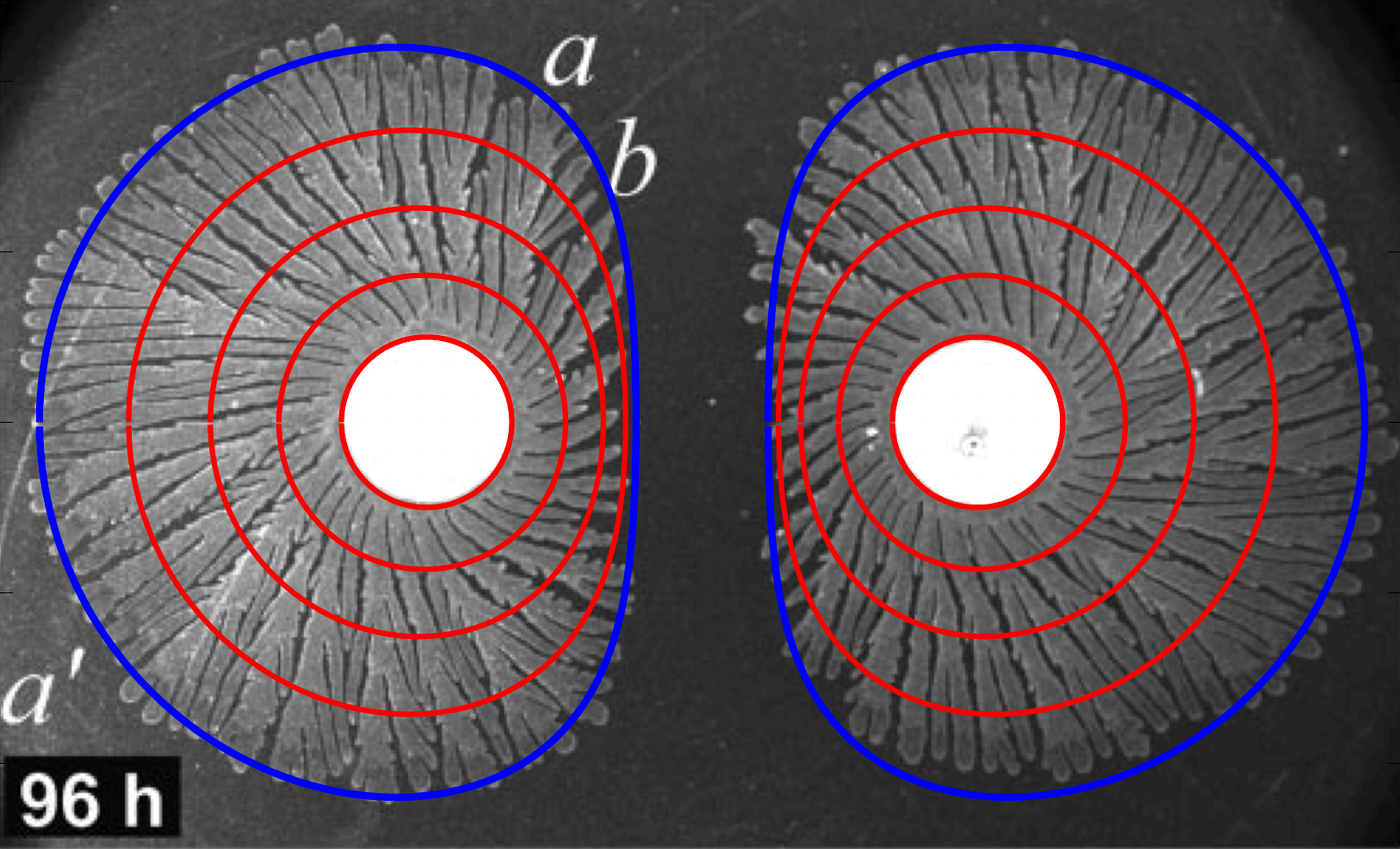}
        \includegraphics[height=1.3in]{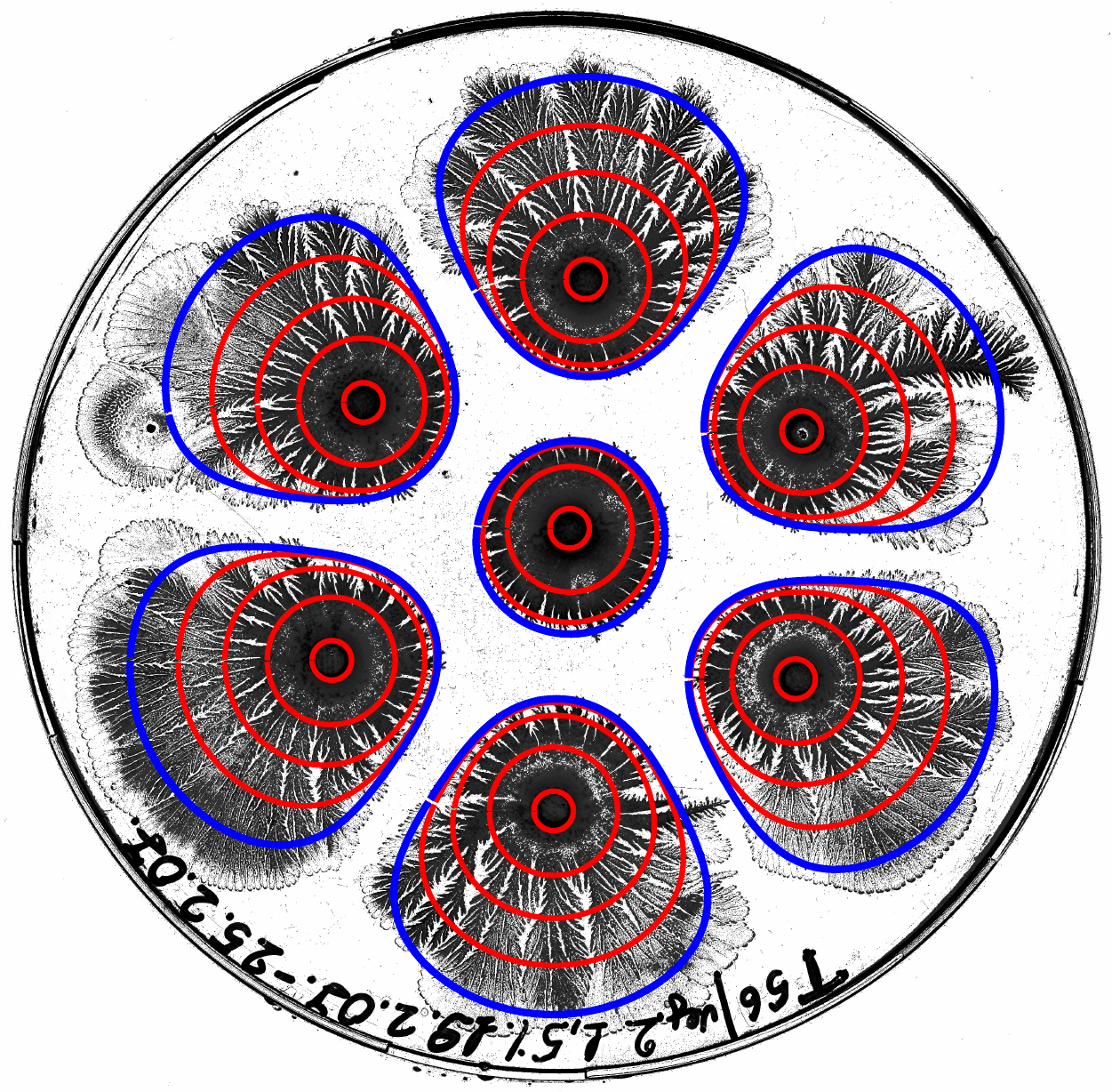}
        \caption{Colonies of \emph{P. dendritiformis} appear to exhibit nonlocal effects as they grow. In the first figure at far left, a single colony leads to self arrest in the interior boundary as it grows \cite{BeerPNAS09,MvB16}.  In the second figure, the colony is repelled from the orange barrier on the left through its own toxin emission (experimental results courtesy of Jason Zeng and James Wilking). In the third \cite{BeerPNAS09,MvB16} and fourth figures \cite{SevenColony,MvB16}, sibling colonies interact and arrest each others' growth creating complex patterns . In each figure, time slices of the evolution \eqref{eq:introdyn} are superimposed over the evolving colony.}
        \label{7colfig}
\end{figure}
The model \eqref{eq:introdyn} can also arise in more general physical settings that are not tied to an underlying reaction-diffusion dynamic. These models can be taken \emph{a priori} as a modeling paradigm wherein almost any interaction kernel could conceivably be used. For example, such a nonlocally forced curvature flow was introduced in \cite{MvB16} to capture bacterial colony formation for a subtype of \emph{Paenibacillis dendritiformis} that diffuses a toxin called sibling lethal factor (Slf) \cite{BeerPNAS09, BeerPNAS10}. This compound prevents the growth of the other competing colonies and leads to self-inhibition which amounts to a repulsive interaction between the colonies.  The interfaces then can approach a stationary front when their forward motions are dominated by these repulsive interactions, and this stationary structure is called an \emph{arrested front}: see Figure~\ref{7colfig}. . The numerical scheme is detailed in Section~\ref{sec:Numeric}). Even for this simple experimental system, a wide range of phenomena from self and boundary interactions to full nonlocal interactions between different colonies is exhibited. A similar model can also arise from studying interfaces between particle types in a collisional model \cite{mccalla2018consistent}. 

The promise of this modeling approach motivates us to develop an analytical foundation for the rigorous analysis of general nonlocal curve flows. While the mathematical properties of curvature flows and related evolutionary interfacial systems have been well studied, the class of these models that allow for different interfaces to interact with each other has received relatively little attention. To illustrate the need for such an analysis, we recall the situation for a pure motion by mean curvature where the long-term behavior of solutions is very well understood. Initially embedded curves in the plane remain embedded for all time \cite{AB11,Huisken}, eventually become circular and shrink to a point (the Gage--Hamilton--Grayson Theorem, c.f. \cite{GH,Grayson,AB11}) while their distortion (in the sense of Gromov) converges to that of the standard circle. By contrast, under motion by mean curvature a generic non-embedded curve will form curvature singularities in finite time. We therefore have a nearly one-to-one correspondence between embeddedness of the initial datum and global well-posedness of the dynamic. Such broad statements fail dramatically when we include nonlocal effects. The curves driven by \eqref{eq:introdyn} can grow, intersect, and establish stationary complex shapes that do not necessarily become circular, in sharp contrast to the curve flow by mean curvature. Moreover, some embedded curves may lose embeddedness in finite time, while even for the same choice of model parameters, others may not. In other words, it is unlikely that an analogous dichotomy can be established between ``globally well-posed'' and ``finite-time singularity'' initial conditions. Nevertheless, it is precisely this variety of dynamic behaviors that makes these systems useful for modeling real world phenomena. 

We therefore wish to develop some broadly-applicable tools to help understand these evolutions over a reasonable range of possible applications. From an analytical point-of-view, we wish to allow for a wide range of nonlocal forcing terms while still having a well-posed dynamic. A good deal of our contribution involves suitable nonlocal estimates that, in turn, allows us to formulate a well-behaved dynamic in as general a setting as possible. Additionally, unlike the motion by mean curvature case we cannot guarantee embededdess for all time. We therefore develop some machinery for computing the evolution of various measures of embeddedness, and this allows us to understand how embeddedness is preserved or lost for various model instances. Finally, we conclude by documenting an efficient and robust numerical scheme for simulations. We proceed as follows. In section~\ref{sec:PreMat} we set up the notation and basic lemmas that are used in the paper, and in
section~\ref{sec:Rep} we derive two alternate representations of \eqref{eq:introdyn} that are found to be useful. We then turn  to the heart of the matter in section~\ref{sec:Nonlocal}, which provides estimates on the nonlocal forcing term for a broad range of kernels under various geometric hypotheses on the data. We leverage these estimates in section~\ref{sec:Wellposed} to provide local well-posedness and regularity results. We then develop some machinery to study how the geometric properties of the curves in section~\ref{sec:Embeded}, and illustrate how broad statements (e.g. a type of global well-posedness/embeddedess dichotomy) necessarily fail. The last section details the numerical scheme used for these analyses, while the appendix ~\ref{sec:Appendix} proves at few technical lemmas that are needed, but not central to the main thrust of the analysis. Apart from the estimates in section ~\ref{sec:Nonlocal}, this analysis does not rely on the specific structure \eqref{eq:nonlocstructure} of the nonlocality. In this sense our analysis holds for a quite general class of nonlocal forcings; anything satisfying the bounds in section ~\ref{sec:Nonlocal} will do.

\section{Preliminary Material}\label{sec:PreMat}
We begin by introducing our preferred notation and collecting a few standard facts that we will refer to frequently. In various estimates we use the notations $a\vee b = \max\{a, b\}$ and $a\wedge b= \min\{a,b\}$ interchangeably. The notation $\frC$ refers to a generic, universal constant whose value may change from line to line. The decorated version
$$
\frC^{*}(x_1,\ldots,x_n)
$$
refers to some generic continuous function of its arguments; the underlying function itself may change from line to line as well.

For spatial dependence we use $I_{\pi} := [-\pi,\pi]$ and $\T$ to denote domains, where use of the former signifies the domain of an arbitrary function and the latter serves to emphasize periodicity when we wish to impose it. Given some integrable $f:\T \mapsto \R^{d}$ we use
\begin{align*}
\bar f = \mu_{f} := \fint_{\T} f(z) \, \rd z
\end{align*}
for its scalar-valued or vector-valued mean. The operator $f \mapsto \pc f$ stands for the mean-zero primitive
$$
\pc f(x) := \fint_{\T} (z-\pi) f(z) \, \rd z + \int^{x}_{-\pi} f(z) \, \rd z 
$$
of an integrable function, which defines an absolutely continuous function $\pc f \in C(I_{\pi})$ in general and a periodic function $\pc f\in C(\T)$ in the specific case that $\mu_f$ vanishes. The analogous notation
$$
\po f(x) := \int^{x}_{-\pi} f(z) \, \rd z - (x+\pi)\mu_{f}
$$
defines the operator $f \mapsto \po f \in C(\T)$ furnishing a zero-Dirichlet primitive of $f - \mu_f,$ which again defines an absolutely continuous, periodic function. In particular, the basic integration by parts identity
\begin{align*}
\pc \big(fg'\big) &= fg -\pc\big( f^{\prime} g \big)  - \mu_{fg}
\end{align*}
holds whenever $f^{\prime}$ and $g^{\prime}$ exist in the weak sense as integrable functions. For a given square-integrable function $\psi: \T \to \R^{d}$ we shall use the notation
$$
\|\psi\|^{2}_{L^{2}(\T)} = \int_{\T} |\psi(x)|^2 \, \rd x := 2\pi \fint_{\T} |\psi(x)|^2 \, \rd x
$$
to denote the $L^{2}(\T)$-norm, and we shall use
\begin{align*}
\hat{\psi}_{k} = \fint_{\T} \gamma(x) \re^{- i k x } \, \rd x \qquad \text{and} \qquad \psi(x) = \sum_{k \in \Z} \hat{\psi}_k \, \re^{ ikx }
\end{align*}
to denote the forward and inverse Fourier transforms, respectively. After setting $\Z_0 := \Z \setminus \{0\},$ the usual relations
\begin{align*}
\|\psi\|^{2}_{L^{2}(\T)} = 2\pi \sum_{k \in \Z} |\hat{\psi}_k|^2, \quad \|\psi\|^{2}_{\dot{H}^{s}(\T)} = 2\pi \sum_{k \in \Z_0} |k|^{2s} |\hat{\psi}_k|^2 \quad \text{and} \quad \|\psi\|^{2}_{H^{s}(\T)} = |\hat{\psi}_0|^2 + \| \psi\|^{2}_{\dot{H}^{s}(\T)}
\end{align*}
then provide the $L^{2}(\T)$ norm an equivalent definition of the $H^{s}(\T)$ Sobolev norm. In a similar fashion, we shall use
\[
\| \psi \|^{p}_{\dot W^{k,p}(\T)} := \int_{\T} \big|\psi^{(k)}(x)\big|^{p} \, \rd x \qquad \text{and} \qquad \|\psi\|^{p}_{W^{k,p}(\T)} = |\hat{\gamma}_0|^p + \| \psi\|^{p}_{\dot{W}^{k,p}(\T)}
\]
for integer-order, non-Hilbertian Sobolev norms while
$$
\| \psi \|_{L^{\infty}(\T) } := \underset{x \in \T}{{\rm ess}\,{\rm sup}} \, |\psi(x)|
$$
denotes the sup-norm. Finally, when dealing with a collection of $m$ functions $\phi_{i},i \in [m]$ we reserve the notation
$$
\mathcal{H}^{s}_{m} := H^{s}\left(\T\right) \times \cdots \times H^{s}\left(\T\right) 
$$
for the set of all $m$-tuples $\Phi = (\phi_1,\ldots,\phi_m)$ with $H^{s}(\T)$ regularity in each component. We use the $\ell^{\infty}$-norm across components
$$
\|\Phi\|_{ \cH^{s}_{m} } := \max_{ i \in [m] } \;\, \| \phi_{i} \|_{H^{s}(\T)}
$$
for such tuples. With these conventions in hand, we may record for later reference the following instances of Gagliardo-Nirenberg embeddings.
\begin{lemma}
Assume that $\psi \in W^{1,1}(\T)$ and let 
$$
\mu_{\psi} := \fint_{\T} \psi(x) \, \rd x
$$
denote its mean. Then the embeddings
\begin{align}\label{eq:gns} 
\|\psi - \mu_{\psi}\|_{L^{p}(\T)} &\leq \left(\frac34\right)^{\alpha} \| \psi_{x} \|^{\alpha}_{L^{2}(\T)} \| \psi -  \mu_{\psi}\|^{1 - \alpha}_{L^{1}(\T)} \qquad &&\alpha = \frac{2p-2}{3p}, \;\; p \geq 1 \nonumber \\
\|\psi - \mu_{\psi}\|_{L^{p}(\T)} &\leq \| \psi_{x} \|^{\alpha}_{L^{2}(\T)} \| \psi -  \mu_{\psi}\|^{1 - \alpha}_{L^{2}(\T)} \qquad &&\alpha = \frac{p-2}{2p}, \;\;\;\, p \geq 2  \\
\|\psi - \mu_{\psi} \|_{L^{p}(\T)} &\leq \left(\frac12\right)^{\alpha}\| \psi_{x} \|^{\alpha}_{L^{1}(\T)} \| \psi -  \mu_{\psi} \|^{1 - \alpha}_{L^{q}(\T)} \qquad &&\alpha = \frac{p-q}{p}, \;\;\;\, p \geq q \nonumber
\end{align}
hold.
\end{lemma}
\noindent For scalar-valued functions $f(x,t)$ or vector-valued functions $\ga(x,t)$ involving a space-time variable $(x,t)$ we use $Q_{T} := \T \times [0,T]$ for a parabolic cylinder. The interchangeable family of notations
$$
\dot f(x,t) = f^{\prime}(x,t) = \partial_x f(x,t) \qquad \text{and} \qquad \ddot \ga(x,t) = \ga^{\prime\prime }(x,t) = \partial_{xx} \ga(x,t)
$$
will always refer to spatial derivatives; we will always display temporal derivatives $\partial_t f$ or $\partial_{tt} \ga$ explicitly. From time to time we will find it convenient to view bi-variate functions $\phi(x,t)$ as defining a one-parameter family of mappings $\{ \phi(\cdot,t) \}_{t \in (a,b)}$ from a temporal interval $(a,b)$ into some appropriate Banach space. In such cases we find it convenient to omit the spatial dependence and simply use $\{ \phi(t) \}_{t \in (a,b)}$ to refer to the function family. So, an $m$ tuple of the form
$$
\Phi(t) = \left( \phi_1(t),\ldots,\phi_m(t) \right) \in \cH^{s}_{m} \qquad \text{for} \qquad t \in [0,T]
$$
will therefore refer to a mapping $\Phi : [0,T] \mapsto \cH^{s}_{m}$ that, for each $t$, furnishes a collection of $m$ functions with $\cH^{s}_{m}$ regularity. By $C^{s}_m(T) := C([0,T];\mathcal{H}^{s}_m)$ we then mean the collection of such mappings that vary continuously in time with respect to the max-norm
$$
\| \Phi \|_{C^{s}_{m}(T)} := \max_{t \in [0,T]}\;\, \| \Phi(t)\|_{\cH^{s}_{m}}
$$
across components. Similarly, for an $m$-tuple $\bveps(t) := (\veps_1(t),\ldots,\veps_m(t))$ of scalars that vary in time we use the definitions
$$
\| \bveps(t) \|_{\R^m} := \| \bveps(t)\|_{\ell^{\infty}(\R^{m})} \qquad \text{and} \qquad \|\bveps\|_{C_{m}(T)} := \max_{t \in [0,T]} \;\, \|\bveps(t)\|_{\R^{m}}
$$
for the space $C_{m}(T)$ of $m$ continuous functions endowed with the max-norm.
\subsection*{Immersions and Embeddings:} We will frequently need to appeal to elementary notions regarding immersions, embeddings and their various representations. Let $\ga: \T \mapsto \R^d$ denote a closed, $H^{2}(\T)$-regular curve and
$$
\ell(\ga) := \int_{\T} |\dot \ga(x)| \, \rd x
$$
its length. By an $H^{2}(\T)$-\emph{immersion} we simply mean an $H^{2}(\T)$ map $\ga:\T \mapsto \R^{d}$ whose minimal speed
$$
\frs_{*}(\ga) := \min_{x \in \T} \;\, |\dot \ga(x)|
$$
does not vanish. Any immersion induces a pair $(\eta_{\gamma},\xi_{\gamma})$ of $H^{2}\big( I_{\pi} \big)$ maps
$$
\eta_{\ga}(x) := -\pi + \frac{2\pi}{\ell(\ga)}\int^{x}_{-\pi} |\dot \ga(s)|\, \rd s \qquad \text{and} \qquad \xi_{\ga}\big( \eta_{\ga}(x) \big) = x
$$
of $I_{\pi}$ onto $I_{\pi}$ that convert $\ga$ into a canonical parametrization $\psi_{\ga} := \ga \circ \xi_{\ga},$ which we refer to as its constant speed representation. If $
\ga$ has constant speed we shall frequently use some variant of the notation
$$
|\dot \ga(x)| \equiv \sigma(\ga) := \frac{\ell(\ga)}{2\pi}
$$
to denote the constant speed of the curve. As we shall frequently find it convenient to work in the class of constant speed immersions, we need a means of asserting that transitions to constant speed coordinates behave well under compositions in an appropriate sense. The following lemma will suffice for these purposes. 
\begin{lemma}\label{lem:chofvar}
Assume that $f \in H^{1}(\T)$ and that $\ga_i \in W^{1,1}(\T), i \in \{0,1\}$ are immersions. Let
$$
\eta_i(x) := -\pi + \frac{2\pi}{\ell(\ga_i)} \int^{x}_{-\pi} |\dot \ga_i(z)| \, \rd z \qquad \text{and} \qquad \xi_i \circ \eta_i(x) = x
$$
denote their transitions to constant speed coordinates. Then the estimate
$$
\|f - f \circ \eta_1 \circ \xi_0\|_{L^{2}(\T)} \leq \mathfrak{C}\|\dot f\|_{L^{2}(\T)}\frac{\| \dot \ga_1 - \dot \ga_0\|_{L^{1}(\T)}}{\ell(\ga_0) \vee  \ell(\ga_1)}
$$ 
holds for $\mathfrak{C}>0$ a universal constant. Moreover, the map $\ga_i \mapsto \psi_i$ is locally Lipschitz with respect to the $H^{1}(\T)$ topology, in the sense that the bound
$$
\|\psi_1 - \psi_0\|_{H^{1}(\T)} \leq \mathfrak{C}^{*}\left( \frac{\ell(\ga_0)\vee\ell(\ga_1)}{\frs_{*}(\ga_0)\wedge \frs_{*}(\ga_1)},\frac{\|\ddot \ga_0\|_{L^{2}(\T)}\vee\|\ddot \ga_1\|_{L^{2}(\T)}}{\frs_{*}(\ga_0)\wedge \frs_{*}(\ga_1)}\right)\|\ga_1 - \ga_0\|_{H^{1}(\T)}
$$
holds for $\mathfrak{C}^{*} : \R^{2} \mapsto \R$ some continuous, increasing function of its arguments.
\end{lemma}
\noindent We provide a simple proof in lemmas \ref{lem:chofvarapp} and \ref{lem:equivboundapp} in the appendix. In a similar vein, we shall also have occasion to consider a family
\begin{align}\label{eq:cspdrep}
\eta(x,t) := -\pi + \frac{2\pi}{\ell(\ga(t))}\int^{x}_{-\pi} |\dot \ga(z,t)| \, \rd z \qquad \text{and} \qquad \eta( \xi(x,t),t) = x
\end{align}
of such maps induced by a one-parameter collection $\{\ga(t)\}_{t \in (a,b)}$ of immersions. The corresponding constant speed representation $\psi(x,t) = \ga( \xi(x,t),t)$ of this family then obeys the transport relation
\begin{align}\label{eq:cspdder}
\psi_{t}(x,t) + \left( \frac{2\pi}{\ell(\ga(t))} \right)^2 P_0 \big( \langle \dot \psi , \dot \vv \rangle \big)\dot \psi(x,t) = \ga_{t}( \xi(x,t),t ) := \vv(x,t)
\end{align}
provided that the requisite derivatives exist in the appropriate sense.

For $H^{2}(\T)$-immersions we may define the unit-tangent and squared curvature
$$
\tang_{\ga}(x) := \frac{\dot \ga(x)}{|\dot \ga(x)|} \qquad \text{and} \qquad \kappa^2_{\ga}(x) := \frac{|\dot \tang_\ga(x)|^2}{|\dot \ga(x)|^2}
$$
provided we interpret $\dot \tang_\ga$ in the sense of weak derivatives. The \emph{bending energy} of $\ga$ then refers to the total squared curvature
$$
\kappa^{2}_{2}(\ga) := \int_{\T} \kappa^{2}_{\ga}(x) |\dot \ga(x)| \, \rd x = \left( \frac{2\pi}{\ell(\ga)} \right)^{3} \int_{\T} |\ddot \psi_{\ga}(x)|^2 \, \rd x,
$$
which has units of inverse length and obeys the lower bound
$$
\ell(\ga) \kappa^{2}_{2}(\ga) \geq (2\pi)^{2}
$$
for all immersions. The inverse of the bending energy
$$
\ell_{\kappa}(\ga) := \frac1{\kappa^2_2(\ga)} \leq \frac{\ell(\ga)}{\,\,(2\pi)^{2}}
$$
defines a natural length-scale along the curve, and so the dimensionless ratio
$$
\frac{ \ell(\ga)}{\ell_{\kappa}(\ga)}
$$
provides a dimensionless measure of the deviation of $\ga$ from a standard circle. This quantity will appear frequently in our estimates of nonlocal integrals.

By an \emph{embedding} we simply mean an immersion without self-intersections. We may quantify embeddedness precisely using well-known machinery. Given a pair $(x,y) \in \T \times \T$ with $x<y$ we let
$$
\ell_{\ga}(x,y):= \int^{y}_{x} |\dot \ga(z)| \, \rd z \qquad \text{and} \qquad D_{\ga}(x,y) := \min\left\{ \ell_{\ga}(x,y)  , \ell(\ga) - \ell_{\ga}(x,y) \right\}
$$
denote the arc-length and distance along $\ga$ between the points $\ga(x)$ and $\ga(y),$ respectively. The quantity
$$
\delta_{\infty}(\ga) := \sup_{y\neq x} \;\, \frac{D_{\ga}(x,y)}{|\ga(x) - \ga(y)|}
$$
then defines the \emph{distortion} of a closed curve $\ga$ in the sense of Gromov. A similar notion applies to sub-arcs of a closed curve. Indeed, if $x<y<x+2\pi$ then the pointwise restriction
$$
\alpha_{x,y} := \left.\ga\right|_{[x,y]}
$$
forms an immersed sub-arc of $\ga,$ in which case we use the modified definition
$$
\delta_{\infty}\big( \alpha_{x,y} \big) := \sup_{ x \leq s < t \leq y} \;\, \frac{\ell_{\ga}(s,t)}{|\ga(s) - \ga(t)|} 
$$
for its distortion. 

A few basic properties of the distortion will prove useful. The distortion is both parametrization and scale invariant, and a standard fact (c.f. \cite{Wirtinger}) asserts it attains its minimum value 
$$ \frac{\pi}{2} = \inf_{\ga} \;\, \delta_{\infty}(\ga)$$ 
at the standard embedding of any circle. A simple argument from the Taylor expansion
$$
\ga(y) - \ga(x) = \ell_{\ga}(x,y)\tang(x) + \int^{y}_{x} \ell_{\ga}(v,y) \dot \tang_{\ga}(v) \, \rd v 
$$
and the Cauchy-Schwarz inequality
$$
\left| \int^{y}_{x} \ell_{\ga}(v,y) \dot \tang_{\ga}(v) \, \rd v \right| \leq \frac1{\sqrt{3}}\ell^{\frac32}_{\ga}(x,y)\kappa_{2}(\ga)
$$
suffices to establish the implication
$$
D_{\ga}(x,y) = \ell_{\ga}(x,y) \leq \ell_{\kappa}(\ga)\qquad \longrightarrow \qquad \delta_{\ga}(x,y) := \frac{D_{\ga}(x,y)}{|\ga(x) - \ga(y)|} \leq \frac{\sqrt{3}}{\sqrt{3} - 1},
$$
and so any $\ga \in H^{2}(\T)$ without self-intersections has finite distortion. Similarly, any sub-arc $\alpha_{x,y}$ of length less than $\ell_{\kappa}(\ga)$ has finite distortion. By decomposing $\T$ into 
$$
\T = \cup^{N}_{k=1} \; [x_{k-1},x_{k}) \qquad \text{with} \qquad x_k := \xi_{\ga}\left( -\pi + \frac{2\pi k}{N} \right),
$$
it follows that $\ga$ partitions into at most
$$
N(\ga) := \left \lceil \frac{\ell_{\ga}}{\ell_{\kappa}(\ga)} \right \rceil
$$
almost-disjoint subarcs $\alpha_{k} = \alpha_{x_{k-1},x_k}$ that have distortion bounded by $3$ and are of equal length. Finally, the distortion remains stable with respect to the $H^{2}(\T)$-topology; the inequality
$$
\frac1{\delta_{\infty}(\ga_2)} \geq \frac1{\delta_{\infty}(\ga_1)}\left( \frac{1 - R\delta_{\infty}(\ga_1)}{1+R} \right) \qquad  \qquad R := \frac{\mathrm{Lip}\big(\ga_2-\ga_1\big)}{\frs_{*}(\ga_1)}
$$
and the embedding $H^2(\T) \subset C^{0,1}(\T)$ reveal that the distortion defines a locally Lipschitz function with respect to $H^{2}(\T)$ convergence.

\section{Representations of the Dynamic}\label{sec:Rep}
We elucidate the meaning for the dynamics \eqref{eq:introdyn} by representing it in three equivalent ways. The elementary calculus identity
$$
\tang_i(x,t) := \frac{\dot \ga_i(x,t)}{|\dot \ga_i(x,t)|}  \quad \quad \quad  \kappa_i(x,t)\nrml_i(x,t) = \frac{\dot \tang_{i}(x,t)}{ |\dot \ga_i(x,t)|}
$$ 
leads to the most natural (and standard) representation, for we may simply define the normal vector
\begin{equation}\label{eq:nrmldef}
\nrml_i(x,t) := \tang^{\perp}_i(x,t)\qquad \text{where} \qquad \begin{bmatrix} a \\ b \end{bmatrix}^{\perp} := \begin{bmatrix} \;\;\,b \\ -a \end{bmatrix}
\end{equation}
in some consistent fashion and then use the quasi-linear PDE
\begin{align}\label{eq:basic}
\partial_{t} \ga_i(x,t) := \frac{\dot \tang(x,t)}{|\dot \ga_i(x,t)|} + F_i(x,t)\nrml_{i}(x,t) \qquad \qquad
 F_i(x,t) := f_i(x,t) + c_i
\end{align}
to represent the dynamic. Selecting the opposite orientation $-\tang^{\perp}_i$ for the normal $\nrml_i$ simply amounts to a different choice of sign for the forcing. In addition to the standard formulation \eqref{eq:basic}, two alternate descriptions of the dynamic prove useful in our study.\\

\noindent\textbf{The Shape-Scaling Representation:} Given a solution to \eqref{eq:basic} existing on $\T \times [0,T_*)$ for some $T_*>0$, recall that we may assign to each $\ga_i \in \Gamma$ a corresponding change of variables $x \to \xi_i(x,t)$ via
\begin{align*}
\eta_i(x,t) := -\pi + \frac1{\sigma_i(t)} \int^{x}_{-\pi} |\dot \ga_i(y,t)| \, \rd y \qquad \sigma_{i}(t) := \fint_{\T} |\dot \ga_i(y,t)| \, \rd y \qquad \eta_i( \xi_i(x,t),t) = x,
\end{align*}
that converts $\ga_i$ to its constant speed representation. We may then use this change of variables to define a ``scale'' $\veps_i(t)$ and a unit speed representation $\phi_i(x,t)$ of the ``shape'' of each interface. Specifically, if we let $\psi_i := \psi_{\ga_i} := \ga_i \circ \xi_i$ denote the constant speed representation of the interface $\ga_i$ then the re-scaled interface
\[\phi_i(x,t) := \frac{\psi_{i}( x , t )}{\sigma_i(t)} := \veps^{\frac12}_{i}(t) \psi_i(x,t)\]
has unit speed. Applying this change of variables gives
$$
\veps^{\frac12}_i(t)F_i(x,t) := f_i\big( \xi_i(x,t) ,t \big) + c_i = \sum^{m}_{j=1} \int_{\T} g_{ij}\left( \frac12|\psi_i(x,t) - \psi_j(y,t)|^2 \right)\sigma_j(t)  \, \rd y + c_i,\qquad \psi_i := \sigma_i \phi_i,
$$
for the forcing in constant speed coordinates, while differentiating in space gives the usual relations
$$
\dot \phi_i = \tang_{\ga_i} \circ \xi_i \qquad \dot \phi^{\perp}_{i} = \nrml_{\ga_i}\circ \xi_i \qquad \text{and} \qquad \veps^{\frac12}_i \ddot \phi_i = \left(\kappa_{\ga_i} \circ \xi_i\right)\, \nrml_{\ga_i}\circ \xi_i
$$
for the unit tangent, unit normal and curvature. As a consequence, we obtain the representation
$$
\vv_i(x,t) := \partial_{t} \ga_i( \xi_i(x,t),t) = \veps^{\frac12}_i(t)\left( \ddot \phi_i(x,t) + F_i(x,t) \dot\phi^{\perp}_i(x,t) \right) = \langle \vv_i(x,t),\dot \phi^{\perp}_i(x,t)\rangle \dot \phi^{\perp}_i(x,t)
$$
for the interfacial velocity in constant speed coordinates. Now as any choice of $\perp:\R^2 \mapsto \R^2$ is necessarily an anti-symmetric linear operator, we may deduce that the relation
\begin{align*}
\langle \dot \vv_i, \dot \phi_i \rangle &= -\langle \vv_i,\dot \phi^{\perp}_i\rangle \langle \ddot \phi_i,\dot \phi^{\perp}_i\rangle = -\veps^{\frac12}_i\left( |\ddot \phi_i|^2 + \langle \ddot \phi_i, \dot \phi^{\perp}_i \rangle F_i \right) := -\veps^{\frac12}_i a_i,
\end{align*}
holds. A straightforward application of \eqref{eq:cspdder} and the product rule then gives the evolution
\begin{align*}
\partial_{t} \phi_i = \veps_i\left( \ddot \phi_i + \po\big( a_i\big) \dot \phi_i +  F_i \dot \phi^{\perp}_i + \frac{\rd_t \veps_i }{2 \veps^2_i} \phi_i \right)
\end{align*}
for the shape component of each interface. An integration by parts and a change of variables give the evolution
\begin{align*}
\frac{\rd_t \veps_i}{2 \veps^2_i} = \veps^{-\frac12}_i\fint_{\T} \kappa_{\ga_i} \langle \nrml_{\ga_i}, \partial_t \ga_i \rangle |\dot \ga_i| = \veps^{-\frac12}_i \fint_{\T} \langle \ddot \phi_i, \vv_i \rangle = \mu_{a_i},
\end{align*}
for the scale component, thus closing the system. All-together, we obtain the coupled system
\begin{align}\label{eq:shapedynamic}
\partial_{t} \phi_i &\;= \veps_i \left( \ddot \phi_i + \po\left(a_i\right)\dot \phi_i +  F_i \dot \phi^{\perp}_i + \mu_{a_i}\phi_i\right),\qquad \rd_t \veps_i = 2\veps^2_i \mu_{a_i} \qquad 
a_{i} := |\ddot \phi_i|^2 + \langle \ddot \phi_i,\dot \phi^{\perp}_i\rangle F_i, \nonumber \\
F_i(x,t) &:= \veps^{-\frac12}_i(t)\left( c_i + \sum^{m}_{j=1} \int_{\T} g_{ij}\left(\frac12|\psi_i(x,t)-\psi_j(y,t)|^2\right)|\dot \psi_j(y,t)| \, \rd y \right), \qquad \psi_i := \veps^{-\frac12}_i \phi_i
\end{align}
for representing the dynamic in unit speed coordinates, together with periodic boundary conditions for $\phi_i$ and its derivatives.  Conversely, we may easily recover the original dynamic \eqref{eq:basic} from the shape dynamic \eqref{eq:shapedynamic} in the usual way via the method of characteristics. Given a collection $\Phi = \{\phi_1,\ldots,\phi_m\}$ solving \eqref{eq:shapedynamic} on $\T \times [0,T_*),$ we simply define $\eta_i(x,t)$ through the family
\begin{align*}
\partial_{t} \eta_i(x,t) &= -\veps_i(t) \po\big( a_{i} \big)(\eta_i(x,t),t) \\
\eta_{i}(x,0) &= x
\end{align*}
of ordinary differential equations. We have $\eta_i(-\pi,t) = -\pi, \eta_i(\pi,t) = \pi$ since $\po\big(a_i)$ vanishes at $\pm \pi$, and moreover
$$
\dot \eta_i(x,t) = \mathrm{exp}\left( \int^{t}_{0} \left[ \mu_{a_i}(s) - a_i( \eta_i(x,s) , s ) \right]\veps_i(s) \, \rd s \right)
$$
never vanishes for as long as it exists. The family of functions $x \mapsto \eta_i(x,t)$ therefore defines a collection of invertible maps of $I_{\pi}$ onto $I_{\pi}$ as long as the $\phi_i$ exist. The definitions $\ga_i(x,t) = \sigma_i(t)\phi_i( \eta_i(x,t) , t )$ and the chain rule then suffice to show that the collection $\Gamma = \{\ga_1,\ldots\ga_m\}$ defines a solution to the standard evolution.\\

\noindent \textbf{The Tangent-Angle Representation:}
The representation \eqref{eq:shapedynamic} describes the evolution of each interface $\ga_i \in \Gamma$ in terms of separate evolutions for its shape (i.e. $\phi_i$) and its scale. In contrast to \eqref{eq:basic}, the shape-scaling representation has the benefit of placing a constant $\veps_i(t)$ in front of the diffusion term at the expense of an additional, nonlinear transport term. Together with the dynamic for the diffusion constant, this formulation makes \eqref{eq:shapedynamic} easier to deal with for many analytical purposes. Nevertheless, the shape-scaling representation still manifests a leading-order nonlinearity (i.e. involving $\ddot \phi_i$) in the transport coefficient and in the evolution of the diffusion coefficient. We may fortuitously affect a further reduction that pushes these nonlinearities to lower order as well, yielding a description of each interface in terms of its unit tangent. The resulting tangent-angle representation is somewhat more cumbersome for analysis, but is well-suited for developing stable numerical procedures.

Pick some $i \in [m]$ and let $\phi(x,t) := \phi_i(x,t)$ denote the shape of the resulting interface. For as long as a solution exists, say on some temporal interval $[0,T_*),$ the unit tangent
$$
\tang(x,t) :=  \phi_{x}(x,t)
$$
to this interface is well-defined. For such $\tang$ define a function $\vth: \T \times [0,T_{*}) \mapsto \R$ by the relation
$$
\vth(x,t) := \mu_{\vth_0} - \int^{t}_{0} \fint_{\T} \langle \tang_t(z,s) , \tang^{\perp}(z,s) \rangle \, \rd z \rd s - \pc  \langle \dot \tang , \tang^{\perp}\rangle
$$
where the initial value $\vth_0(x) = \vth(x,0)$ guarantees that the representation
$$
\tang(x,0) = \begin{bmatrix}
-\sin \vth(x,0) \\
\;\;\;\,\cos \vth(x,0) \end{bmatrix}
$$
holds initially. Given $\vth$ defined in this way, let
$$
\hat \tang(x,t) := \begin{bmatrix}
-\sin \vth(x,t) \\
\;\;\;\, \cos \vth(x,t)
\end{bmatrix}
$$
denote a putative representation of the unit tangent. Differentiating in time gives
\begin{align*}
\hat \tang_t(x,t) &= \left( \mu_{ \langle \tang_t , \tang^{\perp} \rangle } + \pc \langle \dot \tang_t , \tang^{\perp}\rangle + \pc \langle \dot \tang , \tang^{\perp}_{t}\rangle \right)(x,t) \hat \tang^{\perp}(x,t)\\
&= \left( \langle \tang_t , \tang ^{\perp} \rangle + \pc \langle \dot \tang , \tang^{\perp}_{t}\rangle - \pc \langle \tang_t , \dot \tang^{\perp}\rangle \right)(x,t) \hat \tang^{\perp}(x,t),
\end{align*}
with the latter inequality holding due to integration by parts. The identities $\tang_t = \langle \tang_t, \tang^{\perp}\rangle \tang^{\perp}$ and $\dot \tang = \langle \dot \tang, \tang^{\perp}\rangle \tang^{\perp}$ hold since $\tang$ has unit length, and so the cancellations
\begin{align*}
\langle \tang^{\perp}_{t},\dot \tang \rangle &= - \left< \langle \tang_t,\tang^{\perp}\rangle\tang, \dot \tang \right> = 0\\
\langle \tang_{t},\dot \tang^{\perp} \rangle &= - \left< \tang_t , \langle \dot \tang,\tang^{\perp}\rangle \tang \right> = 0
\end{align*}
follow as well. Thus both ordinary differential equations 
$$
\hat \tang_{t}(x,t) = \langle \tang_t(x,t) , \tang^{\perp}(x,t) \rangle \hat \tang^{\perp}(x,t)\qquad \text{and} \qquad \tang_{t}(x,t) = \langle \tang_t(x,t) , \tang^{\perp}(x,t) \rangle \tang^{\perp}(x,t),
$$
hold for as long as $\tang$ exists. In other words, both $t \mapsto \hat \tang(x,t)$ and $t \mapsto \tang(x,t)$ solve the linear initial value problem
$$
\sigma^{\prime}(t) = \langle \dot \tang(x,t) , \tang^{\perp}(x,t) \rangle \sigma^{\perp}(t)
$$
and so $\tang(x,t) = \hat \tang(x,t)$ for $t \in [0,T_*)$ provided they agree on $\T$ at time zero. At time zero, a similar argument starting from the relation
$$
\vth(x,0) = \vth(-\pi,0) - \int^{x}_{-\pi} \langle \dot \tang(x,0) , \tang^{\perp}(x,0) \rangle \, \rd x
$$
shows the linear initial value problems
\begin{align*}
\hat \tang_{x}(x,0) &= \langle \dot \tang(x,0) , \tang^{\perp}(x,0) \rangle \hat \tang^{\perp}(x,0) \\
\tang_{x}(x,0) &=  \langle \dot \tang(x,0) , \tang^{\perp}(x,0) \rangle \tang^{\perp}(x,0)
\end{align*}
hold, and so $\hat \tang(x,0)$ and $\tang(x,0)$ agree on $\T$ if they agree at a single point. This latter condition follows by choosing $-\pi \leq \vth(-\pi,0) < \pi$ so that
$$
\tang(-\pi,0) = \hat \tang(-\pi,0) = \begin{bmatrix}
-\sin \vth(-\pi,0) \\
\;\;\;\, \cos \vth(-\pi,0)
\end{bmatrix}
$$
by taking an inverse tangent. We therefore have
$$
\tang(x,t) = \hat \tang(x,t) = \begin{bmatrix}
-\sin \vth(x,t) \\
\;\;\; \cos \vth(x,t)
\end{bmatrix},
$$
on $\T \times [0,T_*)$, and moreover that $\vth(\pi,t) - \vth(-\pi,t) = 2\pi k$ for some fixed $k \in \Z$ since $\tang(x,t)$ is periodic and continuous in time.

This construction gives a well-defined function $\vth(x,t)$ on $[0,T^{*})$ so that the tangent-angle representation
$$
\tang(x,t) = \begin{bmatrix}
-\sin \vth(x,t) \\
\;\;\,\,\cos \vth(x,t)
\end{bmatrix} = \dot \phi(x,t)
$$
holds. Moreover, we assume by convention that
$$
\tang^{\perp}(x,t) = \nrml(x,t) = \begin{bmatrix}
\cos \vth(x,t) \\
\sin \vth(x,t)
\end{bmatrix} 
$$
furnishes the unit normal vector. A  reduction to a dynamic for $\vth$ then follows easily. A straightforward differentiation of the representation shows 
\begin{align*}
\ddot \phi(x,t) & = -\dot \vth(x,t) \nrml(x,t)\\
\dddot \phi(x,t) & = -\ddot \vth(x,t) \nrml(x,t) - \dot \vth^2(x,t) \tang(x,t) \\
\dot \phi_t(x,t) &= -\vth_{t}(x,t) \nrml(x,t)
\end{align*}
on the one hand, while differentiating the evolution \eqref{eq:shapedynamic} for $\phi$ reveals
\begin{align*}
\dot \phi_t &= \veps\big( \dddot \phi + \po\big(a\big) \ddot \phi + a\dot \phi + \dot F \dot \phi^{\perp} + F \ddot \phi^{\perp} \big) \qquad \qquad a = |\ddot \phi|^2 + F \langle \ddot \phi,\dot \phi^{\perp}\rangle = \dot \vth^2 - F\dot \vth \\
&= \veps \big( \dot F  -  \ddot \vth   - \po\big(a\big) \dot \vth \big) \nrml + \veps \big( F\dot \vth + a 
-  \dot \vth^2  \big)\tang \\
&=\veps \big( \dot F  -  \ddot \vth   - \po\big(a\big) \dot \vth \big) \nrml 
\end{align*}
on the other. We therefore obtain the evolution
$$
\vth_{t} = \veps\left( \vth_{xx} + \po\big(a\big) \vth_x - \dot F \right)
$$
for the angular variable, together with the boundary conditions
\begin{align*}
\vth(\pi,t) &= \vth(-\pi,t) + 2\pi k \qquad k \in \Z \\
\vth_{x}(\pi,t) &= \vth_{x}(-\pi,t)
\end{align*}
that follow from periodicity of $\ddot \phi$ and the tangent. We may proceed to close the evolution by making a few observations.  First, if we have knowledge of the tangent
$$
\tang(x,t) = \begin{bmatrix}
-\sin \vth(x,t) \\ \;\;\,\,\cos \vth(x,t)
\end{bmatrix}
$$
and the center of mass $\cm(t) := \mu_{ \psi }$ then we may therefore recover the interface itself
$$
\psi(x,t) = \cm(t) + \sigma(t) \pc \tang(x,t)
$$
by taking a mean-zero primitive. But integrating \eqref{eq:shapedynamic} gives an evolution
$$
\frc_{t} = \veps^{\frac12} \mu_{\vv} \qquad \text{for} \qquad \vv := \po\big(a\big)\tang + F\nrml
$$
for the center of mass, so we retain the ability to compute the nonlocal forcings
\begin{align*}
F_i(x,t) &= \veps^{-\frac12}_i(t)\left( c_i + \sum^{m}_{j=1} \int_{\T} g_{ij}\left(\frac12|\psi_i(x,t)-\psi_j(y,t)|^2\right)\sigma_j(t) \, \rd y \right), \qquad \psi_i := \cm_i + \sigma_i \pc \tang_i\\
\dot F_i(x,t) &= \sum^{m}_{j=1}  \int_{\T} \dot g_{ij}\left( \frac12|\psi_i(x,t) - \psi_j(y,t)|^2 \right)\langle \psi_i(x,t) - \psi_j(y,t) , \tang_i(x,t) \rangle\sigma_j(t) \, \rd y
\end{align*}
from knowledge of the collection $(\Theta,\mathfrak{C}) := \{\theta_1,\ldots,\theta_m\} \times \{\frc_1,\ldots,\frc_m\}$ of angular variables and centroids. To summarize, the shape dynamic \eqref{eq:shapedynamic} induces the dynamic
\begin{align}\label{eq:evo2}
\partial_{t} \vth_i &= \veps_i\left( \ddot \vth_i + \po\big(a_i\big) \dot \vth_i - \dot F_i\right), \qquad  \vth_i(\pi,t) - \vth_i(-\pi,t) = 2\pi k_i,\;\,k_i \in \Z, \qquad \dot \vth_i(\pi,t) = \dot \vth_i(-\pi,t) \nonumber \\
\rd_{t} \veps_i &= 2\veps^2_i \mu_{a_i} \qquad \,\rd_{t} \frc_i = \veps^{\frac12}_i \mu_{\vv_i},\qquad\,  \vv_i := \po\big(a_i\big)\tang_i + F_i\nrml_i, \qquad \qquad \qquad \;\;\;\;\;\, a_i = \dot \vth^2_i - F_i \dot \vth_i  \\
F_i(x,t) &= \sigma_i(t)\left( c_i + \sum^{m}_{j=1} \int_{\T} g_{ij}\left(\frac12|\psi_i(x,t)-\psi_j(y,t)|^2\right)\sigma_j(t) \, \rd y \right)\qquad \qquad \;\;\;\;\; \psi_i = \cm_i + \sigma_i \pc \tang_i \nonumber
\end{align}
for the collection $\Theta = \{\vth_1,\ldots,\vth_m\}$ of angular variables. To show the converse, assume that we have a collection $\Theta = \{\vth_1,\ldots,\vth_m\}$ that solves \eqref{eq:evo2} on $\T \times [0,T_*)$ and let
\begin{align*}
\tang_i(x,t) := \begin{bmatrix}
-\sin \vth_i(x,t) \\
\;\;\,\,\cos \vth_i(x,t)\end{bmatrix} \qquad \qquad \phi_i(x,t): = \sqrt{\veps_i(t)} \cm_i(t) +  \pc \tang_i(x,t)
\end{align*}
denote the putative solution $\Phi =\{\phi_1,\ldots,\phi_m\}$ to the shape dynamic. Then
$$
\phi_i(\pi,t) - \phi_i(-\pi,t) = \int_{\T} \tang_i(x,t) \, \rd x =:  \ww_i(t)
$$
by definition, and so $\phi_i$ defines a closed curve provided $\ww_i$ vanishes. Fixing $\vth = \vth_i$ for some $i \in [m],$ we may note
\begin{align*}
\left[ \int_{\T} \sin \vth(x,t) \, \rd x \right]_{t} &= \int_{\T} \cos \vth(x,t) \vth_t(x,t) \, \rd x = \veps(t) \int_{\T}\big( \vth_{xx} + \po\big( a \big)\vth_x - \dot F\big)  \cos \vth \,\rd x \\
&= \veps \int_{\T} \left[ \big(  \vth^2_x -  (a-\mu_a) \big)\sin \vth - \dot F \cos \vth \right]\, \rd x 
\end{align*}
by periodicity of $\vth_x,\sin \vth$ and the fact that $\po\big(a\big)$ defines a zero-Dirichlet primitive. Now
\begin{align*}
\int_{\T} \dot F \cos \vth \, \rd x &= \big( F(\pi,t)\cos \vth(\pi,t) - F(-\pi,t) \cos \vth(-\pi,t) \big) + \int_{\T} F\dot \vth \sin \vth \, \rd x \\
&= \cos \vth(\pi,t) \big( F(\pi,t) - F(-\pi,t) ) + \int_{\T} F \dot \vth \sin \vth \, \rd x
\end{align*}
by the periodicity of $\cos \vth$ and integration by parts. Set $q(t) = F(\pi,t) - F(-\pi,t)$ as the boundary jump for the nonlocal forcing, so that
\begin{align*}
\left[ \int_{\T} \sin \vth(x,t) \, \rd x \right]_{t} &= \veps(t) \int_{\T} \big( \vth^2_x - (a - \mu_a) - F \dot \vth \big) \sin \vth \, \rd x -  \veps(t) \cos \vth(\pi,t)q(t)\\
&= \frac{ \rd_t \veps(t) }{ 2 \veps(t) } \int_{\T} \sin \vth(x,t) \, \rd x-  \veps(t) \cos \vth(\pi,t)q(t) \\
\left[ \int_{\T} \cos \vth(x,t) \, \rd x \right]_{t} &=  \frac{ \rd_t \veps(t) }{ 2 \veps(t) }\int_{\T} \cos \vth(x,t) \, \rd x + \veps(t)\sin \vth(\pi,t) q(t),
\end{align*}
with the last equality following from a nearly identical calculation. The dynamic
$$
\frac{\rd \sigma_i \ww_i}{\rd t}(t) = \nrml_i(\pi,t)\sqrt{\veps_i(t)}q_i(t)
$$
holds for the mean of each tangent vector. Recalling the definitions of $q_i,F_i$ and $\ww_i$ then shows
\begin{align*}
\sqrt{\veps_i(t)}q_i(t) &= \sum^{m}_{j=1} \int_{\T} \left[ g_{ij}\left( \frac12|\psi_i(\pi,t) - \psi_j(x,t)|^{2} \right) - g_{ij}\left( \frac12|\psi_i(-\pi,t) - \psi_j(x,t)|^{2} \right) \right] \sigma_j(t) \, \rd y,\qquad\psi_i := \veps^{-\frac12}_i \phi_i \\
&= \sum^{m}_{j=1} \int_{\T} \left[ g_{ij}\left( \frac12|\psi_i(-\pi,t) - \psi_j(x,t) + \sigma_i(t)\ww_i(t)|^{2} \right) - g_{ij}\left( \frac12|\psi_i(-\pi,t) - \psi_j(x,t)|^{2} \right) \right] \sigma_j(t) \, \rd y \\
&:= Q_i\big( \sigma_i(t)\ww_i(t) \big)
\end{align*}
for $Q_i(\vv)$ some Lipschitz function that vanishes at the origin. Thus $\sigma_i\ww_i$ vanishes for all time if it vanishes initially, and so the tangents
$$
\tang_i(x,t) = \begin{bmatrix}
-\sin \vth_i(x,t) \\
\;\;\,\,\cos \vth_i(x,t)
\end{bmatrix}
$$
define closed curves for all time. In particular, $\phi_i$ and its derivatives are periodic. Using a straightforward differentiation of
$$
\phi_i(x,t) := \sqrt{\veps_i(t)} \cm_i(t) +  \pc \tang_i(x,t)
$$
and the integration by parts formula for $\pc$ then shows that \eqref{eq:shapedynamic} holds pointwise. All together, we may conclude
\begin{proposition}\label{prop:equivdyn}
The representations (\ref{eq:basic},\ref{eq:shapedynamic}) and \eqref{eq:evo2} are equivalent.
\end{proposition}
\noindent We may therefore use these representations interchangeably, depending upon convenience in a particular context.

\section{Nonlocal Forcings}\label{sec:Nonlocal} 
In principle, we may now use any of the representations (\ref{eq:basic},\ref{eq:shapedynamic},\ref{eq:evo2}) to show that a wide class of nonlocal forcings 
$$
f_{(\ga_i,\ga_j)}(x,t) := \int_{\T} g_{ij}\left(\frac12|\ga_i(x,t) - \ga_j(y,t)|^2 \right) |\dot \ga_j(y,t)| \, \rd y
$$
lead to a well-behaved dynamic. The main obstacle we face in this endeavor is that the general dynamic \eqref{eq:basic} models systems that exhibit vastly different physical behaviors, yet we want to allow enough freedom in choosing the $g_{ij}$ to cover as many potential applications as possible. For example, if we consider an instance of \eqref{eq:basic} that describes a coarsening dynamic then we should allow for the possibility that distinct interfaces $\ga_i,\ga_j$ intersect when defining the nonlocal force. Such intersections place limits on the types of kernels $g_{ij}$ that lead to well-behaved nonlocal forces. Conversely, if we rule out such intersections \emph{a-priori} then we may allow for a much wider range of kernels. In this way, the geometric properties of interfaces in the flow determine the analytical properties of the kernels we may use in a dynamical model.

As the shape-scaling representation \eqref{eq:shapedynamic} demands at least $H^{2}(\T)$ regularity of each interface, it suggests a way to phrase these considerations in analytical terms. The main insight is that allowable selections for the nonlocal kernels $g_{ij}(s)$ should induce nonlocal forces obeying some type of $H^{1}(\T) \mapsto L^{2}(\T)$ Lipschitz bound of the form
\begin{align}\label{eq:gummybounds}
\|f_{(\ga_i,\ga_j)} - f_{(\tilde \ga_i,\ga_j)}\|_{L^{2}(\T)} \lesssim \|\ga_i - \tilde \ga_i\|_{H^{1}(\T)} \qquad \text{and} \qquad \|f_{(\ga_i,\ga_j)} - f_{(\ga_i,\tilde \ga_j)}\|_{L^{2}(\T)} \lesssim \|\ga_j - \tilde \ga_j\|_{H^{1}(\T)},
\end{align}
for then \eqref{eq:shapedynamic} leads, by the results of the next section, to a well-behaved dynamic in the $H^{2}(\T)$-sense. We therefore focus on showing that \eqref{eq:gummybounds} indeed holds under reasonably broad yet easily verifiable hypotheses.

By a kernel $g_{ij}$ we will always mean a function $g_{ij} \in C^{2}\big( \mathbb{R}^{+} \setminus \{0\} \big)$ that has at least two derivatives away from the origin, and we will use the notation
$$
\mathcal{G}^{*}_{ij}(s) := \sup_{t \geq s} \;\, |g_{ij}(t)|
$$
to denote its monotone non-increasing envelope of such a kernel. Consider the standard embedding $\ga(x)$ of the circle of radius $R$ and a compactly supported kernel $g = g_{ii}$ that agrees with its envelope. Finiteness of the corresponding force
\begin{align*}
f(x) &:=  R \int_{\T} \cG^{*}\left( \frac12|\ga(x) - \ga(y)|^{2} \right) \, \rd y = 2R \int^{\pi}_{0} \cG^{*}\left( R^2\big( 1 - \cos z \big) \right) \, \rd z \\
&\,\approx 2R \int^{\infty}_{0} \cG^{*}\left( \frac{R^2 z^2}{2} \right) \, \rd z = \sqrt{2} \int^{\infty}_{0} \cG^{*}\left(u^2\right) \, \rd u := \sqrt{2}\,\frc^{*}_{0}(g)
\end{align*}
necessitates integrability of the envelope. We therefore introduce the integrability hypothesis
\begin{enumerate}[\indent\indent{\rm H0)}]
\item The kernel $g$ has integrable envelope: $\frc^{*}_{0}(g) < +\infty.$
\end{enumerate}
to encode local information about singularity of the kernel near the origin; the far-field behavior does not matter to the dynamics on finite time intervals, for we can always truncate a kernel at some finite length-scale without affecting the forcing. Additionally, we must impose some type of regularity on $\ga$ if we want to obtain a uniform estimate for such nonlocal forcings. For example, for each $m \in \mathbb{N}$ and $\sigma > 0$ consider the $m$-covered circle 
$$
\ga_{m,\sigma}(x) = \frac{ \sigma }{m} \begin{bmatrix}
\cos mx \\
\sin mx
\end{bmatrix}
$$
of length $2 \pi \sigma$ and let $g(s) := \mathrm{exp}(-s)$ denote a simple Gaussian kernel. A simple computation then gives the nonlocal force
$$
f_{m,\sigma}(x) := \int_{\T} g\left( \frac{|\ga_{m,\sigma}(x) - \ga_{m,\sigma}(y)|^2}{2} \right) \sigma \, \rd y =  \frac{ m \sqrt{\lambda} I_0(\lambda) }{ \re^{\lambda}} \qquad \qquad \lambda := \frac{\sigma^2}{m^2}
$$
in terms of a modified Bessel function $I_0(x)$ of the first kind. In particular, if we take $\lambda$ constant as $m \to \infty$ then the nonlocal force $f_{m,\sigma}$ becomes unbounded. We must therefore impose some regularity measure, such as finite length or finite total curvature, if we want a uniform bound on the nonlocal force.

A broad but relatively simple class of immersions, which we call \emph{$(N,\lambda)$-regular immersions}, will suffice for these purposes. Given an integer $N \in \mathbb{N}$ and a real $\lambda \geq 1$ we call $\gamma \in C^{0,1}(\T)$ an $(N,\lambda)$-regular immersion if $I_{\pi}$ decomposes into $N$ almost-disjoint subintervals
$$
I_{\pi} = \bigcup^{N}_{k=1} [x_{k-1},x_{k}] \qquad \text{where} \qquad  - \pi = x_0 < x_1 < \cdots < x_{N-1} < x_N = \pi,
$$
and each of the $N$ corresponding sub-arcs 
$$\alpha_{k} := \left.\ga\right|_{[x_{k-1},x_{k}]}$$
of $\ga$ are $\leq \lambda$-distorted. For example, any $H^{2}(\T)$-immersion is necessarily an $(N(\ga),3)$-regular immersion while any embedded curve is a $(1,\delta_{\infty}(\ga))$-regular immersion. Under this hypothesis we may relatively easily establish the following bound ---
\begin{lemma}\label{lem:velbound}
Assume the kernel $g$ satisfies {\rm H0)} and that $\ga_i \in C^{0,1}(\T),\,i \in {1,2}$ have constant speed. Then for any $\veps>0$ the nonlocal force
$$
f^{\veps}(x) := \int_{\T} g\left( \frac12|\ga_1(x) - \ga_2(y)|^2 + \frac{\veps^2}{2}\right)|\dot \ga_2(y)| \, \rd y
$$
obeys the bound
$$
\|f^{\veps}\|_{L^{\infty}(\T)} \leq 2\pi N\lambda \,\frc^{*}_{0}(g)$$
for any $(N,\lambda)$-regular immersion.
\end{lemma}
\begin{proof}
Let $I_k = [x_{k-1},x_{k}]$ and decompose the total force
$$
f_{\veps}(x) = \sum^{N}_{k=1} \int_{I_k} g\left( \frac12|\ga_1(x) - \ga_2(y)|^2 + \frac{\veps^2}{2} \right) \sigma_2 \, \rd y
$$
into $N$ pieces, where $\sigma_2 := |\dot \ga_2(y)|$ denotes the constant speed of the second curve. For each $k$ let $y_k = y_{k}(x)$ denote any solution to the minimization
$$
\min_{y \in I_k} \;\, |\ga_1(x) - \ga_2(y)|,
$$
and note that the bound
$$
|\ga_2(y) - \ga_2(y_k)| \leq |\ga_2(y) - \ga_1(x)| + |\ga_1(x) - \ga_2(y_k)| \leq 2 | \ga_2(y) - \ga_1(x)|
$$
holds for $y \in I_k$ by the triangle inequality. Thus
$$
|\ga_2(y) - \ga_1(x)| \geq \frac{|\ga_2(y) - \ga_2(y_k)|}{2} \geq \frac{ \ell_{\ga_2}(y,y_k)}{2 \lambda} = \frac{ \sigma_2 |y - y_k| }{2 \lambda}
$$
since $\ga_2$ is $(N,\lambda)$-regular and has constant speed. The pointwise bound
$$
\left| g\left( \frac{|\ga_1(x) - \ga_2(y)|^2}{2} + \frac{\veps^2}{2} \right) \right| \leq g^{*}\left( \frac{|\ga_1(x) - \ga_2(y)|^2}{2} + \frac{\veps^2}{2} \right) \leq g^{*}\left( \frac{\sigma^2_2(y-y_k)^2}{8 \lambda^2}\right)
$$
therefore holds on $I_k$ by definition of the monotone non-increasing envelope. The bound
$$
\left| \int_{I_k} g\left( \frac{|\ga_1(x) - \ga_2(y)|^2}{2} + \frac{\veps^2}{2} \right) \sigma_2 \, \rd y \right| \leq \int_{\R} g^{*}\left( \frac{\sigma^2_2(y-y_k)^2}{8 \lambda^2}\right) \sigma_2 \, \rd y = 2\sqrt{8} \lambda \frc^{*}_{0}(g) \leq 2 \pi \lambda \frc^{*}_{0}(g)
$$
therefore follows from a change of variables, which gives the claim upon summing.
\end{proof}
\noindent We may also infer that the limit
$$
f_{(\ga_1,\ga_2)}\left(x\right) := \lim_{\veps \downarrow 0} \; f^{\veps}(x) = \int_{\T} g\left( \frac12|\ga_1(x) - \ga_2(y)|^2\right)|\dot \ga_2(y)| \, \rd y
$$
converges in $L^{p}(\T),\, p < \infty$ and obeys the $L^{\infty}(\T)$ bound asserted in the lemma. We shall apply lemma \ref{lem:velbound} with $(N,\lambda)=(N(\ga_2),3)$ for $H^{2}(\T)$ immersions in the context of our well-posedness arguments. Examples show that both of the na\"{i}ve choices
$$
N\lambda = \delta_{\infty}(\gamma) \qquad \text{or} \qquad N\lambda = 3N(\ga_2)
$$
can prove far too pessimistic, so the ability to specify some alternative $(N,\lambda)$ decomposition of $\ga$ is worthwhile.

We also need some type of continuity of $f_{(\ga_1,\ga_2)}$ with respect to the $\ga_i$ in the context of our existence arguments. To guarantee such continuity properties of the nonlocal force we must impose more hypotheses. For example, if $|\dot g(s)| \sim C s^{-p}$ for some $p>1$ then the nonlocal velocity cannot define an $H^{1}(\T) \mapsto L^{1}(\T)$ Lipschitz map. In other words, a bound of the form
$$
\|f_{(\ga_1,\ga_2)} - f_{(\ga_1,\tilde \ga_2)}\|_{L^{1}(\T)} \lesssim \|\ga_2 - \tilde \ga_2\|_{H^{1}(\T)}
$$
cannot hold in general. (It suffices to consider an infinite line and an infinitesimally rotated version of the same line to see this.) We must therefore impose either regularity with respect to the kernel or with respect to the geometry of the immersions. We consider each case in turn. We shall employ the hypothesis
\begin{enumerate}[\indent\indent{\rm H1)}]
\item The kernel $g$ is regular: $\sup_{s>0} s|\dot g(s)| \leq \frc^{*}_{1}(g)$ and $\sup_{s>0} s^2|\ddot g(s)| \leq \frc^{*}_{1}(g)$ for some $\frc^{*}_{1}(g) < +\infty$.
\end{enumerate}
to force that the kernel behaves no worse than $\dot g(s) \sim C s^{-1}$ in some fashion. Once again, this hypothesis only encodes local information about singularity of the kernel near the origin; we may assign far-field behavior to $g(s)$ at will. Under this assumption lemma \ref{lem:derbound} will allow us to show that the nonlocal force obeys further continuity properties, in addition to simply boundedness. While bounds of this type follow from the results of \cite{DvD}, we will work explicitly to not have to take a detour through the requisite machinery. We shall try to keep various constants reasonably explicit, but not at all sharp, while pursuing our analysis.

\begin{lemma}\label{lem:derbound}
Assume the kernel $g$ satisfies {\rm H1)}, that $\vv \in H^{1}(\T)$, that $\ga_1 \in C(\T)$ and that $\ga_2 \in H^{2}(\T)$ has constant speed. Then for any $\veps>0$ the function
$$
f^{\veps}_{\vv}(x) := \int_{\T} \dot g\left( \frac12|\ga_1(x) - \ga_2(y)|^2 + \frac{\veps^2}{2} \right)\langle \ga_1(x) - \ga_2(y),\vv(y)\rangle \, \rd y
$$
obeys the estimate
\begin{align}\label{eq:regularest}
\|f^{\veps}_{\vv}\|_{L^{\infty}(\T)} \leq (8\pi)^{3}\frc^{*}_1(g)\kappa^2_2(\ga_2)\max\left\{ \|\vv\|_{L^{\infty}(\T)}, \|\dot \vv\|_{L^{2}(\T)} \right\}.
\end{align}
\end{lemma}
\begin{proof}
Put $\frc^{*}(g) := \frc^{*}_1(g)$ for ease of notation. Fix an integer $N = 4N(\ga_2)$ with
$$
N(\ga_2) := \left \lceil \frac{\ell(\ga_2)}{\ell_{\kappa}(\ga_2)} \right \rceil =  \left \lceil \frac{2\pi \sigma_2}{\ell_{\kappa}(\ga_2)} \right \rceil 
$$
and $\sigma_2 := |\dot \ga_2(x)|$ the constant speed of the inducing curve. Note that $N \geq 160$ by the lower bound $\ell(\ga)\kappa^2_2(\ga) \geq (2\pi)^2$ for the bending energy.

Divide $I_{\pi} = [-\pi,\pi]$ into the $N$ equal length sub-intervals $[x_{k-1},x_{k}], \, k \in [N]$ with $x_k := -\pi + 2\pi k/N$. Let $\chi \geq 0$ denote a smooth function on $\R$ obeying $\chi \equiv 1$ on $[-1,1],$ $|\dot \chi(x)|\leq 2$ and $\mathrm{supp}(\chi) \subset [-2,2]$. For $0 \leq k \leq N$ put
$$
\chi_{k}(x) :=  \chi \left( \frac{N}{\pi}(x-x_k) \right) \qquad \text{with} \qquad \mathrm{supp}\big( \chi_{k} \big) \subset [x_{k-1},x_{k+1}].
$$
In particular, if $-\pi \leq x \leq \pi$ then at most two of the $\chi_k$ are non-trivial and $\chi_j(x) = 1$ for at least one $j$ at any point. The collection
$$
\zeta_{k}(x) := \frac{\chi_{k}(x)}{\sum^{N}_{j=0} \chi_j(x) }
$$
therefore forms a smooth partition of unity on $[-\pi,\pi]$ obeying the properties
$$
\mathrm{(i)}\;\, \mathrm{supp}\big(\zeta_k\big) \subset [x_{k-1},x_{k+1}] \qquad \text{and} \qquad \mathrm{(ii)} \;\, |\dot \zeta_k(x)| \leq \|\dot \chi\|_{L^{\infty}(\R)} \frac{N}{\pi} \leq N
$$
at all points in the interval. Thus the decomposition
\begin{align*}
f^{\veps}_{\vv}(x) &=\sum^{N}_{k=0} \int_{\T} \dot g\left( \frac12|\ga_1(x) - \ga_2(y)|^{2} + \frac{\veps^2}{2} \right)\langle \ga_1(x) - \ga_2(y) , \vv(y) \rangle \zeta_{k}(y) \, \rd y
\end{align*}
of the velocity holds, and it suffices to show that each
$$
f^{\veps}_{k}(x) := \int_{\T} \dot g \left( \frac12|\ga_1(x) - \ga_2(y)|^2 + \frac{\veps^2}{2} \right)\langle \mathbf{v}(y), \ga_1(x) - \ga_2(y) \rangle \zeta_k(y) \, \rd y
$$
obeys a uniform bound of the form
\begin{align}\label{eq:goal1}
|f^{\veps}_{k}(x)| \leq  34\pi^{2} \frac{\frc^{*}(g)[\vv]_{H^{1}(\T)}}{\sigma_2}, \qquad \qquad [\vv]_{H^{1}(\T)} := \max\left\{\|\vv\|_{L^{\infty}(\T)},\|\dot \vv\|_{L^{2}(\T)} \right\}.
\end{align}
Indeed, summing over $k$ gives the claimed estimate
$$
|f^{\veps}_{\vv}(x)| \leq 34\pi^{2}\frac{\frc^{*}(g)[\vv]_{H^{1}(\T)}(N+1)}{\sigma_2} \leq 256\pi^{2}\frc^{*}(g) [\vv]_{H^{1}(\T)} \frac{\ell(\ga_2)}{\sigma_2 \ell_{\kappa}(\ga_2)} = (8\pi)^{3}\frc^{*}(g)[\vv]_{H^{1}(\T)}\kappa^2_2(\ga_2)
$$
for the velocity.

To show \eqref{eq:goal1} fix $0 \leq k \leq N$ arbitrary, then let $y_{k}(x)$ denote the least solution to the minimization
$$
\min_{ x_{k-1} \leq y \leq x_{k+1} } \;\, |\ga_1(x) - \ga_2(y)|^{2}
$$
and $d_{k}(x) := |\ga_1(x) - \ga_2(y_{k}(x))|$ the corresponding distance at the minimizer. A simple expansion of $f(y) := |\ga_1(x) - \ga_2(y)|^2$ yields
\begin{align*}
f(y) &= d^2_{k}(x) + \sigma^{2}_{2}( y - y_{k}(x) )^{2} + \dot f( y_{k}(x) )(y - y_{k}(x)) + 2\int^{y}_{y_{k}(x)} \langle \ga_2(\vth) - \ga_1(x) , \ddot \ga_2(\vth) \rangle (y-\vth)\,\rd \vth,
\end{align*}
where $\sigma_2 := |\dot \ga_2(y)|$ denotes the constant speed of the corresponding curve.

Consider first the case when $d_{k}(x)>0,$ and note that the inequality $\dot f( y_k(x) )(y - y_k(x)) \geq 0$ holds on $[x_{k-1},x_{k+1}]$ by first-order optimality. Thus the inequality
\begin{align*}
\frac{f(y)}{d^2_k(x) + \sigma^{2}_2 (y-y_k(x))^2} \geq 1 + \frac{2 \int^{y}_{y_k(x)} (y - \vth) \langle \ga_2(\vth) - \ga_1(x),\ddot \ga_2(\vth) \rangle \,\rd \vth}{d^2_k(x) + \sigma^{2}_2 (y-y_k(x))^2} := 1 + \frac{ 2 R(x,y) }{d^2_k(x) + \sigma^{2}_2 (y-y_k(x))^2}
\end{align*}
must hold as well. The simple estimate
\begin{align}\label{eq:remest1}
|R(x,y)| &= \left| \int^{y}_{y_k(x)} (y-\vth)\left( \langle \ga_2(\vth) - \ga_2(y_k(x)) , \ddot \ga_2(\vth) \rangle + \langle \ga_2(y_k(x)) - \ga_1(x) , \ddot \ga_2(\vth) \rangle\right)\, \rd \vth  \right| \nonumber \\
&\leq \sigma_2 \int^{y}_{y_k(x)} (y-\vth)(\vth-y_k(x))|\ddot \ga_2(\vth)| \, \rd \vth + d_k(x) \int^{y}_{y_k(x)} (y-\vth)|\ddot \ga_2(\vth)| \, \rd \vth  \\
&\leq \left( \frac{\sigma^2_2}{4}(y-y_k(x))^2 + \frac{d_k(x)|y-y_k(x)|\sigma_2}{\sqrt{3}}\right) \left( \frac{ \ell_{\ga_2}(x_{k-1,k+1})}{\ell_{\kappa}(\ga_2) } \right)^{\frac12} \nonumber
\end{align}
then follows by Cauchy-Schwarz. Applying Young's inequality $ab \leq c a^2/2 + b^2/2c$ for $c>0$ appropriately gives
\begin{align*}
|R(x,y)| \leq \frac{\sigma^2_2(y-y_k(x))^2 + d^2_k(x)}{2} \left( \frac{ \ell_{\ga_2}(x_{k-1},x_{k+1})}{\ell_{\kappa}(\ga_2) } \right)^{\frac12},
\end{align*}
and so the inequality
$$
|\ga_1(x) - \ga_2(y)|^2 \geq \left( 1 - \left( \frac{ \ell_{\ga_2}(x_{k-1},x_{k+1})}{\ell_{\kappa}(\ga_2) } \right)^{\frac12} \right) \left( \sigma^2_2(y-y_k(x))^2 + d^2_k(x) \right)
$$
holds on $\mathrm{supp}(\zeta_k)$. In particular, for $N = 4N(\ga_2)$ the lower bound
$$
|\ga_1(x) - \ga_2(y)|^2 \geq \frac14 \left( \sigma^2_2(y-y_k(x))^2 + d^2_k(x) \right) \qquad \text{for all} \qquad y \in \mathrm{supp}(\zeta_k)
$$
follows. Now decompose
$$
f^{\veps}_{k}(x) := \int_{\T} \dot g \left( \frac12|\ga_1(x) - \ga_2(y)|^2 + \frac{\veps^2}{2} \right)\langle \mathbf{v}(y), \ga_1(x) - \ga_2(y) \rangle \zeta_k(y) \, \rd y = \mathrm{I}_k(x) + \mathrm{II}_k(x)
$$ 
according to the definitions
\begin{align*}
\mathrm{I}_k(x) &:= \int_{\T} \dot g\left( \frac12|\ga_1(x) - \ga_2(y)|^2 + \frac{\veps^2}{2}\right)\langle \ga_1(x) - \ga_2( y_k(x) ), \mathbf{v}(y) \rangle \zeta_{k}(y) \, \rd y, \\
\mathrm{II}_k(x) &:= \int_{\T} \dot g\left( \frac12|\ga_1(x) - \ga_2(y)|^2 + \frac{\veps^2}{2}\right)\langle \ga_2(y_k(x) ) - \ga_2(y), \mathbf{v}(y) \rangle \zeta_{k}(y) \, \rd y.
\end{align*}
As the kernel $g$ is regular, there exists a constant $\frc^{*}(g)$ so that
$$
\sup_{ s > 0 } \; |s \dot g(s)| \leq \frc^{*}(g)
$$
holds. For the first term, the estimate
\begin{align*}
|\mathrm{I}_k(x)| &\leq 8\frc^{*}(g) d_{k}(x)\|\vv\|_{L^{\infty}(\T)} \int^{x_{k+1}}_{x_{k-1}} \frac{\zeta_k(y)}{\sigma^2_2(y-y_k(x))^2 + d^2_k(x)} \, \rd y \\
&\leq \frac{8 \frc^{*}(g) \|\vv\|_{L^{\infty}(\T)}}{\sigma_2} \int_{\R} \frac{ \zeta_k\left( y_k(x) + \frac{d_k(x)}{\sigma_2}v \right)}{1 + v^2} \, \rd v  \leq \frac{8\pi \frc^{*}(g) \|\vv\|_{L^{\infty}(\T)}}{\sigma_2}
\end{align*}
will suffice. For the second term, note that there exists a $\Xi(x,y)$ so that
\begin{align*}
\dot g\left(\frac12|\ga_1(x) - \ga_2(y)|^2 + \frac{\veps^2}{2} \right) &= \dot g\left(\frac{d^2_k(x) + \sigma^2_2(y-y_k(x))^2}{2} + \frac{\veps^2}{2} \right) + \ddot g\left( \Xi(x,y) + \frac{\veps^2}{2} \right)R(x,y)\\
R(x,y) &= \dot f\big( y_k(x) \big)(y - y_k(x) )+ 2\int^{y}_{y_{k}(x)} \langle \ga_2(\vth) - \ga_1(x) , \ddot \ga_2(\vth) \rangle (y-\vth)\,\rd \vth\\
\Xi(x,y) & \geq \frac14 \left( \sigma^2_2(y-y_k(x))^2 + d^2_k(x) \right) \qquad \text{for all} \qquad y \in \mathrm{supp}(\zeta_k)
\end{align*}
by the mean value theorem. The decomposition $\mathrm{II}_{k}(x) = \mathrm{II}^{{\rm i}}_{k}(x) + \mathrm{II}^{{\rm ii}}_{k}(x)+ \mathrm{II}^{{\rm iii}}_{k}(x)$ with
\begin{align*}
\mathrm{II}^{{\rm i}}_k(x) &= \int_{\R} \dot g\left( \frac{d^2_k(x) + \sigma^2_2(y-y_k(x))^2}{2} + \frac{\veps^2}{2} \right) (y_k(x) - y)\langle \vv(y), \dot \ga_2( y_k(x) )\rangle   \zeta_k(y) \, \rd y \\
\mathrm{II}^{{\rm ii}}_k(x) &= \int_{\R} \dot g\left( \frac{d^2_k(x) + \sigma^2_2(y-y_k(x))^2}{2} + \frac{\veps^2}{2} \right)\left( \int^{y}_{y_k(x)} (\vth-y) \langle \vv(y) , \ddot \ga_2( \vth )\rangle \, \rd \vth \right) \zeta_k(y) \, \rd y \\
\mathrm{II}^{{\rm iii}}_k(x) &= \int_{\R} \ddot g\left( \Xi(x,y) + \frac{\veps^2}{2} \right)\langle \ga_2( y_k(x) ) - \ga_2(y),\vv(y)\rangle R(x,y) \zeta_k(y) \, \rd y
\end{align*}
then holds by expansion of $\langle \ga_2( y_k(x) ) - \ga_2(y),\vv(y)\rangle$ to leading-order. Put $Q_{k}(y) := \langle \vv(y), \dot \ga_2( y_k(x) )\rangle   \zeta_k(y),$ so that $\mathrm{supp}(Q_k) \subset [x_{k-1},x_{k+1}]$ and
\begin{align*}
&\int_{\R} \dot g\left( \frac{d^2_k(x) + \sigma^2_2(y-y_k(x))^2}{2} + \frac{\veps^2}{2} \right) (y_k(x) - y) Q_k(y) \, \rd y = \int_{\R} \dot g\left( \frac{d^2_k(x) + \sigma^2_2 z^2}{2} + \frac{\veps^2}{2} \right) Q_k( y_k(x) - z )z \, \rd z \\
&= \int^{\delta}_{-\delta} \dot g\left( \frac{d^2_k(x) + \sigma^2_2 z^2}{2} + \frac{\veps^2}{2} \right)\big( Q_k( y_k(x) - z ) - Q_k(y_k(x)) \big)z \, \rd z
\end{align*}
for any $\delta>0$ as long as $\mathrm{supp}( Q_k( y_k(x) - \cdot ) ) \subset (-\delta,\delta);$ in particular, the choice $\delta = 4\pi/N$ will do. Now
$$
|Q_{k}(y_k(x)-z) - Q_{k}(y_k(x))| = \left|\int^{y_{k}(x)}_{y_{k}(x)-z} \dot Q_k(\vth) \, \rd \vth \right| \leq \|\dot Q_k\|_{L^{2}(\R)}|z|^{\frac12}
$$
by Cauchy-Schwarz, and since $g$ is regular the bound
$$
|\mathrm{II}^{{\rm i}}_k(x)| \leq 2 \frc^{*}(g) \|\dot Q_k\|_{L^{2}(\R)} \int^{\delta}_{-\delta} \frac{ z^{\frac32} }{ d^2_k(x) + \sigma^2_2 z^2 } \, \rd z \leq \frac{8\frc^{*}(g) \|\dot Q_k\|_{L^{2}(\R)} }{\sigma^2_2}\left(\frac{4\pi}{N}\right)^{\frac12}
$$
then follows. As $\|\dot \zeta_k\|_{L^{\infty}(\R)} \leq N$ the crude upper bound
\begin{align*}
\| \dot Q_k\|_{L^{2}(\R)} &= \left( \int^{x_{k+1}}_{x_{k-1}} \big(  \langle \dot \vv(y),\dot \ga_2(y_k(x)) \rangle\zeta_k(y) + \langle \vv(y),\dot \ga_2(y_k(x)) \rangle \dot \zeta_k(y) \big)^{2} \, \rd y \right)^{\frac12} \\
& \leq \sigma_2 \|\dot \vv\|_{L^{2}(\T)} + \sigma_2 \|\vv\|_{L^{\infty}(\T)}N\sqrt{x_{k+1}-x_{k-1}} \leq \sigma_2\big( 1 + \sqrt{4\pi N} \big)\max\{ \|\dot \vv\|_{L^{2}(\T)}, \|\vv\|_{L^{\infty}(\T)} \}
\end{align*}
holds, and so the upper bound
$$
|\mathrm{II}^{{\rm i}}_k(x)| \leq \frac{36\pi \frc^{*}(g)  }{\sigma_2}\max\{ \|\dot \vv\|_{L^{2}(\T)},\|\vv\|_{L^{\infty}(\T)} \}
$$
holds as well. The simple estimate
$$
\left| \int^{y}_{y_k(x)} \langle  \vv(y) , \ddot \ga_2(\vth) \rangle (y - \vth) \, \rd \vth \right| \leq   \frac{ \|\vv\|_{L^{\infty}(\T)} \sigma^{\frac32}_{2}(y-y_k(x))^{\frac32} }{\ell^{\frac12}_{\kappa}(\ga_2)}
$$
and the regularity of $g$ give
\begin{align*}
|\mathrm{II}^{{\rm ii}}_k(x)| &\leq \frac{2\frc^{*}(g)\|\vv\|_{L^{\infty}(\T)}}{\ell^{\frac12}_{\kappa}(\ga_2)} \int^{x_{k+1}}_{x_{k-1}}  \sigma^{-\frac12}_2 |y-y_k(x)|^{-\frac12} \, \rd y \leq \frac{4 \frc^{*}(g)\|\vv\|_{L^{\infty}(\T)}}{\sigma^{\frac12}_2 \ell^{\frac12}_{\kappa}(\ga_2)} \left( \sqrt{ x_{k+1} - y_k(x) } + \sqrt{ y_k(x) - x_{k-1} } \right) \\
& \leq 8\frc^{*}(g)\|\vv\|_{L^{\infty}(\T)} \left( \frac{2\pi}{N \sigma_2 \ell_{\kappa}(\ga_2)} \right)^{\frac12} \leq \frac{4\frc^{*}(g)\|\vv\|_{L^{\infty}(\T)}}{\sigma_2}
\end{align*}
in a similar fashion, with the choice $N = 4N(\ga_2) \geq 4 \ell_{\ga_2}/\ell_{\kappa}(\ga_2) = 8\pi\sigma_2/\ell_{\kappa}(\ga_2)$ being sufficient to justify the final inequality. To conclude the case when $d_{k}(x)>0$, note that $\dot f(y_k(x)) = \langle \ga_2( y_k(x) ) - \ga_1(x) , \dot \ga_2(x) \rangle$ and so the upper bound (c.f. \ref{eq:remest1})
\begin{align*}
|R(x,y)| &\leq d_{k}(x)\sigma_2|y-y_k(x)| + \frac{d_k(x)\sigma^{\frac32}_2|y-y_k(x)|^{\frac32}}{\sqrt{3}\ell^{\frac12}_{\kappa}(\ga_2)} +\frac{\sigma^{\frac52}_2|y-y_k(x)|^{\frac52}}{2\ell^{\frac12}_{\kappa}(\ga_2)} \\
& \leq d_{k}(x)\sigma_2|y-y_k(x)|\left( 1 + \left(\frac{4\pi \sigma_2 }{3N\ell_{\kappa}(\ga_2)}\right)^{\frac12} \right) +  \frac{\sigma^{\frac52}_2|y-y_k(x)|^{\frac52}}{2\ell^{\frac12}_{\kappa}(\ga_2)} \qquad \text{on} \qquad \mathrm{supp}(\zeta_k)\\
&\leq 2 d_{k}(x)\sigma_2|y-y_k(x)| +  \frac{\sigma^{\frac52}_2|y-y_k(x)|^{\frac52}}{2\ell^{\frac12}_{\kappa}(\ga_2)} \qquad \text{on} \qquad \mathrm{supp}(\zeta_k)
\end{align*}
holds for the remainder. The regularity of $g$ ensures
$$
\sup_{s > 0} \;\, |s^{2}\ddot g(s)| \leq \frc^{*}(g)
$$
by assumption, and so the estimate
\begin{align*}
|\mathrm{II}^{{\rm iii}}_{k}(x)| &\leq 2\frc^{*}(g)\|\vv\|_{L^{\infty}(\T)}\int_{\R} \frac{d_{k}(x)\sigma^2_2|y-y_k(x)|^2 + \frac12\ell^{-\frac12}_{\kappa}(\ga_2)\sigma^{\frac72}_2 (y-y_k(x))^{\frac72}}{\Xi^{2}(x,y)} \zeta_k(y) \, \rd y\\
& \leq 8 \frc^{*}(g)\|\vv\|_{L^{\infty}(\T)} \int_{\R} \frac{d_{k}(x)\sigma^2_2|y-y_k(x)|^2 + \frac12\ell^{-\frac12}_{\kappa}(\ga_2)\sigma^{\frac72}_2 |y-y_k(x)|^{\frac72}}{ \left(d^2_k(x) + \sigma^2_2 (y-y_k(x))^2 \right)^2} \zeta_k(y) \, \rd y \\
& \leq 8\frc^{*}(g)\|\vv\|_{L^{\infty}(\T)} \left( \frac1{\sigma_2}\int_{\R} \frac{v^2}{(1+v^2)^2} \, \rd v + \frac1{2\ell^{\frac12}_{\kappa}(\ga_2)}\int^{x_{k+1}}_{x_{k-1}} \sigma^{-\frac12}_2|y-y_k(x)|^{-\frac12}\,\rd y \right) \\
& \leq \frac{8\pi\frc^{*}(g)\|\vv\|_{L^{\infty}(\T)}}{\sigma_2}
\end{align*}
holds. All together, the desired estimate
\begin{align*}
|f^{\veps}_k(x)| &\leq |\mathrm{I}_k(x)| + |\mathrm{II}^{{\rm i}}_k(x)| + |\mathrm{II}^{{\rm ii}}_k(x)| + |\mathrm{II}^{{\rm iii}}_k(x)| \leq \left(52 \pi + 4 \right)\frac{ \frc^{*}(g)\max\{\|\vv\|_{L^{\infty}(\T)},\|\dot \vv\|_{L^{2}(\T)}\} }{\sigma_2} \\
&\leq 34\pi^2\frc^{*}(g)\frac{[ \vv]_{H^{1}(\T)}}{\sigma_2}
\end{align*} 
holds in this case.

It remains to consider the case when the distance $d_k(x) = 0$ vanishes. Then $\ga_1(x) = \ga_2(y_k(x))$ and so $f^{\veps}_k(x) = \mathrm{I}_k(x) + \mathrm{II}_k(x) + \mathrm{III}_k(x)$ decomposes according to
\begin{align*}
\mathrm{I}_k(x) &= \int_{\R} \dot g \left( \frac{\sigma^2_2(y-y_k(x))^2}{2} + \frac{\veps^2}{2} \right)(y-y_k(x))Q_{k}(y)\, \rd y \\
\mathrm{II}_k(x) &= \int_{\R} \dot g\left( \frac{\sigma^2_2(y-y_k(x))^2}{2} + \frac{\veps^2}{2} \right)\left( \int^{y}_{y_k(x)} (\vth-y) \langle \vv(y) , \ddot \ga_2( \vth )\rangle \, \rd \vth \right) \zeta_k(y) \, \rd y \\
\mathrm{III}_k(x) &= \int_{\R} \ddot g\left( \Xi(x,y) + \frac{\veps^2}{2} \right)\langle \ga_2( y_k(x) ) - \ga_2(y),\vv(y)\rangle R(x,y) \zeta_k(y) \, \rd y,
\end{align*}
where since $\dot f(y_k(x)) = 0$ the error $\Xi(x,y)$ and remainder $R(x,y)$ satisfy
$$
\Xi(x,y) \geq \frac14 \left( \sigma^2_2(y-y_k(x))^2 \right) \qquad \text{and} \qquad |R(x,y)| \leq  \frac{\sigma^{\frac52}_2|y-y_k(x)|^{\frac52}}{2\ell^{\frac12}_{\kappa}(\ga_2)} 
$$
on $\mathrm{supp}(\zeta_k)$. The desired estimate of $\mathrm{I}_k(x)$ follows by mimicking the estimate of $\mathrm{II}^{{\rm i}}_k(x),$ the estimate of $\mathrm{II}_k(x)$ follows by mimicking the estimate of $\mathrm{II}^{{\rm ii}}_k(x)$ and the estimate of $\mathrm{III}_k(x)$ follows by mimicking the estimate of the second term of $\mathrm{II}^{{\rm iii}}_k,$ respectively. In all cases, the desired bound
\begin{align*}
|f^{\veps}_k(x)| &\leq 34\pi^{2}\frc^{*}(g)\frac{[\vv]_{H^{1}(\T)}}{\sigma_2}
\end{align*}
therefore holds.
\end{proof}

\noindent With this lemma in hand, we may now proceed to prove further regularity and continuity properties of the nonlocal forcing. More specifically, we may show ---
\begin{proposition}\label{prop:nlest}
Assume the kernel $g$ satisfies {\rm H0)}, {\rm H1)} and that $\ga_{i} \in H^{2}(\T),\, i \in \{0,1,2\}$ have constant speed. Then the nonlocal forcings
$$
f_{i}(x) := \int_{\T} g\left( \frac12|\ga_2(x) - \ga_i(y)|^{2} \right)|\dot \ga_i(y)|\,\rd y
$$
obey $f_i \in W^{1,\infty}(\T),$ the bound
$$
\|\dot f_i\|_{L^{\infty}(\T)} \leq (16\pi)^{2} \frc^{*}_1(g) \left( \frac{\ell(\ga_i)}{\ell_{\kappa}(\ga_i)}\right)\sigma(\ga_2) \qquad \text{for all} \qquad i \in \{0,1\}
$$
as well as the Lipschitz estimate
\begin{align}\label{eq:lip0}
\|f_1 - f_0\|_{L^{\infty}(\T)} &\leq \mathfrak{C}^{*}\left(\kappa^2_2(\ga_0)\vee \kappa^2_2(\ga_1),\frac{\ell(\ga_1)}{\ell_{\kappa}(\ga_1)}\vee\frac{\ell(\ga_0)}{\ell_{\kappa}(\ga_0)},\frac{\ell(\ga_1)\vee\ell(\ga_2)}{\ell(\ga_1)\wedge\ell(\ga_2)}\right)\frc^{*}_{0,1}(g)\| \ga_1 -  \ga_0\|_{H^{1}(\T)} \nonumber \\
\frc^{*}_{0,1}(g) &:= \frc^{*}_{0}(g)+\frc^{*}_{1}(g)
\end{align}
for $\mathfrak{C}^{*}: \R^{3} \mapsto [0,\infty)$ some continuous, coordinate-wise increasing function. The nonlocal forcings
$$
f_i(x) := \int_{\T} g\left( \frac12|\ga_i(x) - \ga_2(y)|^2 \right) |\dot \ga_2(y)| \, \rd y
$$
obey the Lipschitz estimate
\begin{align}\label{eq:lip1}
\|f_1 - f_0\|_{L^{\infty}(\T)} & \leq (16\pi)^2 \frc^{*}_{1}(g) \left( \frac{\ell(\ga_2)}{\ell_{\kappa}(\ga_2)}\right) \| \ga_1 -  \ga_0\|_{L^{\infty}(\T)}
\end{align}
for $\ga_2 \in H^{2}(\T)$ arbitrary.
\end{proposition}
\begin{proof}
For $\veps>0$ the forcings
$$
f^{\veps}_{i}(x) := \sigma(\ga_i)\int_{\T} g\left( \frac12|\ga_2(x) - \ga_i(y)|^{2} + \frac{\veps^2}{2}\right)\,\rd y
$$
are sufficiently regular to differentiate
$$
\dot f^{\veps}_{i}(x) = \sigma(\ga_i) \int_{\T} \dot g\left( \frac12|\ga_2(x) - \ga_i(y)|^{2} + \frac{\veps^2}{2}\right)\langle \ga_2(x) - \ga_i(y), \dot \ga_2(x) \rangle \,\rd y.
$$
Applying lemma \ref{lem:derbound} with $\vv(y) = \dot \ga_2(x)$ a constant function then gives
$$
\|\dot f^{\veps}_i\|_{L^{\infty}(\T)} \leq  (16\pi)^{2} \frc^{*}_1(g) \left( \frac{\ell(\ga_i)}{\ell_{\kappa}(\ga_i)}\right)\sigma(\ga_2),
$$
and moreover the integral identity
$$
\int_{\T} f^{\veps}_{i}(x) \dot \phi(x) \, \rd x = - \int_{\T} \dot f^{\veps}_i(x) \phi(x) \, \rd x
$$
holds for $\phi$ any smooth, periodic test function. But $f^{\veps}_i \to f_i$ in $L^{1}(\T)$ as $\veps \to 0$ and the $\dot f^{\veps}_i$ are uniformly bounded in $L^{2}(\T)$, so by passing to a subsequence if necessary, for all $\phi$ the identity
$$
\int_{\T} f_{i}(x) \dot \phi(x) \, \rd x = - \int_{\T} \dot f_i(x) \phi(x) \, \rd x
$$
holds with $\dot f_i$ an $L^{2}(\T)$ weak limit along a subsequence. Thus $f_i \in W^{1,\infty}(\T)$ and the claimed bound holds.

To prove the first Lipschitz estimate, assume first that $\ga_0,\ga_1$ are compatibly oriented. That is, $\langle \dot \ga_0(x),\dot \ga_1(x) \rangle \geq 0$ holds everywhere on $\T$.  Put $\ga(x,\vth) := \vth \ga_1(x) + (1-\vth)\ga_0(x)$ and let $\psi(x,\vth)$ denote the constant speed representation of this family (c.f. \eqref{eq:cspdrep}). Then $\ga(x,0) = \ga_0(x),\,\ga(x,1) = \ga_1(x)$ and as long as
$$
0 < |\dot \ga(x,\vth)|^2 = \vth \sigma^2(\ga_1) + (1-\vth)\sigma^2(\ga_1)  - \vth(1-\vth)|\dot \ga_0(x) - \dot \ga_1(x)|^2
$$
the map $\vth \mapsto \ga(x,\vth)$ defines a one-parameter family of immersions. In particular, if $\ga_0,\ga_1$ are compatibly oriented then
$$
|\dot \ga(x,\vth)| \geq \frac{ \min\{ \sigma(\ga_0),\sigma(\ga_1) \} }{\sqrt{2}}
$$
for all $\vth$ in the unit interval. Assume, without loss of generality, that $\sigma(\ga_0)$ achieves the minimum. Let $\sigma(\vth) := |\dot \psi(x,\vth)|$ denote the speed of the interpolant, $\ell(\vth)$ its length, $\ell_{\kappa}(\vth)$ its curvature length-scale and
$$
f^{\veps}_{\vth}(x) := \int_{\T} g\left( \frac12|\ga_2(x) - \psi(y,\vth)|^{2} + \frac{\veps^2}{2} \right)\sigma(\vth) \, \rd y
$$
the corresponding nonlocal force. The relation $\ga_{\vth}(x,\vth) = \Delta \ga(x)$ and \eqref{eq:cspdder} give 
\begin{align*}
\dot \ell(\vth) = \int_{\T} \langle \tang_{\psi},\dot \vv \rangle \, \rd x,  \quad \vv(x,\vth) = \Delta \ga( \xi(x,\vth) ) \quad \text{and} \quad \psi_{\vth} = \vv - \po\left( \langle \tang_{\psi} , \dot \vv \rangle \right)\tang_{\psi},
\end{align*}
while the identity
\begin{align*}
f^{\veps}_1(x) -  f^{\veps}_0(x) &= \int^{1}_{0}\left( I_1(x,\vth) + I_2(x,\vth) \right) \, \rd \vth \qquad I_{1}(x,\vth) = \frac{\dot \sigma(\vth)}{\sigma(\vth)} f^{\veps}_{\vth}(x)\\
I_{2}(x,\vth) &= \int_{\T} \dot g \left( \frac12|\ga_2(x) - \psi(y,\vth)|^2 +\frac{\veps^2}{2}\right)\langle \ga_2(x) - \psi(y,\vth),\psi_{\vth}(y,\vth)\rangle \sigma(\vth) \, \rd y
\end{align*}
holds simply by differentiating. On one hand, lemma \ref{lem:velbound} gives
$$
|I_1(x,\vth)| \leq 8\pi\,\frc^{*}_{0}(g)\left(\frac{\ell(\vth)}{\ell_{\kappa}(\vth)}\right)\frac{\dot \sigma(\vth)}{\sigma(\vth)} \leq 8\pi \frc^{*}_{0}(g)\left(\frac{\|\dot \ga_1 - \dot \ga_0\|_{L^{1}(\T)}}{\ell_{\kappa}(\vth)}\right)
$$
since $\psi$ has constant speed. On the other hand, lemma \ref{lem:derbound} gives
\begin{align*}
|I_2(x,\vth)| &\leq  (8\pi)^{3} \frc^{*}_{1}(g)\left( \frac{\sigma(\vth)}{\ell_{\kappa}(\vth)} \right)[\psi_{\vth}]_{H^{1}(\T)} \\
[\psi_{\vth}]_{H^{1}(\T)}  &:= \max\{ \|\psi_{\vth}\|_{L^{\infty}(\T)} , \| \dot \psi_{\vth} \|_{L^{2}(\T)} \}
\end{align*}
for the same reason. Now
\begin{align*}
\|\psi_{\vth}\|_{L^{\infty}(\T)} &\leq \|\Delta \ga\|_{L^{\infty}(\T)} + 2\|\dot \ga_1 - \dot \ga_0\|_{L^{2}(\T)} \leq \frac52 \|\ga_1 - \ga_0\|_{W^{1,1}(\T)} \leq \frac{5\sqrt{1+2\pi}}{2} \|\ga_1 - \ga_0\|_{H^{1}(\T)}
\end{align*}
since the embedding \eqref{eq:gns} holds. The triangle inequality, a direct computation and Jensen's inequality then give
\begin{align*}
\| \dot \psi_{\vth} \|_{L^{2}(\T)} &\leq \| \dot \vv\|_{L^{2}(\T)} + \|\po\big( \langle \dot \vv,\tang_{\psi} \rangle \big)\|_{L^{\infty}(\T)}\| \dot \tang_{\psi}\|_{L^{2}(\T)}\\
&= \| \dot \vv\|_{L^{2}(\T)} + \|\po\big( \langle \dot \vv,\tang_{\psi} \rangle \big)\|_{L^{\infty}(\T)} \left( \frac{\ell(\vth)}{2\pi\ell_{\kappa}(\vth)} \right)^{\frac12},
\end{align*}
which yields the overall bound
$$
\| \dot \psi_{\vth} \|_{L^{2}(\T)}\leq \left( \left( \frac{2\sigma(\vth)}{\sigma(0)} \right)^{\frac12} +\left( \frac{\ell(\vth)}{\ell_{\kappa}(\vth)} \right)^{\frac12} \right)\| \ga_1 - \ga_0\|_{H^{1}(\T)}
$$
after undoing the change of variables. In particular, the latter bound will achieve the maximum. All together, this yields the bound
\begin{align*}
\|f_1 - f_0\|_{L^{\infty}(\T)} &\leq 8\pi \frc^{*}_{0}(g)\left( \int^{1}_{0} \kappa^2_2(\vth) \, \rd \vth\right)\|\dot \ga_1 - \dot \ga_0\|_{L^{1}(\T)} \\
&+ (16\pi)^2 \frc^{*}_{1}(g) \left[\int^{1}_{0} \left(  \left(\frac{2\sigma(\vth)}{\sigma(0)} \right)^{\frac12} + \left( \frac{\ell(\vth)}{\ell_{\kappa}(\vth)} \right)^{\frac12} \right)\frac{\ell(\vth)}{\ell_{\kappa}(\vth)}\, \rd \vth \right]  \| \ga_1 -  \ga_0\|_{H^{1}(\T)}
\end{align*}
for the difference. But $\sigma(1) \geq \sigma(0),$ $\sigma(\vth) \leq \sigma(1)$ and $|\dot \ga(x,\vth)| \geq \sigma(0)/\sqrt{2}$ combine to yield
\begin{align*}
\int^{1}_{0} \kappa^2_2(\vth) \, \rd \vth &\leq 2\left( 1 + \left(\frac{\sigma(\ga_1)}{\sigma(\ga_0)}\right)^3 \right)\max\{ \kappa^2_2(0),\kappa^2_2(1) \} \\
\frac{\ell(\vth)}{\ell_{\kappa}(\vth)}& \leq  \frac{\sqrt{2}\sigma(\ga_1)}{\sigma(\ga_0)}\left( 2\vth \left(\frac{\sigma(\ga_1)}{\sigma(\ga_0)}\right)^{2} + (1-\vth) \right)\max\left\{ \frac{\ell(0)}{\ell_{\kappa}(0)} ,\frac{\ell(1)}{\ell_{\kappa}(1)}\right\},
\end{align*}
and so all together there exists some continuous, coordinate-wise increasing function $\mathfrak{C} : \R^{3} \mapsto [0,\infty)$ so that the claimed bound
$$
\|f_1 - f_0\|_{L^{\infty}(\T)} \leq \mathfrak{C}^{*}\left(\kappa^2_2(\ga_0)\vee \kappa^2_2(\ga_1),\frac{\ell(\ga_1)}{\ell_{\kappa}(\ga_1)}\vee\frac{\ell(\ga_0)}{\ell_{\kappa}(\ga_0)},\frac{\ell(\ga_1)\vee\ell(\ga_2)}{\ell(\ga_1)\wedge\ell(\ga_2)}\right)\frc^{*}_{0,1}(g)\| \ga_1 -  \ga_0\|_{H^{1}(\T)}
$$
holds. If $\ga_0,\ga_1$ are not compatibly orientated then there exists some $x$ for which $\langle \dot \ga_1(x),\dot \ga_0(x)\rangle <0,$ and so necessarily the lower bound
$$
\|\ga_1 - \ga_0\|^2_{L^{\infty}(\T)} \geq |\dot \ga_1(x) - \dot \ga_0(x)|^2 = \sigma^2(1) + \sigma^2(0) - 2 \langle \dot \ga_1(x) ,\dot \ga_0(x) \rangle \geq  2\sigma^2(0)
$$
must hold. By the embedding
\begin{align*}
\|\dot \ga_1 - \dot \ga_0\|_{L^{\infty}(\T)} &\leq \|\dot \ga_1 - \dot \ga_0\|^{\frac12}_{L^{2}(\T)}\|\ddot \ga_1 - \ddot \ga_0\|^{\frac12}_{L^{2}(\T)} \leq \sqrt{2}\|\dot \ga_1 - \dot \ga_0\|^{\frac12}_{L^{2}(\T)}\left( \|\ddot \ga_0\|_{L^{2}(\T)} \vee  \|\ddot \ga_1\|_{L^{2}(\T)} \right)^{\frac12} \\
& \leq \frac{\sigma(1)}{\sqrt{\pi}}\left( \frac{\ell(\ga_1)}{\ell_{\kappa}(\ga_1)}\vee\frac{\ell(\ga_0)}{\ell_{\kappa}(\ga_0)} \right)^{\frac12}\|\dot \ga_1 - \dot \ga_0\|^{\frac12}_{L^{2}(\T)}
\end{align*}
the lower bound
$$
\|\dot \ga_1 - \dot \ga_0\|_{L^{2}(\T)}\left( \frac{\ell(\ga_1)}{\ell_{\kappa}(\ga_1)}\vee\frac{\ell(\ga_0)}{\ell_{\kappa}(\ga_0)} \right)\frac{\sigma^{2}(1)}{\sigma^2(0)} \geq 2\pi
$$
must hold as well. But then by lemma \ref{lem:velbound} and the triangle inequality, the claimed bound
\begin{align*}
\|f_1 - f_0\|_{L^{\infty}(\T)} \leq 16 \pi \frc^{*}_{0}(g) \left( \frac{\ell(\ga_1)}{\ell_{\kappa}(\ga_1)}\vee\frac{\ell(\ga_0)}{\ell_{\kappa}(\ga_0)} \right) \leq 8 \frc^{*}_{0}(g)\left( \frac{\ell(\ga_1)}{\ell_{\kappa}(\ga_1)}\vee\frac{\ell(\ga_0)}{\ell_{\kappa}(\ga_0)} \right)^2\left(\frac{\sigma(1)}{\sigma(0)}\right)^{2}\|\dot \ga_1 - \dot \ga_0\|_{L^{2}(\T)}
\end{align*}
holds in this case as well. For the second Lipschitz estimate it suffices, by changing variables if necessary, to assume that $\ga_2$ has constant speed. The estimate then follows easily from lemma \ref{lem:derbound}, for
\begin{align*}
| f^{\veps}_{0}(x) - f^{\veps}_{1}(x) | &= \sigma(\ga_2)\left| \int_{\T} g\left( \frac12|\ga_0(x) - \ga_2(y)|^2 + \frac{\veps^2}{2} \right) \, \rd y - \int_{\T} g\left( \frac12|\ga_1(x) - \ga_2(y)|^2 + \frac{\veps^2}{2}\right) \, \rd y \right| \\
&= \sigma(\ga_2)\left| \int^{1}_{0} \int_{\T} \dot g\left( \frac12|\ga_{\vth}(x) - \ga_2(y)|^2 + \frac{\veps^2}{2} \right) \langle \ga_{\vth}(x) - \ga_2(y),\ga_1(x) - \ga_0(x) \rangle \, \rd y \right| \\
&\leq (8\pi)^{3}\frc^{*}_{1}(g) \sigma(\ga_2)\kappa^2_2(\ga_2)\| \ga_1 - \ga_0\|_{L^{\infty}(\T)} = (16\pi)^2\frc^{*}_1(g)\frac{\ell(\ga_2)}{\ell_{\kappa}(\ga_2)}\|\ga_1 - \ga_0\|_{L^{\infty}(\T)}
\end{align*}
as claimed.
\end{proof}

Regular kernels form the widest class of kernels for which the corresponding nonlocal force behaves well with respect to the $H^{2}(\T)$ topology. Essentially, the occurrence of intersections between disjoint arcs forces the restriction to regular kernels. We may move beyond this class by ruling out the possibility of such intersections \emph{a-priori}, or in other words, by imposing additional geometric structure. For the case of self-interactions
\begin{align}\label{eq:self}
f_{ii}(x) := \int_{\T} g_{ii}\left( \frac12|\ga_i(x) - \ga_i(y)|^2 \right) |\dot \ga_i(y)| \, \rd y 
\end{align}
this means imposing that $\ga_i$ has finite distortion, while for the case of cross-interactions
\begin{align}\label{eq:cross}
f_{ij}(x) := \int_{\T} g_{ij}\left( \frac12|\ga_i(x) - \ga_j(y)|^2 \right) |\dot \ga_j(y)| \, \rd y 
\end{align}
this means imposing that $\ga_i,\ga_j$ lie at a positive distance from one another. Given a kernel $g \in C^{2}\big(\R^{+} \setminus \{0\} \big)$ we define
$$
 \dot{\mathcal{G}}^{*}(s) := \sup_{t \geq s} \;t |\dot g(t)| \qquad \text{and} \qquad \frc^{*}_{2}(g) := \int^{\infty}_{0} \dot{\mathcal{G}}^{*}\left( u^2 \right) \, \rd u 
$$
as the envelope function and corresponding integrability constant of its derivative. We shall employ the hypothesis
\begin{enumerate}[\indent\indent{\rm H2)}]
\item The kernel $g$ is singular: $\frc^{*}_{2}(g) < +\infty$.
\end{enumerate}
to force that the kernel behaves no worse than $ g(s) \sim C s^{-\frac12}$ locally near the origin. Under these assumptions, we may show ---
\begin{proposition}\label{prop:nlest0}
Assume the kernel $g$ satisfies {\rm H0)}, {\rm H2)} and that $\ga \in H^{2}(\T)$ has finite distortion. Then the nonlocal force
$$
f(x) := \int_{\T} g\left( \frac12|\ga(x) - \ga(y)|^2 \right) \, |\dot \ga(y)| \, \rd y
$$
lies in $H^{1}(\T)$ and obeys the bound
$$
\|f\|_{H^{1}(\T)} \leq \mathfrak{C}^{*}\left( \frac{\ell(\ga)}{\ell_{\kappa}(\ga)}, \delta_{\infty}(\ga) \right)\left( \frc^{*}_{0}(g) + \frc^{*}_{2}(g) \right)
$$
for $ \mathfrak{C}^{*}:\R^{2}_{+} \mapsto [0,\infty)$ some continous, increasing function of its arguments. If $\ga_i \in H^{2}(\T) , \, i \in \{0,1\}$ have constant speed then the difference $f_0 - f_1$ between their corresponding nonlocal forces obeys the $H^{1}(\T) \mapsto L^{2}(\T)$ estimate
\begin{align*}
\|f_0 - f_1\|_{L^{2}(\T)} &\leq \mathfrak{C}^{*}\left( \delta_{\infty}(\ga_0)\vee\delta_{\infty}(\ga_1),\frac{\ell(\ga_0)\vee\ell(\ga_1)}{\ell(\ga_0)\wedge\ell(\ga_1)} \right)\left( \frc^{*}_{0}(g) + \frc^{*}_{2}(g) \right)\left(\frac{ \|\dot \ga_1 - \dot \ga_0\|_{L^{2}(\T)}}{\ell(\ga_0)\vee\ell(\ga_1)}\right)
\end{align*}
for $\mathfrak{C^{*}}:\R^{2}_{+} \mapsto [0,\infty)$ some continuous, increasing function of its arguments.
\end{proposition}
\noindent The proof simply reiterates calculations from \cite{knot} (c.f. the appendix therein), to which we refer for the details.

While embeddedness properties ensure that self-interactions \eqref{eq:self} behave well, cross interactions \eqref{eq:cross} necessarily involve intersections between a pair $(\ga_i,\ga_j)$ of distinct immersions unless we enforce that they lie at a positive distance
$$
\mathrm{dist}\big(\ga_i,\ga_j\big) := \min_{(x,y) \in \T \times \T} \; |\ga_i(x) - \ga_j(y)|
$$
from one another. As for the allowable kernels in this setting, to make life easy we shall impose the hypothesis
\begin{enumerate}[\indent\indent{\rm H3)}]
\item The kernel $g$ is asymptotically finite: if $s>0$ then $\mathcal{G}^{*}(s)\vee \dot{\mathcal{G}}^{*}(s) < +\infty$. 
\end{enumerate}
where $\cG^{*}$ and $\dot{\mathcal{G}}^{*}$ denote the monotone non-increasing envelopes of $g(s)$ and $s \dot g(s),$ respectively. We allow the kernel to blow up arbitrarily quickly at the origin, but otherwise impose that it defines a bounded function. The far-field behavior does not matter locally in time, so we pay no price for enforcing that $\dot g(s)$ vanishes for $s$ large; we simply want $g$ and its derivatives bounded away from the origin. We adopt the hypothesis in this form simply to avoid introducing new notation. Under this hypothesis we may easily establish the needed properties ---
\begin{proposition}\label{prop:nlest2}
Assume the kernel $g$ satisfies {\rm H3)}, that $\mathrm{d}_{ij} := \mathrm{dist}\big(\ga_i,\ga_j\big) > 0$ for all $(i,j)$ and that $\ga_i \in H^{2}(\T),\, i \in\{0,1,2\}$ have constant speed. Then the nonlocal forcings
$$
f_{ij}(x) := \int_{\T} g\left( \frac12|\ga_i(x) - \ga_j(y)|^2 \right)|\dot \ga_j(y)| \, \rd y
$$
lie in $W^{1,\infty}(\T)$ with the estimate
$$
\|f_{ij}\|_{L^{\infty}(\T)} \leq \cG^{*}\left( \frac{\rd^2_{ij}}{2}\right)\ell(\ga_j) \qquad \text{and} \qquad \|\dot f_{ij}\|_{L^{\infty}(\T)} \leq \frac{2}{\rd_{ij}}\dot\cG^{*}\left( \frac{\rd^2_{ij}}{2}\right)\ell(\ga_i)\ell(\ga_j),
$$
and the corresponding Lipschitz estimates
\begin{align*}
\|f_{02} - f_{12}\|_{L^{\infty}(\T)} &\leq \dot \cG^{*}\left( \frac{\rd^2_{02}\wedge \rd^2_{12}}{2} \right)\frac{2}{\rd_{02}\wedge \rd_{12}}\|\ga_0 - \ga_1\|_{L^{\infty}(\T)}\left( 1 + \frac{2}{\rd_{02}\wedge \rd_{12}}\|\ga_0 - \ga_1\|_{L^{\infty}(\T)}\right)\ell(\ga_2) \\
\|f_{20} - f_{21}\|_{L^{\infty}(\T)} &\leq 2\pi \cG^{*}\left( \frac{\rd^2_{12}\wedge \rd^2_{02}}{2} \right)\|\dot \ga_0 - \dot \ga_1\|_{L^{1}(\T)}\\
&\dot \cG^{*}\left( \frac{\rd^2_{02}\wedge \rd^2_{12}}{2} \right)\frac{2}{\rd_{02}\wedge \rd_{12}}\|\ga_0 - \ga_1\|_{L^{\infty}(\T)}\left( 1 + \frac{2}{\rd_{02}\wedge \rd_{12}}\|\ga_0 - \ga_1\|_{L^{\infty}(\T)}\right)\ell(\ga_0)\wedge \ell(\ga_1)
\end{align*}
hold with a decreasing dependence on the minimal distance.
\end{proposition}
\begin{proof}
The $W^{1,\infty}(\T)$ bounds are completely trivial; only the Lipschitz estimates require any justification. For the first estimate, note
\begin{align*}
f_{02}(x) - f_{12}(x) = \sigma(\ga_2) \int_{\T} \left( g\left( \frac12|\ga_0(x) - \ga_2(y)|^2 \right) - g\left( \frac12|\ga_1(x) - \ga_2(y)|^2 \right) \right) \, \rd y
\end{align*}
and that there exists some $\Xi(x,y)$ between $|\ga_0(x) - \ga_2(y)|^2/2$ and $|\ga_1(x) - \ga_2(y)|^2/2$ so that
$$
g\left( \frac12|\ga_0(x) - \ga_2(y)|^2 \right) - g\left( \frac12|\ga_1(x) - \ga_2(y)|^2 \right) = \frac{\dot g\left( \Xi(x,y) \right)}{2}\langle \ga_0(x) - \ga_2(y) + \ga_1(x) - \ga_2(y),\ga_0(x) - \ga_1(x) \rangle 
$$
by the mean value theorem. In particular, for $(x,y)$ fixed the lower bound 
$$\Xi(x,y) \geq \left( \frac12|\ga_0(x) - \ga_2(y)|^2\right) \wedge \left( \frac12|\ga_1(x) - \ga_2(y)|^2\right)$$
holds. Assume $|\ga_0(x) - \ga_2(y)|$ achieves the minimum, so that
\begin{align*}
\frac{|\ga_0(x) - \ga_2(y)|}{\Xi(x,y)} \leq \frac{2}{|\ga_0(x) - \ga_2(y)|} \leq \frac{2}{\rd_{02}\wedge \rd_{12}} \qquad \text{and} \qquad \frac1{\Xi(x,y)} \leq \frac{2}{\rd^2_{02} \wedge \rd^2_{12}} 
\end{align*}
both hold. Write $h(x,y) := $
\begin{align*}
\langle \ga_0(x) - \ga_2(y) + \ga_1(x) - \ga_2(y),\ga_0(x) - \ga_1(x) \rangle = 2\langle \ga_0(x) - \ga_2(y),\ga_0(x) - \ga_1(x) \rangle - |\ga_1(x) - \ga_0(x)|^2 
\end{align*}
for the inner product, so that the upper bound
$$
\frac{h(x,y)}{\Xi(x,y)} \leq \frac{4}{\rd_{02}\wedge \rd_{12}}\|\ga_0 - \ga_1\|_{L^{\infty}(\T)} + \frac{2}{\rd^2_{02}\wedge \rd^2_{12}}\|\ga_1 - \ga_0\|^2_{L^{\infty}(\T)}
$$
holds. Thus for $(x,y)$ fixed the bound
\begin{align*}
&\left| \frac{\dot g\left( \Xi(x,y) \right)}{2}\langle \ga_0(x) - \ga_2(y) + \ga_1(x) - \ga_2(y),\ga_0(x) - \ga_1(x) \rangle \right| \\
&\leq \dot \cG^{*}\left( \frac{\rd^2_{02}\wedge \rd^2_{12}}{2} \right)\frac{2}{\rd_{02}\wedge \rd_{12}}\|\ga_0 - \ga_1\|_{L^{\infty}(\T)}\left( 1 + \frac{2}{\rd_{02}\wedge \rd_{12}}\|\ga_0 - \ga_1\|_{L^{\infty}(\T)}\right)
\end{align*}
follows, and so the claimed Lipschitz estimate follows upon integrating. The second Lipschitz estimate follows similarly. Simply write
\begin{align*}
f_{20}(x) - f_{21}(x) &= \frac{\sigma(\ga_0) - \sigma(\ga_1)}{\sigma(\ga_1)}f_{21}(x) \\
&\sigma(\ga_0) \int_{\T}\left[ g\left( \frac12|\ga_2(x) - \ga_0(y)|^2 \right) - g\left( \frac12|\ga_2(x) - \ga_1(y)|^2 \right) \right] \, \rd y := \mathrm{I}(x) + \mathrm{II}(x)
\end{align*}
and use the first part of the lemma to obtain the bound
$$
|\mathrm{I}(x)| \leq \frac{\|\dot \ga_1 - \dot \ga_0\|_{L^{1}(\T)}}{\sigma(\ga_1)}\cG^{*}\left( \frac{\rd^2_{21}}{2} \right)\ell(\ga_1) = 2\pi \cG^{*}\left( \frac{\rd^2_{21}}{2} \right)\|\dot \ga_0 - \dot \ga_1\|_{L^{1}(\T)}
$$
for the first term. An appropriate bound for the second term
\begin{align*}
|\mathrm{II}(x)|
&\leq \dot \cG^{*}\left( \frac{\rd^2_{02}\wedge \rd^2_{12}}{2} \right)\frac{2}{\rd_{02}\wedge \rd_{12}}\|\ga_0 - \ga_1\|_{L^{\infty}(\T)}\left( 1 + \frac{2}{\rd_{02}\wedge \rd_{12}}\|\ga_0 - \ga_1\|_{L^{\infty}(\T)}\right)\ell(\ga_0)
\end{align*}
follows by appealing to the mean value theorem and arguing as before.
\end{proof}

\noindent The exact form of these Lipschitz estimates will not prove too important in our analysis; we only need to know that the nonlocal forces satisfy some sort of uniform $H^{1}(\T)$  bound and an $H^{1}(\T) \mapsto L^{2}(\T)$ Lipschitz property locally near an initial condition. Under the requisite combination of analytic and geometric hypotheses, any of the three propositions \ref{prop:nlest},\ref{prop:nlest0} and \ref{prop:nlest2} will suffice. 

As we shall briefly need to work with non-canonical representations of immersions when pursuing local existence, we shall make a few brief observations that will allow us to apply these estimates for arbitrary immersions. Let $\ga_0,\ga_1,\ga_2 \in H^{2}(\T)$ denote any triple of $H^{2}(\T)$-immersions and let $\psi_i := \ga_i \circ \xi_i$ denote their constant speed equivalents. All three propositions assert that the nonlocal force
$$
f_{(\psi_0,\psi_1)} (x) := \int_{\T} g\left( \frac12|\psi_0(x) - \psi_1(y)|^{2} \right) |\dot \psi_1(y)| \, \rd y
$$
satisfies an $H^{1}(\T)$ bound of the form
$$
\|f_{(\psi_0,\psi_1)}\|_{H^{1}(\T)} \leq \mathfrak{C}^{*}\big(\ga_0,\ga_1\big),
$$
where $\mathfrak{C}^{*}\big(\ga_0,\ga_1\big)$ represents some continuous function depending on some combination of the parametrization invariant, geometric quantities
$$
\ell(\ga_i), \qquad \ell_{\kappa}(\ga_i), \qquad \delta_{\infty}(\ga_i) \qquad \text{and} \qquad \mathrm{dist}\big( \ga_0,\ga_1 \big) 
$$
that vary continuously with respect to the $H^{2}(\T)$ topology. A simple argument based on the identities
$$
f_{(\ga_0,\ga_1)}(x) = f_{(\psi_0,\psi_1)}\big(\eta_0(x)\big) \qquad \text{and} \qquad \dot f_{(\ga_0,\ga_1)}(x) = \dot f_{(\psi_0,\psi_1)}\big( \eta_0(x)\big) \dot \eta_0(x)
$$
and a straightforward change of variables shows that a similar $H^{1}(\T)$ bound
\begin{align}\label{eq:h1boundtmp}
\|f_{(\ga_0,\ga_1)}\|_{H^{1}(\T)} \leq \mathfrak{C}^{*}\big(\ga_0,\ga_1\big)
\end{align}
holds, provided we allow $\mathfrak{C}^{*}\big(\ga_0,\ga_1\big)$ to also depend upon the minimal speed $\frs_{*}(\ga_0)$ of the immersion. The local $H^{1}(\T) \mapsto L^{2}(\T)$ Lipschitz property follows from similar considerations. Propositions \ref{prop:nlest}, \ref{prop:nlest0} and \ref{prop:nlest2} furnish $H^{1}(\T) \mapsto L^{2}(\T)$ bounds of the form
\begin{align}\label{eq:dummybounds}
\|f_{(\psi_0,\ga_2)} - f_{(\psi_1,\ga_2)}\|_{L^{2}(\T)} &\leq \mathfrak{C}^{*}\big(\ga_0,\ga_1,\ga_2\big)\|\psi_0 - \psi_1\|_{H^{1}(\T)} \nonumber \\
\|f_{(\psi_2,\ga_0)} - f_{(\psi_2,\ga_1)}\|_{L^{2}(\T)} &\leq \mathfrak{C}^{*}\big(\ga_0,\ga_1,\ga_2\big)\|\psi_0 - \psi_1\|_{H^{1}(\T)} \\
\|f_{(\psi_0,\ga_0)} - f_{(\psi_1,\ga_1)}\|_{L^{2}(\T)} &\leq \mathfrak{C}^{*}\big(\ga_0,\ga_1\big)\|\psi_0 - \psi_1\|_{H^{1}(\T)}\nonumber
\end{align}
as the $\psi_i$ have constant speed. Lemma \ref{lem:chofvar} allows us to replace $\|\psi_0 - \psi_1\|_{H^{1}(\T)}$ with $\|\ga_0 - \ga_1\|_{H^{1}(\T)}$ by modifying the functions $\mathfrak{C}^{*}$ appropriately. We may then simply appeal to a change of variables, the $H^{1}(\T)$ bound \eqref{eq:h1boundtmp}, the triangle inequality and lemma \ref{lem:chofvar} to find that the appropriate variants
\begin{align}\label{eq:l2liptmp}
\|f_{(\ga_0,\ga_2)} - f_{(\ga_1,\ga_2)}\|_{L^{2}(\T)} &\leq \mathfrak{C}^{*}\big(\ga_0\big)\|f_{(\psi_0,\ga_2)} - f_{(\psi_1,\ga_2)}\circ \eta_1 \circ \xi_0\|_{L^{2}(\T)} \leq \mathfrak{C}^{*}\big(\ga_0,\ga_1,\ga_2\big)\|\ga_0 -\ga_1\|_{H^{1}(\T)},\nonumber \\
\|f_{(\ga_2,\ga_0)} - f_{(\ga_2,\ga_1)}\|_{L^{2}(\T)} &\leq \mathfrak{C}^{*}\big(\ga_2\big)\|f_{(\psi_2,\ga_0)} - f_{(\psi_2,\ga_1)}\|_{L^{2}(\T)} \leq  \mathfrak{C}^{*}\big(\ga_0,\ga_1,\ga_2\big)\|\ga_0 - \ga_1\|_{H^{1}(\T)} \\
\|f_{(\ga_0,\ga_0)} - f_{(\ga_1,\ga_1)}\|_{L^{2}(\T)} &\leq \mathfrak{C}^{*}\big(\ga_0\big)\|f_{(\psi_0,\ga_0)} - f_{(\psi_1,\ga_1)}\circ \eta_1 \circ \xi_0\|_{L^{2}(\T)} \leq \mathfrak{C}^{*}\big(\ga_0,\ga_1\big)\|\ga_0 -\ga_1\|_{H^{1}(\T)}, \nonumber
\end{align}
of \eqref{eq:dummybounds} hold for non-canonical representations as well.

\section{Local Well-Posedness and Regularity}\label{sec:Wellposed}

We now turn our attention toward our earlier claim, i.e. that nonlocal estimates of the form (\ref{eq:dummybounds},\ref{eq:l2liptmp}) are sufficient for a well-behaved dynamic. We base our analysis on the shape dynamic \eqref{eq:shapedynamic} and rely upon proposition \ref{prop:equivdyn} to infer local well-posedness for the alternative formulations of the dynamic. Our approach uses standard techniques; the main novelty is realizing that (\ref{eq:dummybounds},\ref{eq:l2liptmp}) are enough. We therefore proceed in somewhat a cursory fashion and only dwell on those aspects unique to the present context. The roadmap is a familiar one --- we show existence in the class of mild solutions, and then proceed to establish regularity and uniqueness results.

Recall the shape evolution equations \eqref{eq:shapedynamic}
\begin{align*}
\partial_{t} \phi_i &= \veps_i\left( \ddot \phi_i +  \po\left(a_i\right)\dot \phi_i +  F_i \dot \phi^{\perp}_i + \mu_{a_i}\phi_i\right), \qquad
a_{i} := |\ddot \phi_i|^2 + \langle \ddot \phi_i,\dot \phi^{\perp}_i\rangle F_i, \qquad
\rd_t \veps_i = 2\veps^2_i \mu_{a_i}, \nonumber \\
F_i(x,t) &= \veps^{-\frac12}_i(t)\left( c_i + \sum^{m}_{j=1} \int_{\T} g_{ij}\left(\frac12|\psi_i(x,t)-\psi_j(y,t)|^2\right)|\dot \psi_j(y,t)| \, \rd y \right), \qquad \quad \,\psi_i = \veps^{-\frac12}_i \phi_i,
\end{align*}
and let $\Phi(t) = (\phi_1(\cdot,t),\ldots,\phi_m(\cdot,t) \big)$ denote the $m$-tuple of interfacial shapes at any instant in time. We shall use the notation
$$
\mathcal{H}^{s}_{m} := H^{s}(\T) \times \cdots \times H^{s}(\T) 
$$
to denote the set of all such $m$-tuples with $H^{s}(\T)$ regularity in each component, and endow $C^{s}_m(T) := C([0,T];\mathcal{H}^{s}_m)$ with the max-norm
$$
\| \Phi \|_{C^{s}_{m}(T)} := \max \left\{ \|\phi_1\|_{C([0,T];H^{s}(\T))} , \ldots ,   \|\phi_m\|_{C([0,T];H^{s}(\T))} \right\}
$$
across components. Similarly, let $\bveps(t) := (\veps_1(t),\ldots,\veps_m(t))$ denote the $m$-tuple of (inverse squared) interfacial lengths at any instant in time, which lies in $\R^{m}_{+}$ endowed once again
$$
\| \bveps(t) \|_{\R^m} := \| \bveps(t)\|_{\ell^{\infty}(\R^{m})} \qquad \text{and} \qquad \|\bveps\|_{C_{m}(T)} := \max\left\{ \|\veps_1\|_{C([0,T];\R)},\ldots,\|\veps_m\|_{C([0,T];\R)} \right\}
$$
with the max-norm. For $T>0$ given, let $Q_{T} := \T \times [0,T]$ denote the corresponding parabolic cylinder, and
$$
C^{\infty}_{\rm p,0} := \left\{ \gamma \in C^{\infty}(Q_T) : \forall \, k \geq 0,\;\;\, \gamma^{(k)}(-\pi,t) = \gamma^{(k)}(\pi,t), \;\;\; \partial^{(k)}_{t} \gamma(x,T) = \partial^{(k)}_{t} \gamma(x,0) = 0 \right\}  
$$
of smooth, periodic functions vanishing on the boundary of the parabolic cylinder. We require at least $H^{2}(\T)$ regularity for the transport coefficient in \eqref{eq:shapedynamic} to make sense, and so we will search for solutions in this class. Specifically, if $\Phi \in C^{2}_{m}(T)$ and $\bveps \in C^{1}([0,T])$ satisfy the usual weak relation
\begin{align*}
0 &= \int_{Q_T}  \langle \phi_i(x,t), \gamma_{t}(x,t) + \veps_i(t) \gamma_{xx}(x,t) \rangle \, \rd x \rd t + \int_{Q_{T}} \veps_i(t)\langle f_i(x,t), \gamma(x,t)\rangle \, \rd x \rd t \\
f_i &:= \po\left( a_i \right)\dot \phi_i +  F_i \dot \phi^{\perp}_i + \mu_{a_i}\phi_i
\end{align*}
for all $\gamma \in C^{\infty}_{\rm p,0}(Q_T),$ the $\veps_i$ satisfy \eqref{eq:shapedynamic} in the classical sense and $(\Phi(t),\bveps(t))$ attain a given initial condition $(\Phi(0),\bveps(0)) = (\Phi_0,\bveps_0)$ in the strong sense then we refer to $(\Phi,\bveps)$ as a mild solution. It is easy to see that this definition is equivalent to solving directly for the Fourier coefficients
$$
\widehat \phi^{\prime}_{i,k}(t) = -\veps_i(t)k^{2}\widehat\phi_{i,k}(t) + \veps_i(t) \hat f_{i,k}(t) \qquad i \in [m], \; \, k \in \Z
$$
in the usual manner, and then verifying that the constructed solution 
$$
\phi_i(x,t) = \sum_{k \in \Z} \widehat \phi_{i,k}(t) \re^{ikx}
$$
has the requisite regularity properties.

Constructing such solutions will not prove too difficult, with one caveat: we need hypotheses guaranteeing that the nonlocal force
$$
F_i(x,t) := \veps^{-\frac12}_i(t)\left(c_i + \sum^{m}_{j=1} \int_{\T} g_{ij}\left(\frac1{2}|\psi_i(x,t) - \psi_j(x,t)|^{2}\right) |\dot \psi_j(y,t)| \, \rd y \right)\qquad \psi_i(x,t) := \veps^{-\frac12}_i(t)\phi_i(x,t)
$$
is well-behaved near a prescribed initial condition. Fix an initial condition $(\Phi_0,\bveps_0) \in \cH^{2}_{m} \times \R^{m}_{+}$ where each $\phi_i(x,0)$ has unit speed and let
$$
\psi_i(x,0) := \veps^{-\frac12}_i(0)\phi_i(x,0)
$$
denote the corresponding immersions that represent the initial interfaces.  Let $\mathcal{G} = \{ g_{ij}(s) : (i,j) \in [m]\times[m] \}$ denote a given collection of nonlocal kernels. If at least one of the conditions
\begin{enumerate}[{\rm I)}]
\item The kernel $g_{ij}$ satisfies {\rm H0)}, {\rm H1)}, or
\item The kernel $g_{ii}$ satisfies {\rm H0)}, {\rm H2)} and the interface $\psi_i(\cdot,0)$ has finite distortion, or
\item The kernel $g_{ij}$ satisfies {\rm H3)} and the interfaces $\psi_i(\cdot,0),\psi_j(\cdot,0)$ do not intersect,
\end{enumerate}
holds for all $(i,j) \in [m] \times [m]$ then we call the triple $(\Phi_0,\bveps_0,\cG)$ \emph{compatible}. For a given $\delta>0$ let
$$
B_{\delta}\big( (\Phi_0,\bveps_0) \big) := \left\{ (\Phi,\bveps) \in \cH^{2}_{m} \times \R^{m}_{+} : \|\Phi-\Phi_0\|_{\cH^{2}_{m}} + \|\bveps - \bveps_0\|_{\R^m} < \delta \right\}
$$
denote the corresponding $\cH^{2}_{m} \times \ell^{\infty}(\R^{m}_{+})$ ball near such a compatible initial condition. If we view the construction of the nonlocal force as a mapping
$$
(\Phi,\bveps) \in \cH^{2}_{m} \times \R^{m}_{+} \qquad \mapsto \qquad F_{i}[(\Phi,\bveps)] \in H^{1}(\T)
$$
then by propositions \ref{prop:nlest}, \ref{prop:nlest0} and \ref{prop:nlest2} (more specifically, by \eqref{eq:h1boundtmp}) the corresponding nonlocal forces obey a uniform $H^{1}(\T)$ bound
\begin{align*}
\|F_{i}[(\Phi,\bveps)]\|_{H^{1}(\T)} \leq \mathfrak{F}^{*}_{i}
\end{align*}
provided we choose $\delta = \delta( (\Phi_0,\bveps_0) ) > 0$ small enough. Similarly, by \eqref{eq:l2liptmp} and the triangle inequality we may infer a local $H^{1} \mapsto L^{2}$ Lipschitz property
\begin{align*}
\| F_i[(\hat \Phi,\hat \bveps)] - F_i[(\Phi,\bveps)]\|_{L^{2}(\T)} \leq \mathfrak{F}^{*}_i\left(  \|\Phi - \hat \Phi \|_{\cH^{1}_m(\T)} + \|\bveps - \hat \bveps\|_{\R^m} \right)
\end{align*}
within some $B_{\delta}\left( (\Phi_0,\bveps_0) \right)$ for $\delta>0$ small enough as well. A simple estimate based on the embeddings \eqref{eq:gns} shows that the overall forcing
$$
\po(a_i)\dot \phi_i + F_i[(\Phi,\bveps)] + \mu_{a_i}\phi_i \qquad \qquad a_{i} := |\ddot \phi_i|^2 + \langle \ddot \phi_i,F_i[(\Phi,\bveps)]\rangle
$$
defines a mapping $(\Phi,\bveps) \in \cH^{2}_{m} \times \R^{m}_{+} \longmapsto \cF_{i}[(\Phi,\bveps)] \in W^{1,1}(\T)$ that obeys a uniform bound of the form
\begin{align}\label{eq:loc1}
\|\cF_{i}[(\Phi,\bveps)]\|_{W^{1,1}(\T)} \leq \mathfrak{F}^{*}_{i}
\end{align}
locally near an initial condition. Analogously, the local $H^{1}(\T) \mapsto L^{2}(\T)$ Lipschitz property of the nonlocal force $F_i$ induces a weaker local $H^{2}(\T) \mapsto L^{2}(\T)$ Lipschitz property
\begin{align}\label{eq:loc2}
\| \cF_i[(\hat \Phi,\hat \bveps)] - \cF_i[(\Phi,\bveps)]\|_{L^{2}(\T)} \leq \mathfrak{F}^{*}_i\left(  \|\Phi - \hat \Phi \|_{\cH^{2}_m(\T)} + \|\bveps - \hat \bveps\|_{\R^m} \right)
\end{align}
for the overall force. Finally, the scalar coefficient $\mu_{a_i}$ defines a mapping $(\Phi,\bveps) \mapsto \mu_{i}[(\Phi,\bveps)] \in \R$ with
\begin{align}\label{eq:loc3}
| \mu_i[(\hat \Phi,\hat \bveps)] | \leq \mathfrak{F}^{*}_i \qquad \text{and}  \qquad |\mu_i[(\Phi,\bveps) - \mu_i[(\hat \Phi,\hat \bveps)]| \leq \mathfrak{F}^{*}_i\left(  \|\Phi - \hat \Phi \|_{\cH^{2}_m(\T)} + \|\bveps - \hat \bveps\|_{\R^m} \right)
\end{align}
holding locally as well. We may therefore view \eqref{eq:shapedynamic} as a specific instance of the more general family of evolutions
\begin{align}\label{eq:genev}
\partial_{t} \phi_i = \veps_i \ddot \phi_i + \veps_i \cF_i[(\Phi,\bveps)] \qquad \text{and} \qquad \rd_t \veps_i = 2\veps^2_i \mu_i[(\Phi,\bveps)],
\end{align}
driven by mappings $\cF_i,\mu_i$ obeying (\ref{eq:loc1},\ref{eq:loc2}) and \eqref{eq:loc3} locally (in $\cH^2_m \times \R^m$) near an initial condition.

We now show local existence, which we pursue via a standard fixed point argument. This approach requires a few preliminary lemmas that require some exposition. When $\cF_i[(\Phi,\bveps)]=0$ the evolution of $\veps_i$ in \eqref{eq:genev} is passive; a change of the temporal variable $t \mapsto \tau$ decouples the length evolution from the shape of the interface and defines a natural time-scale for the evolution. So for a single, uncoupled interface we may consider the shape evolution separately from its scale. When $\cF_i[(\Phi,\bveps)]\neq 0$ the evolution of the $m$ interfaces couple and we cannot separate out the scale evolution from the shape evolution as in the scalar ($m=1$) case. In this way, the coupling of the system makes existence a more tedious matter. Nevertheless, we may still pursue a similar idea to circumvent this difficulty. Consider the coupled dynamic
\begin{align*}
\phi_{t}(x,t) = \veps(t)\phi_{xx}(x,t) + \veps(t)g(x,t) \qquad \text{and} \qquad \dot \veps(t) = - \veps^2(t)\mu_{f}(t)
\end{align*}
with $f,g \in C([0,\infty),L^{1}(\T))$ a pair of a-priori given functions. Let $\fra(t),\frb(t)$ denote a solution to the system
\begin{align}\label{eq:timemaps}
\dot \fra(t) &= \frb(t) && \fra(0) = 0 \nonumber \\
\dot \frb(t) &= \frb(t)\mu_{f}\big( \fra(t)\big) && \frb(0) = \frac1{\veps_0}
\end{align}
of ordinary differential equations and $\frw( \fra(t) ) = t$ denote the corresponding inverse map. Then the function
$$
h(t) = \int^{\fra(t)}_{0} \veps(s) \, \rd s
$$
satisfies $h(0) = 0, \dot h(0) = 1$ as well as the differential equations
$$
\dot h(t) = \veps\big( \fra(t) \big) \dot \fra(t) \qquad \text{and} \qquad \ddot h(t) = \dot \veps\big( \fra(t) \big) \dot \fra^2(t) + \veps\big( \fra(t) \big)\ddot \fra(t) = \dot h(t)(1 - \dot h(t))\mu_{f}\big( \fra(t)\big),
$$
so that $\dot h(t) = 1$ and thus $h(t) = t$ is the identity. Setting $\eta(x,t) = \phi\big( x, \fra(t) \big)$ and differentiating in time
$$
\eta_{t}(x,t) = \veps\big( \fra(t) \big)\dot \fra(t) \left( \eta_{xx}(x,t) + g\big(x, \fra(t) \big) \right) = \eta_{xx}(x,t) + g\big(x,\fra(t))
$$
reveals that $\eta$ satisfies an ordinary forced heat equation. In other words, by constructing any solution to \eqref{eq:timemaps} and solving the forced heat equation we see
$$
\phi(x,t) := \eta\big( x,\frw(t) \big)
$$
gives the solution of the original problem as long as the inverse map $\frw$ exists. Composing the change-of-variable maps $\fra_2 \circ \frw_1$ from two different such evolutions then provides a way to obtain comparison estimates by accounting for changes in time-scale. We exploit these ideas in the following proposition, which forms the basis of a fixed point approach to existence.
\begin{proposition}
\label{prop:linearA}
For $i=1,2$ let $f_i\in C([0,\infty);L^{1}(\T))$ and $g_i \in L^{\infty}([0,\infty);W^{1,1}(\T))$ obey the global bounds
\begin{align*}
\sup_{t \geq 0}\;\, |\mu_{f_i}(t)| &\leq \frf^{*}, \qquad\, \text{and} \qquad \quad \,\sup_{t \geq 0}\;\,\left| \mu_{f_1}(t) - \mu_{f_2}(t)\right| \leq \frd \frf^{*} \\
\sup_{t \geq 0}\;\, \|g_i(\cdot,t)\|_{W^{1,1}(\T)} &\leq \frg^{*}, \qquad \text{and} \qquad \|g_1(\cdot,t) - g_2(\cdot,t)\|_{L^{2}(\T)} \leq \frd \frg^*,
\end{align*}
and let $\vth_0 \in H^{2}(\T)$ and $\veps_0 > 0$ denote arbitrary initial data. Then the following hold ---
\begin{enumerate}[{\rm i)}]
\item For any time $T < T_{*}$ with
$$
T_{*} := \frac{1}{ \veps_0 \frf^{*}},
$$
there exist unique mild solutions $\vth_i \in C( [0,T];H^{2}(\T) ), \veps_{i} \in C^{1}([0,T];\R)$ to the initial value problems
\begin{align*}
\partial_{t} \vth_{i} &= \veps_i(t)\ddot \vth_i + \veps_i(t)g_i &&\vth(x,0) = \vth_0(x) \in H^{2}(\T)\\
\dot \veps_i(t) &= - \veps^2_i(t)\mu_{f_i}(t) && \;\,\,\, \veps_i(0) = \veps_0 > 0,
\end{align*}
and the solution $\vth_i$ exists for as long as $\veps_i$ remains finite.
\item There exists a continuous, increasing function $\Gamma_{\vth_0} : \R^{+} \mapsto \R^{+},$ depending only on the initial datum $\vth_0,$ so that the properties
$$
\Gamma_{\vth_0}(0) = 0 \qquad \text{and} \qquad \Gamma_{\vth_0}(t) \nearrow \Gamma_{\vth_0}(\infty) = \| \vth_0 \|_{\dot{H}^{2}(\T)}
$$
hold.
\item The solutions $\vth_i,\veps_i$ obey the bounds
\begin{align*}
|\veps_i(t) - \veps_i(s)| & \leq \veps_0\left(1 - \frac{t}{T_{*}}\right)^{-2}\frac{t-s}{T_{*}}\\
\left| \mu_{\vth_i}(t) - \mu_{\vth_i}(s)\right| &\leq \frg^{*}\left( \zeta_i(t) - \zeta_i(s) \right),\qquad \quad \zeta_i(t) := \int^{t}_{0} \veps_i(s) \, \rd s,\\
\| \vth_i(\cdot,t) - \vth_i(\cdot,s) \|_{\dot{H}^{2}(\T)} &\leq \Gamma_{\vth_0}\left( \zeta_i(t) - \zeta_i(s) \right) + \frg^{*}\left( \zeta_i(t) - \zeta_i(s) \right)^{\frac14}
\end{align*}
for all pairs of times $0 \leq s < t$ at which the solutions exist.
\item If $g_1=g_2=0$ then the corresponding homogeneous solutions $\vth^{h}_{i}$ obey the difference bound
$$
\| \vth^{h}_1(\cdot,t) - \vth^{h}_2(\cdot,t)\|_{H^{2}(\T)} \leq  \frac{ \frd\frf^{*} }{4\frf^{*}}\left(1 - \frac{t}{T_*}\right)^{-1}\log^2\left(1-\frac{t}{T^*}\right)\|\vth_0\|_{\dot H^{2}(\T)}
$$
for as long as both exist. 
\item For any $0 < s < 1/2$ the differences $\xi_i := \vth_{i} - \vth^{h}_{i}$ obey the estimates
$$
\|\xi_i(\cdot,t)\|_{H^{2+s}(\T)} \leq \frg^{*}\big( \frc_{s} + \zeta_i(t) \big)
$$
for $0 < \frc_s < \infty$ a finite constant depending only on the modulus $s$ of regularity. Moreover, the estimate
$$
\|\xi_i(\cdot,t) - \xi_i(\cdot,s)\|_{H^{2}(\T)} \leq \frg^{*}\left( \big(\zeta_i(t) - \zeta_i(s)\big)^{\frac14} + \zeta_i(t) - \zeta_i(s) \right)
$$
holds for all pairs of times $0 \leq s < t$ at which the solutions exist.
\item The non-homogeneous solution map $(f_i,g_i) \mapsto \xi_i$ is H\"older continuous with respect to the $H^{2}(\T)$ topology, in the sense that the difference $\delta \xi = \xi_1 - \xi_2$ obeys the bound
\begin{align*}
&\|\delta \xi(\cdot,t) \|_{H^{2}(\T)} \leq \mathfrak{C}^{*}\left( \frac{T_*}{\frf^*(T_*-t)},\frf^* \right)\left[ \frd \frg^{*} \left( 1 + \log^{+}\left( \frac{\frg^* }{\frd \frg^{*}}\zeta^{\frac14}_1(t) \vee \zeta^{\frac14}_2(t) \right) \right) + \frac{\frd \frf^*}{\frf^*} \vee \left( \frac{\frd \frf^*}{\frf^*} \right)^{\frac14}\right]
\end{align*}
on the common interval $0\leq t < T_*$  where both solutions exist.
\end{enumerate}
\end{proposition}

With these preliminaries in hand, we now proceed to show local existence of \eqref{eq:genev} via the Leray-Schauder fixed-point theorem. Fix an initial condition $(\Phi_0,\bveps_0) \in \cH^{2}_{m} \times \R^{m}$ for which the family of $m$ mappings 
$$(\Phi,\bveps) \mapsto (\cF_{i}[(\Phi,\bveps)],\mu_i[(\Phi,\bveps)])$$
obey a local $W^{1,1}(\T)$ bound and $H^{1}(\T) \mapsto L^{2}(\T)$ Lipschitz property locally near this initial condition. In other words, there exists some $\delta>0$ and a constant $\mathfrak{F}>0$ so that the bounds
\begin{align*}
\|\cF_i[(\Phi,\bveps)]\|_{W^{1,1}(\T)} &\leq \mathfrak{F}^{*} \qquad \text{and} \qquad \|\cF_i[(\Phi,\bveps)] - \cF_i[(\hat \Phi,\hat \bveps)] \|_{L^2(\T)} \leq \mathfrak{F}^{*}\|(\Phi,\bveps) - (\hat \Phi,\hat \bveps)\|_{ \cH^{2}_{m} \times \R^{m} } \\
|\mu_i[(\Phi,\bveps)]| &\leq \mathfrak{F}^{*} \qquad \text{and} \qquad \quad \qquad \, |\mu_i[(\Phi,\bveps)] - \mu_i[(\hat \Phi,\hat \bveps)] |\leq \mathfrak{F}^{*}\|(\Phi,\bveps) - (\hat \Phi,\hat \bveps)\|_{ \cH^{2}_{m} \times \R^{m} }
\end{align*}
hold on $B_{\delta}\big( (\Phi_0,\bveps_0)\big)$ for some $\delta>0$ small enough. Let $\frX := C\left( [0,\infty); \cH^{2}_{m} \times \R^{m} \right)$ denote the Banach space of temporal trajectories $(\Phi(t),\bveps(t))$ and endow it
$$
\left\| (\Phi,\bveps) \right\|_{\frX} := \sup_{t \geq 0} \;\, \left\| (\Phi(t),\bveps(t)) \right\|_{\cH^{2}_{m} \times \R^{m}}
$$
with the sup-norm. Define $\mathcal{K}_{\delta} \subset \frX$ as the set of trajectories 
$$
\mathcal{K}_{\delta} := \left\{ (\Phi,\bveps) \in \frX : \sup_{t\geq 0} \;\, \left\| (\Phi(t),\bveps(t)) - (\Phi_0,\bveps_0) \right\|_{\cH^{2}_{m} \times \R^{m}} \leq \delta \right\}
$$
that remain in the prescribed $\delta$-neighborhood of the initial condition. Note that for each $i \in [m]$ proposition \ref{prop:linearA} part ii) furnishes a continuous, increasing function $\Gamma_i : \mathbb{R}^{+} \mapsto \R^{+}$ obeying the properties
$$
\Gamma_i(0) = 0 \qquad \text{and} \qquad \Gamma_i(t) \nearrow \Gamma_i(\infty) = \| (\Phi_0)_i \|_{\dot H^{2}(\T) },
$$
and so the function $\Gamma_{\Phi_0}(t) := \max\{ \Gamma_1(t),\ldots,\Gamma_{m}(t) \}$ obeys
$$
\Gamma_{\Phi_0}(0) \qquad \text{and} \qquad \Gamma_{\Phi_{0}}(t) \nearrow \Gamma_{\Phi_0}(\infty) = \|\Phi_0\|_{\dot \cH^{2}_{m}}
$$
and depends only on the initial data. If we set $H(x) := \max\{x,x^2\}$ this function induces a corresponding modulus of continuity
$$
\omega(t,s) := \Gamma_{\Phi_0}\left( 2\|\bveps_0\|_{\R^{m}}|t-s| \right) + \mathfrak{F}^{*}\left[ H\left( 2\|\bveps_0\|_{\R^{m}}\right)|t-s| + \left( 2\|\bveps_0\|_{\R^{m}}|t-s| \right)^{\frac14} \right]
$$
for temporal trajectories that depends only on $\mathfrak{F}^{*}$ and the initial data. Define $\mathcal{K}_{\re} \subset \frX$ as those trajectories
$$
\mathcal{K}_{\re} := \left\{ (\Phi,\bveps) \in \frX : \sup_{t \neq s} \;\, \frac{ \|(\Phi(t),\bveps(t)) - (\Phi(s),\bveps(s))\|_{\cH^{2}_{m} \times \R^{m}} }{\omega(t,s)} \leq \sqrt{2} \right\}
$$
for which $\omega$ provides such a temporal modulus of continuity. Finally, given some finite time horizon $T>0$ let $\mathcal{K}_T \subset \frX$ denote the subset
$$
\mathcal{K}_{T} = \left\{ (\Phi,\bveps) \in \frX : \forall t \geq T, \;\,(\Phi(t),\bveps(t)) = (\Phi(T),\bveps(T)) \right\} 
$$
of trajectories that remain constant after this fixed, finite time. All three sets $\mathcal{K}_{\delta},\mathcal{K}_{\re},\mathcal{K}_T$ are convex and closed with respect to convergence in $\frX,$ and so their intersection
$$
\mathcal{K} := \mathcal{K}_{\delta} \cap \mathcal{K}_{\re} \cap \mathcal{K}_{T}
$$
defines a closed, convex subset. Moreover, $\mathcal{K}$ is non-empty since it contains the constant trajectory based at the initial condition. We may then try to construct a mapping $\mathcal{A}: \mathcal{K} \mapsto \mathcal{K}$ with a fixed point whose restriction to $[0,T]$ defines a mild solution.

Proposition \ref{prop:linearA} makes this task feasible. We may first appeal to the boundedness and Lipschitz assumptions on $\mu_i,\cF_i$ to conclude that any given trajectory $(\Phi,\bveps) \in \mathcal{K}$ induces forcing functions
$$
\nu_i(t) := \mu_{i}[(\Phi(t),\bveps(t))]\in \R \qquad \text{and} \qquad g_i(t) := \cF_i[(\Phi(t),\bveps(t))] \in W^{1,1}(\T)
$$
that obey the bounds
\begin{align*}
\sup_{t \geq 0} \;\, |\nu_i(t)| &\leq \mathfrak{F}^{*} \qquad \qquad \qquad\quad\quad\;\;\; \sup_{t \geq 0} \;\, \|g_i(t)\|_{W^{1,1}(\T)} \leq \mathfrak{F}^{*} \\
|\nu_i(t)-\nu_i(s)| &\leq \mathfrak{F}^{*}\omega(t,s) \qquad \qquad \qquad \|g_i(t)-g_i(s)\|_{L^{2}(\T)} \leq \mathfrak{F}^{*}\,\omega(t,s)
\end{align*}
globally in time. Proposition \ref{prop:linearA} part i) furnishes a unique mild solution $(\Psi(t),\bsigma(t))$ to 
\begin{align*}
\partial_t \psi_i = \sigma_i \ddot \psi_i + \sigma_i g_i, \qquad \rd_t \sigma_i = -2\sigma^2_i\nu_i \qquad \text{and} \qquad (\Psi(0),\bsigma(0)) = (\Phi_0,\bveps_0)
\end{align*}
existing on $[0,T]$ provided $T>0$ satisfies
$$
0 < T < T_{*} := \min_{i \in [m]} \;\, T^{i}_{*} \qquad \qquad T^{i}_{*} := \frac1{(\bveps_0)_i \mathfrak{F}^{*}},
$$
and so for any such $T>0$ a fixed point of the mapping
\begin{align*}
\mathcal{A}[(\Phi,\bveps)](t)  := \begin{cases} \big( \Psi(t),\bsigma(t)\big) &\quad \text{if} \quad 0 \leq t \leq T \\
\big( \Psi(T) , \bsigma(T) \big) &\quad \text{else}
\end{cases}
\end{align*}
would, in fact, yield a mild solution locally in time. For convenience, we first recall a standard result
\begin{theorem}[Leray-Schauder]
Assume $\mathcal{K} \subset \frX$ is a non-empty, closed and convex subset of a Banach space $\frX$. If $\mathcal{A}: \mathcal{K} \mapsto \mathcal{K}$ is continuous and $\mathcal{A}(\mathcal{K})$ is precompact in $\mathcal{K}$ then $\mathcal{A}$ has a fixed point.
\end{theorem}
\noindent and then set upon the somewhat lengthy task of verifying its hypotheses.

That $\mathrm{im}(\mathcal{A})\subset\mathcal{K}_{T}$ follows directly from the definition. By proposition \ref{prop:linearA} part iii) the components $\psi_i(t),\sigma_i(t)$ of the solution $(\Phi(t),\bsigma(t))$ obey the bounds
\begin{align*}
|\sigma_i(t) - \sigma_i(s)| &\leq (\bveps_0)_{i}\left(1 - \frac{t}{T^{i}_{*}}\right)^{-2}\frac{|t-s|}{T^{i}_*} \qquad \qquad \zeta_i(t) := \int^{t}_{0} \sigma_i(s) \, \rd s\\
2^{-\frac12}\|\psi_i(t) - \psi_i(s)\|_{H^{2}(\T)} &\leq \Gamma_i\left( \zeta_i(t)-\zeta_i(s) \right) + \mathfrak{F}^{*}\left( \zeta_i(t)-\zeta_i(s) \right)^{\frac14} + \mathfrak{F}^{*}(\zeta_i(t)-\zeta_i(s)) 
\end{align*}
on $[0,T_*)$ as well. As a consequence, for any $T>0$ satisfying
$$
T \leq T_{0} :=  \min\left\{ 1, \min_{i \in [m]} \; \frac{\delta}{(\bveps_0)_i} \right\}\frac{T_{*}}{4} 
$$
the uniform upper bounds
$$
\sup_{t \in [0,T]}\;\,\|\bsigma(t) - \bveps_0\|_{\R^{m}} \leq \delta \qquad \text{and} \qquad \sup_{t \in [0,T]} \;\,\|\bsigma(t)\|_{\R^{m}} \leq 2\|\bveps_0\|_{\R^{m}}
$$
necessarily hold. But then the uniform bounds
\begin{align*}
|\zeta_i(t)-\zeta_i(s)| \leq 2\|\bveps_0\|_{\R^{n}}|t-s| \qquad \text{and} \qquad \|\bsigma(t) - \bsigma(s)\|_{\R^{m}} \leq \mathfrak{F}^{*}\left( 2\|\bveps_0\|_{\R^{m}} \right)^2|t-s|
\end{align*}
hold for all $s,t \in [0,T]$ as well, and so
\begin{align*}
2^{-\frac12}\|\psi_i(t) - \psi_i(s)\|_{H^{2}(\T)} &\leq \omega(t,s)\\
\|\bsigma(t) - \bsigma(s)\|_{\R^{m}} &\leq \omega(t,s)
\end{align*}
for all $s,t \in [0,T]$ by definition of the modulus of continuity. As $\alpha(t) := \omega(t,0)$ is strictly increasing, provided
$$
0 < T \leq \min\left\{ T_0 , \alpha^{-1}\left( \frac{\delta}{\sqrt{2}}\right) \right\} := T\left( \mathfrak{F}^{*},\delta,\Phi_0, \bveps_0\right)
$$
the uniform bound
$$
\sup_{t \in [0,T]} \;\, \|\Psi(t) - \Phi_0\|_{\cH^{2}_{m}} \leq \delta
$$
is valid. The mapping $\mathcal{A}: \mathcal{K} \mapsto \frX$ therefore maps $\mathcal{K}$ to itself whenever $T = T\big( \mathfrak{F}_{*},\delta,\Phi_0,\bveps_0 \big) > 0$ is small enough.

It suffices to show that $\mathcal{A}$ is continuous in $\frX$ and that $\mathrm{im}(\mathcal{A}) \subset \mathcal{K}_{T}$ is pre-compact. To show continuity, fix $(\Phi_j,\bveps_j) \in \mathcal{K}_{T}$ for $j=1,2$ arbitrary, and define the corresponding functions
$$
\nu_{i,j}(t) := \mu_{i}[(\Phi_j(t),\bveps_j(t))]\in \R \qquad \text{and} \qquad g_{i,j}(t) := \cF_i[(\Phi_j(t),\bveps_j(t))] \in W^{1,1}(\T)
$$
as before. For $i \in [m]$ fixed let $\nu_j = \nu_{i,j}, g_{j} = g_{i,j}$ and $\sigma_j,\psi_j$ the corresponding mild solutions that define the mapping. Then $\sigma_1(0) = \sigma_2(0) := \sigma(0)$ and so the bound
$$
\left| \sigma_{1}(t) - \sigma_{2}(t)\right| \leq \left(\frac{ \sigma(0) }{1 - \frac{t}{T_{*}}}\right)^{2}T \|\nu_1 - \nu_2\|_{C([0,T])}
$$
follows by direct integration. The Lipschitz assumption on the mappings $\mu_i$ then reveals that the bound
$$
\|\sigma_1 - \sigma_2\|_{C([0,T])} \leq \mathfrak{C}^{*}(T,T_*,\sigma(0))\mathfrak{F}^{*} \|(\Phi_1,\bveps_1) - (\Phi_2,\bveps_2)\|_{\frX}
$$
holds for $\mathfrak{C}^{*}(T,T_*,\sigma(0))$ some continuous function. A similar bound
$$
\|\zeta_1 - \zeta_2\|_{C([0,T])} \leq \mathfrak{C}^{*}(T,T_*)\mathfrak{F}^{*} \|(\Phi_1,\bveps_1) - (\Phi_2,\bveps_2)\|_{\frX} \qquad \zeta_i(t) := \int^{t}_{0} \sigma_i(s) \, \rd s
$$
for the $\zeta_i$ then follows by integrating in time. Now decompose $\psi_j = \psi^{h}_{j} + \psi^{nh}_{j}$ into its homogeneous and non-homogeneous parts. Proposition \ref{prop:linearA} part iv) yields the bound
$$
\|\psi^{h}_{1}(\cdot,t) - \psi^{h}_{2}(\cdot,t)\|_{H^{2}(\T)} \leq \mathfrak{C}^{*}\left(T,T_{*},\Phi_0\right)\frac{ \|\nu_1 - \nu_2\|_{C([0,\infty))}}{\mathfrak{F}^{*}} \leq \mathfrak{C}^{*}\left(T,T_{*},\Phi_0,\mathfrak{F}^{*}\right)\|(\Phi_1,\bveps_1) - (\Phi_2,\bveps_2)\|_{\frX}
$$
by appealing to the Lipschitz property of the $\mu_i$ once again. Similarly, proposition \ref{prop:linearA} part vi) yields an analogous bound
\begin{align*}
\|\psi^{h}_{1}(\cdot,t) - \psi^{h}_{2}(\cdot,t)\|_{H^{2}(\T)} &\leq \mathfrak{C}^{*}(T,T_*,\mathfrak{F}^*)L_{T}\left(  \|g_1 - g_2\|_{C([0,\infty);L^{2}(\T)} , \|\nu_1 - \nu_2\|_{C([0,\infty)} , \mathfrak{F}^{*} , \|\bveps_0\|_{\R^n} \right) \\
L_{T}(x,y,p,q) &:= x \left( 1 + \log^{+}\left( \frac{2pqT}{x} \right)  \right) + \frac{y}{p} \vee \left( \frac{y}{p} \right)^{\frac14}
\end{align*}
for the non-homogeneous part of the solution. The Lipschitz hypothesis on $\mathcal{F}_i,\mu_i$ gives
\begin{align*}
 \|g_1 - g_2\|_{C([0,\infty);L^{2}(\T))} &\leq \mathfrak{F}^{*}\|(\Phi_1,\bveps_1) - (\Phi_2,\bveps_2)\|_{\frX} \\
 \|\nu_1 - \nu_2\|_{C([0,\infty))} &\leq \mathfrak{F}^{*}\|(\Phi_1,\bveps_1) - (\Phi_2,\bveps_2)\|_{\frX},
\end{align*}
and so by the triangle inequality an overall bound
$$
\sup_{0 \leq t \leq T} \;\, \left\| (\Psi_1(t),\bsigma_1(t)) - (\Psi_2(t),\bsigma_2(t)) \right\|_{\cH^{2}_{m} \times \R^{m}} = O\left( \|(\Phi_1,\bveps_1) - (\Phi_2,\bveps_2)\|^{\frac14}_{\frX} \right)
$$
whenever $\|(\Phi_1,\bveps_1) - (\Phi_2,\bveps_2)\|_{\frX}$ is small enough. By construction of the mapping this shows that
$$
\left\| \mathcal{A}[(\Phi_1,\bveps_1)] - \mathcal{A}[(\Phi_2,\bveps_2)]\right\|_{\frX} = O\left( \|(\Phi_1,\bveps_1) - (\Phi_2,\bveps_2)\|^{\frac14}_{\frX} \right)
$$
whenever $\|(\Phi_1,\bveps_1) - (\Phi_2,\bveps_2)\|_{\frX}$ is small enough as well, so in particular $\mathcal{A} : \mathcal{K} \mapsto \mathcal{K}$ is continuous.

It remains to show that $\mathcal{A}: \mathcal{K}_{T} \subset \frX \mapsto \mathcal{K}_{T}$ has pre-compact image. Take $(\Psi_{\ell},\bsigma_{\ell}):= \mathcal{A}[(\Phi_{\ell},\bveps_{\ell})] \in \mathrm{im}(\mathcal{A})$ any sequence; it suffices to show the existence of an $\frX$ convergent sub-sequence. For $0 \leq s \leq t \leq T$ the $\bsigma_{\ell}$ are equicontinuous
$$
\|\bsigma_{\ell}(t) - \bsigma_{\ell}(s)\|_{\R^{m}} \leq \|\bsigma(0)\|_{\R^{m}}\mathfrak{C}^{*}(T,T_{*})(t-s)
$$
and uniformly bounded; in particular, they admit a subsequence (still denoted by $\ell$) and a limit $\bsigma_{\infty} \in C([0,T];\R^{m})$ so that the uniform convergence
$$
\|\bsigma_{\ell} - \bsigma_{\infty}\|_{C_m(T)} \to 0
$$
holds, and so by defining $\bsigma_{\infty}(t) = \bsigma(T)$ as a constant for $t>T$ the global in time convergence
\begin{align}\label{eq:compact1}
\|\bsigma_{\ell} - \bsigma_{\infty}\|_{C_m(\infty)} \to 0
\end{align}
follows by construction of the mapping. By the Lipschitz hypothesis and the definition of $\mathcal{K}_{T}$ the functions $\nu_{i,\ell}(t) := \mu_{i}[ (\Phi_{\ell}(t),\bveps_{\ell}(t)) ]$ obey the uniform bound $|\nu_{i,\ell}(t)| \leq \mathfrak{F}^{*}$ and the equicontinuity property
$$
|\nu_{i,\ell}(t) -  \nu_{i,\ell}(s)| \leq \mathfrak{F}^{*}\|(\Phi(t),\bveps(t)) - (\Phi(s),\bveps(s))\|_{\cH^{2}_{m} \times \R^{m}} \leq \sqrt{2} \mathfrak{F}^{*}\omega(t,s),
$$
and so there exists a further subsequence (still denoted by $\ell$) and a limit $\nu_{i,\infty}$ so that the uniform convergence
$$
\|\nu_{i,\ell}(t) - \nu_{i,\infty}(t)\|_{C([0,T])} \to 0
$$
holds. In particular, by construction of the mapping $\mathcal{A}$ and the set $\mathcal{K}_{T}$ the subsequence $\nu_{i,\ell}$ is Cauchy
$$
\|\nu_{i,\ell} - \nu_{i,k}\|_{C_{m}(\infty)} \to 0 \qquad \text{as} \qquad k,\ell \to \infty
$$
globally in time. By proposition \ref{prop:linearA} part iv) the homogeneous portions $\Psi^{h}_{k},\Psi^{h}_{\ell}$ of the subsequence $\Psi_k = \Psi^{h}_k + \Psi^{nh}_{k}$ obey
$$
\sup_{t \in [0,T] } \;\, \| \Psi^{h}_{k}(\cdot,t) - \Psi^{h}_{\ell}(\cdot,t)\|_{\cH^{2}_{m}} \leq \mathfrak{C}^{*}(T,T^{*},\Phi_0)\frac{\|\nu_{i,\ell} - \nu_{i,k}\|_{C([0,\infty);\R^{m})}}{\mathfrak{F}^{*}}
$$
and are therefore Cauchy. By completeness there exists a limit $\Psi^{h}_{\infty} \in C([0,T],\cH^{2}_{m})$ so that convergence
\begin{align}\label{eq:compact2}
\|\Psi^{h}_{\ell} - \Psi^{h}_{\infty}\|_{C^{2}_{m}(T)} \to 0
\end{align}
holds locally in time. Finally, for any $0 < s < 1/2$ proposition \ref{prop:linearA} part v) shows that the non-homogeneous parts $\Psi^{nh}_{\ell}$ of $\Psi_{\ell}$ obey
\begin{align*}
\|\Psi^{nh}_{\ell}(t)\|_{\cH^{2+s}_{m}} &\leq \mathfrak{F}^{*}\big( \frc_{s} + 2\|\bveps_0\|_{\R^{m}}T \big) \\
\|\Psi^{nh}_{\ell}(t) - \Psi^{nh}_{\ell}(s)\|_{\cH^{2}_{m}} &\leq \mathfrak{F}^{*}\left( \left(2\|\bveps_0\|_{\R^{m}}(t-s)\right)^{\frac14} + 2\|\bveps_0\|_{\R^{m}}(t-s) \right)
\end{align*}
for all $0 \leq s \leq t \leq T$ and $\frc_s>0$ some finite constant. By compactness of the embedding $\cH^{2+s}_{m} \subset\subset \cH^{2}_{m}$ when $s>0$ there exists, by passing to a further subsequence if necessary, a limit $\Psi^{nh}_{\infty} \in C^{2}_{m}(T)$ so that the uniform convergence
\begin{align}\label{eq:compact3}
\|\Psi^{nh}_{\ell} - \Psi^{nh}_{\infty}\|_{C^{2}_{m}(T)} \to 0
\end{align}
holds. Define a limit function $\Psi_{\infty}(t) = \Psi^{h}(t) + \Psi^{nh}(t)$ if $0 \leq t \leq T$ and $\Psi_{\infty}(t) = \Psi_{\infty}(T)$ otherwise. The definition of the mapping $\mathcal{A}$ and the limit function $\Psi_{\infty}$ combine with the convergences (\ref{eq:compact2},\ref{eq:compact3}) to give convergence
$$
\|\Psi_{\ell} - \Psi_{\infty}\|_{C^{2}_{m}(\infty)} = \|\Psi_{\ell} - \Psi_{\infty}\|_{C^{2}_{m}(T)} \leq \|\Psi^{h}_{\ell} - \Psi^{h}_{\infty}\|_{C^{2}_{m}(T)} + \|\Psi^{nh}_{\ell} - \Psi^{nh}_{\infty}\|_{C^{2}_{m}(T)}
$$
globally in time. When combined with \eqref{eq:compact1} this yields an $\frX$ convergent subsequence
$$
\| (\Psi_{\ell},\bsigma_{\ell}) - (\Psi_{\infty},\bsigma_{\infty})\|_{\frX} \to 0
$$
as desired. Thus $\mathcal{A}$ has a fixed point that, by construction, gives a mild solution on $[0,T]$ for $T>0$ small enough. All together, we have shown ---
\begin{theorem}[Local Existence]\label{thm:locmild}
Fix constants $\delta,\mathfrak{F}^{*} > 0$ and an initial datum $(\Phi_0,\bveps_0) \in \cH^{2}_{m} \times \R^{m}_{+}$ arbitrarily. Assume that the mappings $\cF_i : \cH^2_m \times \R^m \mapsto W^{1,1}(\T)$ and $\mu_i:  \cH^2_m \times \R^m \mapsto \R$ obey the uniform estimates
\begin{align*}
\|\cF_i[(\Phi,\bveps)]\|_{W^{1,1}(\T)} &\leq \mathfrak{F}^{*} \qquad \text{and} \qquad \|\cF_i[(\Phi,\bveps)] - \cF_i[(\hat \Phi,\hat \bveps)] \|_{L^2(\T)} \leq \mathfrak{F}^{*}\|(\Phi,\bveps) - (\hat \Phi,\hat \bveps)\|_{ \cH^{2}_{m} \times \R^{m} } \\
|\mu_i[(\Phi,\bveps)]| &\leq \mathfrak{F}^{*} \qquad \text{and} \qquad \quad \qquad \, |\mu_i[(\Phi,\bveps)] - \mu_i[(\hat \Phi,\hat \bveps)] |\leq \mathfrak{F}^{*}\|(\Phi,\bveps) - (\hat \Phi,\hat \bveps)\|_{ \cH^{2}_{m} \times \R^{m} }
\end{align*}
locally in $B_{\delta}\big( (\Phi_0,\bveps_0) \big)$ with respect to the $\cH^{2}_{m} \times \R^{m}$ topology. Then there exists $T = T\big(\mathfrak{F}_{*},\delta,\Phi_0,\bveps_0\big)>0$ so that the coupled system
$$
\left.\begin{aligned}
\partial_{t} \phi_i(t) &= \veps_i \ddot \phi_i + \veps_i \cF_i[(\Phi,\bveps)]  &&\,\,\rd_{t} \veps_i = -2\veps^2_i \mu_i[(\Phi,\bveps)] \quad\\
\phi_i(0) &= (\Phi_0)_{i} && \veps_i(0) = (\bveps_0)_i \quad
\end{aligned}\right\rbrace \quad \text{for all} \quad i \in [m]
$$
has a mild solution $(\Phi(t),\bveps(t))$ on $[0,T]$ that remains in $B_{\delta}\big( (\Phi_0,\bveps_0) \big)$.
\end{theorem}

\subsection*{Regularity} Local existence for the shape dynamic \eqref{eq:shapedynamic} is an immediate by-product of this local existence result. Specifically, a $C^{2}_{m}(T)$ mild solution $(\Phi,\bveps)$ exists whenever the selection $(\Phi_0,\bveps_0,\cG)$ of initial data $(\Phi_0,\bveps_0)$ and family of kernels $\cG$ define a compatible triple. Higher regularity of these solutions, as well as the existence of classical solutions, follows in the usual way from a few additional estimates. 

We begin by recalling the structure of the mappings $\mathcal{F}_i : \cH^{2}_{m} \times \R^{m}_{+} \mapsto W^{1,1}(\T)$ for the dynamic \eqref{eq:shapedynamic}, which take the form
\begin{align*}
\mathcal{F}_{i}\left[ (\Phi,\bveps) \right] &:= \po\big( a_i \big) \dot \phi_i + F_i\left[ (\Phi,\bveps) \right]\dot \phi^{\perp}_{i} + \mu_{a_i}\phi_i \\
 a_i &:= |\ddot \phi_i|^2 + \langle \ddot \phi_i,\dot \phi^{\perp}_{i} \rangle F_i\left[ (\Phi,\bveps) \right],
\end{align*}
and that the non-local forcings $F_i : \cH^{2}_{m} \times \R^{m}_{+}\mapsto H^{1}(\T)$ obey a uniform $H^{1}(\T)$ estimate and $H^{1}(\T) \mapsto L^{2}(\T)$ Lipschitz property
\begin{align}\label{eq:regbound1}
\left.\begin{aligned}
\| F_i\left[(\Phi,\bveps)\right] \|_{H^{1}(\T)} &\leq \mathfrak{F}^{*} \\
\| F_i[(\Psi, \bsigma)] - F_i[(\Phi,\bveps)]\|_{L^{2}(\T)} &\leq \mathfrak{F}^{*}\left(  \|\Psi - \Phi \|_{\cH^{1}_m(\T)} + \|\bsigma -  \bveps\|_{\R^m} \right) \quad
\end{aligned}\right\rbrace 
\quad \text{on} \quad B_{\delta}\left( (\Phi_0,\bveps) \right)
\end{align}
locally near a compatible initial condition. By taking $\mathfrak{F}^{*}$ larger if necessary, we may assume that the composite mapping $\mathcal{F}_i$ obeys the uniform $W^{1,1}(\T)$ bound
\begin{align}\label{eq:regbound2}
\| \mathcal{F}_i\left[(\Phi,\bveps)\right] \|_{W^{1,1}(\T)} \leq \mathfrak{F}^{*} 
\end{align}
locally near a compatible initial condition as well. By theorem \ref{thm:locmild} we may assume that for $t \in [0,T]$ the trajectory $(\Phi(t),\bveps(t))$ remains in the neighborhood  $B_{\delta}\left( (\Phi_0,\bveps) \right)$ where both \eqref{eq:regbound1} and \eqref{eq:regbound2} apply, and that the uniform bounds
$$
0 < \bveps_{*} \leq \veps_{i}(t) \leq \bveps^{*}
$$
hold as well. We may therefore decompose the mappings
\begin{align*}
\mathcal{F}_{i}\left[ (\Phi,\bveps) \right] &\;= \cT_i\left[ (\Phi,\bveps) \right] + \cR_{i}\left[ (\Phi,\bveps) \right] \\
\mathcal{R}_{i}\left[ (\Phi,\bveps) \right] &:= F_i\left[ (\Phi,\bveps) \right]\dot \phi^{\perp}_{i} + \mu_{a_i}\phi_i
\end{align*}
into a $W^{1,1}(\T)$ transport part $\cT_i := \po\big(a_i\big)\dot \phi^{\perp}_i$ and an $H^{1}(\T)$ reaction part; the embedding $H^{1}(\T) \subset L^{\infty}(\T)$ furnishes a uniform bound
\begin{align}\label{eq:reacH1}
\sup_{0 \leq t \leq T} \;\,\| \cR_i\left[ (\Phi(t),\bveps(t)) \right] \|_{H^{1}(\T)} \leq \mathfrak{C}^{*}\left( \|\Phi\|_{C^{2}_{m}(T)},\mathfrak{F}^{*} \right)
\end{align}
along any mild solution trajectory. 

We shall exploit this structure by using (the proof of) proposition \ref{prop:linearA} as a starting point. Given some function $\Phi \in \cH^{0}_{m}$ and an exponent $p \geq 0$ we shall use the semi-norms
\begin{align*}
\| \Phi \|_{\dot \cV^{p}_{m} } := \max_{i \in [m]} \left\{ \sup_{k \in \Z_0} |k|^{p} \left| \widehat \phi_{i,k} \right| \right\} \qquad \text{and} \qquad \| \Phi \|_{\dot \cV^{<p}_{m} } := \max_{i \in [m]} \left\{ \sup_{k \in \Z_0} \frac{ |k|^{p} }{\log(1 + |k|)} \left| \widehat \phi_{i,k} \right| \right\}
\end{align*}
to quantify the decay-rate of its Fourier coefficients. Proposition \ref{prop:linearA} part v) yields a worst-case decay estimate
\begin{align}\label{eq:decay}
\sup_{k \in \Z_0} \, |k|^{3} \left|\widehat \phi_{i,k}(t) \right| \leq \min\left\{ \|\Phi_0\|_{\dot \cH^{2}_m} \zeta^{-\frac12}_i(t) , \|\Phi_0\|_{\dot \cH^{3}_{m}} \right\} + \mathfrak{F}^{*} 
\end{align}
along any mild solution trajectory, and so the $\dot \cV^{3}_m$ semi-norm of the solution remains finite on any time interval $0 < \tau < T$ bounded away from the origin. Additionally, if $\Phi_0 \in \cH^{3}_{m}$ then we may take $\tau = 0$ and obtain a uniform in time $\dot \cV^{3}_m$ estimate. On any time interval $[\tau,T]$ for which
$$
\max\left\{ \sup_{\tau \leq t \leq T} \;\, \| \Phi(t) \|_{\dot \cV^{3}_{m}}, \|\Phi(t)\|_{\cH^{2}_{m}} \right\} := \Phi^{*}_{\tau,T} < +\infty 
$$
we may appeal to \cite{knot,SIMA} (c.f. lemma 2.1 of the former) to obtain a corresponding uniform $\dot \cV^{<2}_m$ estimate
\begin{align}\label{eq:decay2}
\sup_{\tau \leq t \leq T} \;\, \|\left(\cT_{1}\left[(\Phi(t),\bveps(t))\right],\ldots,\cT_{m}\left[(\Phi(t),\bveps(t))\right]\right) \|_{\dot \cV^{<2}_m }\leq \mathfrak{C}^{*}\left( \Phi^{*}_{\tau,T}, \mathfrak{F}^{*} \right),
\end{align}
for the transport part of the forcing. Now set $\phi(t) = \phi_i(t), \veps = \veps_i(t),\mathcal{T}(t) = \mathcal{T}_{i}\left[ (\Phi(t),\bveps(t)) \right]$ and $\mathcal{R}(t) = \mathcal{R}_{i}\left[ (\Phi(t),\bveps(t)) \right]$ for some $i \in [m]$ and recall that
\begin{align*}
\hat \phi_k(t) - \hat \phi_{k}(s) &\;=  \int^{t}_{s} \re^{-k^2\left( \zeta(t) - \zeta(z) \right) }\left( \widehat{ \mathcal{T} }_k(z) + \widehat{ \mathcal{R} }_k(z)\right)\veps(z) \, \rd z \qquad \zeta(t) = \int^{t}_{0} \veps(s) \, \rd s \\
&:=\mathrm{I}_{k}(s,t) + \mathrm{II}_{k}(s,t)
\end{align*}
by definition of a mild solution. If $s,t \in [\tau,T]$ then the upper bound
$$
\left| \mathrm{I}_{k}(s,t) \right| \leq \mathfrak{C}^{*}\left( \Phi^{*}_{\tau,T}, \mathfrak{F}^{*} \right)\frac{ \log\left(1 + |k| \right) }{k^2}\left( \zeta(t) - \zeta(s) \right)
$$
holds by the $\dot \cV^{<2}_{m}$ bound \eqref{eq:decay2} and direct integration. In particular, the $\dot H^{1}(\T)$ norm
$$
\left( \sum_{k \in \Z_0} k^2 \left| \mathrm{I}_{k}(s,t) \right|^{2} \right)^{\frac12} \leq \mathfrak{C}^{*}\left( \Phi^{*}_{\tau,T}, \mathfrak{F}^{*} \right)\left( \zeta(t) - \zeta(s) \right)\left( \sum_{k \in \Z_0} \frac{ \log^{2}(1 + |k|)}{k^2} \right)^{\frac12}
$$ 
is Lipschitz in time. An application of Minkowski's inequality
$$
\left( \sum_{k \in \Z_0} k^2 \left| \mathrm{II}_{k}(s,t) \right|^{2} \right)^{\frac12} \leq \int^{t}_{s} \left( \sum_{k \in \Z_0} k^2 \left| \widehat \cR_{k}(z) \right|^2 \right)^{\frac12} \veps(s)
$$
then combines with the $H^{1}(\T)$ bound \eqref{eq:reacH1} and the triangle inequality to show that the solution itself
$$
\| \Phi(t) - \Phi(s)\|_{\cH^{1}_{m}} \leq \mathfrak{C}^{*}\left( \Phi^{*}_{\tau,T}, \mathfrak{F}^{*} \right)\left( \zeta(t) - \zeta(s) \right)
$$
is similarly Lipschitz in time. 

We shall use this temporal Lipschitz estimate to obtain an $H^{3}(\T)$ bound for the solution. Fix $\tau < t \leq T$ and note that
\begin{align*}
\hat \phi_k(t) &\;= \re^{-k^2\left( \zeta(t) - \zeta(\tau) \right)} \hat \phi_{k}(\tau) + \int^{t}_{\tau} \re^{-k^2\left( \zeta(t) - \zeta(z) \right) }\left( \widehat{ \mathcal{T} }_k(z) + \widehat{ \mathcal{R} }_k(z)\right)\veps(z) \, \rd z \\
&:= \mathrm{I}_{k}(t) + \mathrm{II}_{k}(t) + \mathrm{III}_{k}(t)
\end{align*}
by definition of mild solution. The $\dot \cV^{3}_{m}$ bound \eqref{eq:decay} gives an easy $\dot H^{3}(\T)$ estimate
\begin{align*}
\| \mathrm{I}_{k}(t) \|_{\dot H^{3}(\T)} = \left( \sum_{k \in \Z_0} k^6 \re^{-2k^2\left( \zeta(t) - \zeta(\tau) \right) } |\hat \phi_{k}(\tau)|^{2} \right)^{\frac12} &\leq \Phi^{*}_{\tau,T} \left( \sum_{k \in \Z_0} \re^{-2k^2\left(\zeta(t) - \zeta(\tau) \right)} \right)^{\frac12} \\
&\leq \mathfrak{C}^{*}\left( \Phi^{*}_{\tau,T} \right) \left( \zeta(t) - \zeta(\tau) \right)^{-\frac14}
\end{align*}
for the first term, while the $\dot \cV^{<2}_{m}$ bound \eqref{eq:decay2} gives a similar $\dot H^{3}(\T)$ estimate
\begin{align*}
\| \mathrm{II}_{k}(t) \|_{\dot H^{3}(\T)} &\leq \mathfrak{C}^{*}\left( \Phi^{*}_{\tau,T}, \mathfrak{F}^{*} \right)\left( \sum_{k \in \Z_0} k^{6} \left| \int^{t}_{\tau} \re^{-k^{2}\left( \zeta(t) - \zeta(s) \right)}\frac{ \log(1+|k|)}{|k|^{2}} \veps(s) \, \rd s\right|^{2} \right)^{\frac12} \\
& \leq \mathfrak{C}^{*}\left( \Phi^{*}_{\tau,T}, \mathfrak{F}^{*} \right) \left( \sum_{k \in \Z_0}  \frac{ \log^2(1+|k|)}{k^2} \right)^{\frac12} \leq \mathfrak{C}^{*}\left( \Phi^{*}_{\tau,T}, \mathfrak{F}^{*} \right) 
\end{align*}
for the second term. Now write the reaction term as
\begin{align*}
\cR(t) &= \mu_{a}(t)\phi(t) + F\left[ (\Phi(t),\bveps(t) ) \right] \dot \phi^{\perp} := \cR^{(1)}(t) + \cR^{(2)}(t) \qquad \widehat \cR_{k}(t) = \widehat \cR^{(1)}_{k}(t) + \widehat \cR^{(2)}_{k}(t),
\end{align*}
and note that the $H^{3}(\T)$ bound
$$
\left( \sum_{k \in \Z_0} k^{6} \left| \int^{t}_{\tau} \re^{-k^{2}\left( \zeta(t) - \zeta(s) \right)} \widehat \cR^{(1)}_{k}(z) \veps(s) \, \rd s\right|^{2} \right)^{\frac12} \leq \mathfrak{C}^{*}\left( \Phi^{*}_{\tau,T} \right)
$$
follows as before from the $\dot \cV^{3}_{m}$ estimate. Now perform the decomposition
$$
\int^{t}_{\tau} \re^{-k^{2}\left( \zeta(t) - \zeta(s) \right) }\widehat \cR^{(2)}_{k}(s) \veps(s) \, \rd s = \int^{t}_{\tau} \re^{-k^{2}\left( \zeta(t) - \zeta(s) \right) }\left( \widehat \cR^{(2)}_{k}(s) - \widehat \cR^{(2)}_{k}(t) \right) \veps(s) \, \rd s + \frac{1 - \re^{-k^2\left( \zeta(t) - \zeta(\tau) \right)} }{k^{2}} \widehat \cR^{(2)}_{k}(t)
$$
and use the $H^{1}(\T)$ bound \eqref{eq:reacH1} to obtain
$$
\left( \sum_{k \in \Z_0} k^{6} \left|\frac{1 - \re^{-k^2\left( \zeta(t) - \zeta(\tau) \right)} }{k^{2}} \widehat \cR^{(2)}_{k}(t)\right|^{2} \right)^{\frac12} \leq \mathfrak{C}^{*}\left( \Phi^{*}_{\tau,T} , \mathfrak{F}^{*} \right),
$$
and so an $H^{3}(\T)$ bound will follow from showing that the series
$$
\sum_{k \in \Z_0} k^{6}\left| \int^{t}_{\tau} \re^{-k^{2}\left( \zeta(t) - \zeta(s) \right) }\left( \widehat \cR^{(2)}_{k}(s) - \widehat \cR^{(2)}_{k}(t) \right) \veps(s) \, \rd s \right|^{2}
$$
converges. For any $f \in C([\tau,T];L^{1}(\T))$ the simple inequality
$$
\left| \hat f_{k}(t) - \hat f_{k}(s) \right| = \frac1{2\pi} \left| \int_{\T} \left( f(x,s) - f(x,t) \right) \re^{-ikx} \right| \leq \| f(t) - f(s) \|_{L^{2}(\T)}
$$
holds, and so by the $H^{1}(\T) \mapsto L^{2}(\T)$ property of $F[(\Phi,\bveps)]$ we may infer
\begin{align*}
\left| \widehat \cR^{(2)}_{k}(s) - \widehat \cR^{(2)}_{k}(t) \right| \leq \| \cR^{(2)}(t) - \cR^{(2)}(s) \|_{L^{2}(\T)} &\leq \mathfrak{F}^{*}\left( \| \Phi(t) - \Phi(s) \|_{\cH^{1}_{m}} + \|\bveps(t) - \bveps(s)\|_{\R^{m}_{+}} \right) \\
&\leq\mathfrak{C}^{*}\left( \Phi^{*}_{\tau,T},\bveps_{*},\bveps^{*}, \mathfrak{F}^{*} \right)\left( \zeta(t) - \zeta(s) \right)
\end{align*}
since $\Phi$ is an $H^{1}(\T)$ Lipschitz function in time and $\bveps(t)$ is $C^{1}$ in time. A temporal change of variables then gives
\begin{align*}
\left| \int^{t}_{\tau} \re^{-k^{2}\left( \zeta(t) - \zeta(s) \right) }\left( \widehat \cR^{(2)}_{k}(s) - \widehat \cR^{(2)}_{k}(t) \right) \veps(s) \right| &\leq \mathfrak{C}^{*}\left( \Phi^{*}_{\tau,T},\bveps_{*},\bveps^{*}, \mathfrak{F}^{*} \right)\int^{t}_{\tau} \re^{-k^{2}\left( \zeta(t) - \zeta(s) \right)}\left( \zeta(t) - \zeta(s) \right) \veps(s) \, \rd s \\
&= \frac{\mathfrak{C}^{*}\left( \Phi^{*}_{\tau,T},\bveps_{*},\bveps^{*}, \mathfrak{F}^{*} \right)}{k^{4}} \int^{k^{2}\left( \zeta(t) - \zeta(\tau) \right) }_{0} z \re^{-z} \, \rd z,
\end{align*}
and so the desired $H^{3}(\T)$ bound
$$
\left( \sum_{k \in \Z_0} k^{6}\left| \int^{t}_{\tau} \re^{-k^{2}\left( \zeta(t) - \zeta(s) \right) }\left( \widehat \cR^{(2)}_{k}(s) - \widehat \cR^{(2)}_{k}(t) \right) \veps(s) \, \rd s \right|^{2} \right)^{\frac12} \leq \mathfrak{C}^{*}\left( \Phi^{*}_{\tau,T},\bveps_{*},\bveps^{*}, \mathfrak{F}^{*} \right)
$$
indeed follows. All together, we obtain an $H^{3}(\T)$ estimate
$$
\sup_{\tau \leq t \leq T} \;\, \| \Phi(t) \|_{\cH^{3}_{m}} \leq \mathfrak{C}^{*}\left( \Phi^{*}_{\tau,T} \right) \left( \zeta(t) - \zeta(\tau) \right)^{-\frac14} + \mathfrak{C}^{*}\left( \Phi^{*}_{\tau,T},\bveps_{*},\bveps^{*}, \mathfrak{F}^{*} \right)
$$
for the solution itself. Moreover, if $\Phi_{0} \in \cH^{3}_{m}$ we may modify the first term to obtain a corresponding estimate
$$
\sup_{0 \leq t \leq T} \;\, \| \Phi(t) \|_{\cH^{3}_{m}} \leq \|\Phi_0\|_{\cH^{3}_{m}} + \mathfrak{C}^{*}\left( \Phi^{*}_{0,T},\bveps_{*},\bveps^{*}, \mathfrak{F}^{*} \right)
$$
across the entire time interval.
\subsection*{Uniqueness}
Consider two mild solutions $(\Phi,\bveps)$ and $(\Psi,\bsigma)$ to the evolution. If $\Phi_{0} = \Psi_0 \in \cH^{3}_{m}$ then we may demonstrate uniqueness via a straightforward energy estimate. Fix some $j \in [m]$ arbitrary and compare the shape evolutions
\begin{align*}
\partial_{t} \phi_{j} &= \veps_{j}\left( \ddot \phi_j + \po\big( a_j \big) \dot \phi_j + F_{j}\left[ (\Phi,\bveps) \right] + \mu_{a_j} \phi_j \right) \qquad &&a_{j} = |\ddot \phi_j|^{2} + \langle \ddot \phi_j, F_{j}\left[ (\Phi,\bveps) \right] \rangle \\
\partial_{t} \psi_{j} &= \sigma_{j}\left( \ddot \psi_j + \po\big( b_j \big) \dot \psi_j + F_{j}\left[ (\Psi,\bsigma) \right] + \mu_{b_j} \psi_j \right) \qquad &&b_{j} = |\ddot \psi_j|^{2} + \langle \ddot \psi_j, F_{j}\left[ (\Psi,\bsigma) \right] \rangle
\end{align*}
along with the length evolutions
\begin{align*}
\rd_{t} \veps_i = 2 \veps^{2}_{i} \mu_{a_i} \qquad \rd_{t} \sigma_i = 2 \veps^{2}_{i} \mu_{b_i} \qquad i \in [m]
\end{align*}
for some family of $m$ generic mappings $F_i$ that obey the $H^{1}(\T)$ bound and $H^{1}(\T) \mapsto L^{2}(\T)$ Lipschitz property
\begin{align}\label{eq:regbound2}
\left.\begin{aligned}
\| F_i\left[(\Phi,\bveps)\right] \|_{H^{1}(\T)} &\leq \mathfrak{F}^{*} \\
\| F_i[(\Psi, \bsigma)] - F_i[(\Phi,\bveps)]\|_{L^{2}(\T)} &\leq \mathfrak{F}^{*}\left(  \|\Psi - \Phi \|_{\cH^{1}_m(\T)} + \|\bsigma -  \bveps\|_{\R^m_{+}} \right) \quad
\end{aligned}\right\rbrace 
\quad \text{on} \quad B_{\delta}\left( (\Phi_0,\bveps) \right)
\end{align}
locally near the initial condition. By the regularity assertion we may fix $T>0$ so that both \eqref{eq:regbound2} and the bounds
$$
\| \Phi \|_{C^{3}_{m}(T)} + \|\Psi\|_{C^{3}_{m}(T)} \leq \Phi^{*} \qquad \text{and} \qquad 0 < \bveps_{*} \leq \veps_i(t) \leq \bveps^{*}
$$
hold uniformly in time.

Fix $i \in [m]$ and consider the difference $\xi = \phi_i - \psi_i$ (subscripts omitted for the sake of notational convenience) between mild solutions. Differentiating gives the equality
\begin{align*}
\partial_t \dot \xi &= \veps \left[ \dddot \xi + a \dot\phi - b \dot\psi +  \po\left( a \right) \ddot \phi - \po\left( b \right) \ddot \psi + \dot F\left[ (\Phi,\bveps) \right] - \dot F\left[ (\Psi,\bsigma) \right] \right] \\
&+ \left( \sigma - \veps \right)\left[ \dddot \psi + b \dot \psi + \po\big( b \big) \ddot \psi + \dot F\left[ (\Psi,\bsigma) \right] \right] := \veps Q + \left( \sigma - \veps \right)R
\end{align*}
valid in the $L^{2}(\T)$ sense. Taking the inner product with $\dot \xi$ and integrating over $\T$ yields
\begin{align*}
\int_{\T} \langle Q , \dot \xi \rangle \, \rd x &= - \int_{\T} | \ddot \xi |^{2} \, \rd x - \int_{\T} \langle F\left[ (\Phi,\bveps) \right] -  F\left[ (\Psi,\bsigma) \right]  , \ddot \xi \rangle \, \rd x \\
&+ \int_{\T} \langle a \dot \phi - b \dot \psi , \dot \xi \rangle \, \rd x + \int_{\T} \langle \po\left( a \right) \ddot \phi - \po\left( b \right) \ddot \psi , \dot \xi \rangle \, \rd x := -\left( \mathrm{I} + \mathrm{II} \right) + \mathrm{III} + \mathrm{IV}
\end{align*}
after an integration by parts.

We shall bound the latter three in terms of the first in a straightforward manner using the generic inequality
$$
ab \leq \frac{ a^{p} }{p \lambda^{p}} + \frac{\lambda^{q} b^{q} }{q} \qquad \text{for} \qquad q = \frac{p}{p-1},\; p \geq 1,\; \lambda > 0
$$
and the Gagliardo-Nirenberg embeddings \eqref{eq:gns} as the main tools. An inequality 
$$
|\mathrm{II}| \leq \mathfrak{F}^{*}\left( \|\Phi - \Psi\|^{2}_{H^{1}(\T)} + \| \bveps - \bsigma \|^{2}_{\R^{m}_{+}} \right) + \frac12 \int_{\T} |\ddot \xi|^{2} \, \rd x
$$
for the second term follows easily from the $H^{1}(\T) \mapsto L^{2}(\T)$ Lipschitz hypothesis. Decompose the third term as 
$$
\mathrm{III} = \int_{\T} a |\dot \xi|^{2} \, \rd x + \int_{\T} (a-b)  \langle\dot \xi , \dot \psi \rangle \, \rd x  := \mathrm{III}^{ {\rm A}} + \mathrm{III}^{ {\rm B}}
$$
and note $a \in H^{1}(\T) \subset L^{\infty}(\T)$ thanks to the uniform $H^{3}(\T)$ bound on $\phi$ and the assumption \eqref{eq:regbound2} on the forcing; the uniform bound
$$
\mathrm{III}^{ {\rm A}} \leq \mathfrak{C}^{*}\left( \mathfrak{F}^{*} , \Phi^{*} \right) \| \Phi - \Psi\|^{2}_{H^{1}(\T)}
$$
then easily follows. The finer decomposition
$$
\mathrm{III}^{ {\rm B}} = \int_{\T} \langle \ddot \phi + \ddot \psi , \ddot \xi \rangle \langle \dot \xi , \dot \psi \rangle \, \rd x + \int_{\T} \langle \ddot \phi , F\left[ (\Phi,\bveps)\right] -  F\left[ (\Psi,\bsigma)\right]\rangle \langle \dot \xi , \dot \psi \rangle \, \rd x + \int_{\T} \langle F\left[ (\Psi,\bsigma)\right] , \ddot \xi \rangle\langle \dot \xi,\dot \psi \rangle \, \rd x
$$
combines with the  $H^{3}(\T)$ bound and the assumptions \eqref{eq:regbound2} in a similar way to yield
$$
\mathrm{III}^{ {\rm B}} \leq \frac14 \int_{\T} |\ddot \xi|^{2} \, \rd x + \mathfrak{C}^{*}\left( \Phi^{*},\mathfrak{F}^{*}\right)\left( \|\Phi - \Psi\|^{2}_{H^{1}(\T)} + \| \bveps - \bsigma \|^{2}_{\R^{m}_{+}} \right),
$$
and so we obtain an overall upper bound
$$
\mathrm{III} \leq \frac14 \int_{\T} |\ddot \xi|^{2} \, \rd x + \mathfrak{C}^{*}\left( \Phi^{*},\mathfrak{F}^{*}\right)\left( \|\Phi - \Psi\|^{2}_{H^{1}(\T)} + \| \bveps - \bsigma \|^{2}_{\R^{m}_{+}} \right)
$$
for the third term. Decompose the fourth term
$$
\mathrm{IV} = \int_{\T} \po\big(a \big) \langle \ddot \xi , \dot \xi \rangle \, \rd x + \int_{\T} \po\big(a-b\big)\langle \ddot \psi , \dot \xi \rangle \, \rd x
$$
and apply the straightforward estimate
$$
\mathrm{IV} \leq \|a\|_{L^{\infty}(\T)} \| \ddot \xi\|_{L^{2}(\T)} \|\dot \xi\|_{L^{2}(\T)} + \|\ddot \psi\|_{L^{\infty}(\T)} \| \po\big(a-b) \|_{L^{2}(\T)} \|\dot \xi\|_{L^{2}(\T)},
$$
then combine the bounds 
\begin{align*}
\po\big( f \big) \leq 2 \|f\|_{L^{1}(\T)} \quad \text{and} \quad
\|a - b\|_{L^{1}(\T)} \leq \mathfrak{C}^{*}\left( \Phi^{*},\mathfrak{F}^{*} \right)\left(  \| \ddot \xi \|_{L^{2}(\T)} + \|F\left[ (\Phi,\bveps) \right] - F\left[(\Psi,\bsigma)\right] \|_{L^{2}(\T)} \right)
\end{align*}
together with the Lipschitz assumption \eqref{eq:regbound2} to obtain
$$
\mathrm{IV} \leq \frac18 \int_{\T} |\ddot \xi|^{2} \, \rd x + \mathfrak{C}^{*}\left( \Phi^{*},\mathfrak{F}^{*}\right)\left( \|\Phi - \Psi\|^{2}_{H^{1}(\T)} + \| \bveps - \bsigma \|^{2}_{\R^{m}_{+}} \right)
$$
for the fourth and final term. Now as the remainder
$$
R = \dddot \psi + b \dot \psi + \po\big( b \big) \ddot \psi + \dot F\left[ (\Psi,\bsigma) \right] \qquad \|R\|_{L^{2}(\T)} \leq \mathfrak{C}^{*}\left( \Phi^{*},\mathfrak{F}^{*} \right)
$$
obeys a uniform $L^{2}(\T)$ estimate, the upper bound
$$
\left( \sigma - \veps \right) \int_{\T} \langle \dot \xi , R \rangle \leq \mathfrak{C}^{*}\left( \Phi^{*},\mathfrak{F}^{*}\right)\left( \|\Phi - \Psi\|^{2}_{H^{1}(\T)} + \| \bveps - \bsigma \|^{2}_{\R^{m}_{+}} \right)
$$
follows easily by Cauchy-Schwarz. All together, we conclude that the energy estimate
$$
\frac{\rd}{\rd t} \left(\frac12 \int_{\T} |\dot \phi_i - \dot \psi_i|^{2} \, \rd x \right) \leq - \frac{ \bveps_{*} }{8} \int_{\T} |\ddot \phi_i - \ddot \psi_i|^{2}\, \rd x + \mathfrak{C}^{*}\left( \bveps^{*},\Phi^{*},\mathfrak{F}^{*}\right)\left( \|\Phi - \Psi\|^{2}_{H^{1}(\T)} + \| \bveps - \bsigma \|^{2}_{\R^{m}_{+}} \right)
$$
holds for $i \in [m]$ arbitrary. Arguing as above reveals that the difference in means obeys a similar inequality
$$
\frac{\rd}{\rd t}\frac12  \left( \int_{\T} (\phi_i - \psi _i) \, \rd x \right)^{2} \leq \frac{\bveps_{*}}{16}\int_{\T} |\ddot \phi_i - \ddot \psi_i|^2 \, \rd x + \mathfrak{C}^{*}\left(\bveps_{*}, \bveps^{*},\Phi^{*},\mathfrak{F}^{*}\right)\left( \|\Phi - \Psi\|^{2}_{H^{1}(\T)} + \| \bveps - \bsigma \|^{2}_{\R^{m}_{+}} \right),
$$
and so in fact we may conclude the descent inequality
\begin{align}\label{eq:unique1}
E'(t) &\leq - \frac{\bveps_{*}}{16m} \sum_{i \in [m] } \int_{\T} |\ddot \phi_i - \ddot \psi_i|^{2}\, \rd x +  \mathfrak{C}^{*}\left(\bveps_{*}, \bveps^{*},\Phi^{*},\mathfrak{F}^{*}\right)\left( E(t) + \| \bveps - \bsigma \|^{2}_{\R^{m}_{+}} \right) \\
E(t) &:= \frac1{m} \sum^{m}_{i=1} \| \phi_i - \psi_i\|^{2}_{H^{1}(\T)} \nonumber
\end{align}
holds. A simple comparison estimate on $\veps_i - \sigma_i$ shows
$$
\frac{\rd}{\rd t} \frac{(\veps_i - \sigma_i)^{2} }{2}\leq  \frac{\bveps_{*}}{16}\int_{\T} |\ddot \phi_i - \ddot \psi_i|^{2}\, \rd x +  \mathfrak{C}^{*}\left(\bveps_{*}, \bveps^{*},\Phi^{*},\mathfrak{F}^{*}\right)\left( \|\Phi - \Psi\|^{2}_{H^{1}(\T)} + \|\bveps - \bsigma\|^{2}_{\R^{m}_{+}} \right),
$$
which after averaging, adding to \eqref{eq:unique1} and applying Gronwall's inequality yields uniqueness. Combining this argument with our existence and regularity results then yields the main theorem of this section.

\begin{theorem}[Local Well-Posedness] Fix constants $\delta,\mathfrak{F}^{*} > 0$ and an initial datum $(\Phi_0,\bveps_0) \in \cH^{2}_{m} \times \R^{m}_{+}$ arbitrarily. Assume that the mappings $F_i : \cH^2_m \times \R^{m}_{+} \mapsto H^1(\T)$ obey the $H^{1}(\T)$ bound and $H^{1}(\T) \mapsto L^{2}(\T)$ Lipschitz property
\begin{align*}
\left.\begin{aligned}
\| F_i\left[(\Phi,\bveps)\right] \|_{H^{1}(\T)} &\leq \mathfrak{F}^{*} \\
\| F_i[(\Psi, \bsigma)] - F_i[(\Phi,\bveps)]\|_{L^{2}(\T)} &\leq \mathfrak{F}^{*}\left(  \|\Psi - \Phi \|_{\cH^{1}_m(\T)} + \|\bsigma -  \bveps\|_{\R^m_{+}} \right) \quad
\end{aligned}\right\rbrace 
\quad \text{on} \quad B_{\delta}\left( (\Phi_0,\bveps) \right)
\end{align*}
locally near the initial condition. If $a_i := |\ddot \phi_i|^2 + \langle \ddot \phi_i , \dot \phi^{\perp}_i\rangle F_i[(\Phi,\bveps)]$ for $i \in [m]$ then the following hold ---
\begin{enumerate}[{\rm i)}]
\item There exists $T = T(\mathfrak{F}^{*},\delta,\Phi_0,\bveps_0)$ so that the coupled system
$$
\left.\begin{aligned}
\partial_{t} \phi_i &= \veps_i\left( \ddot \phi_i + \po\big(a_i\big) \dot \phi_{i} + F_i[(\Phi,\bveps)]\dot \phi^{\perp}_i + \mu_{a_i}\phi_i\right)&&\,\,\rd_{t} \veps_i = 2\veps^2_i \mu_{a_i} \quad\\
\phi_i(0) &= (\Phi_0)_{i} && \veps_i(0) = (\bveps_0)_i \quad
\end{aligned}\right\rbrace \quad \text{for all} \quad i \in [m]
$$
has a mild solution $(\Phi,\bveps) \in C^{2}_{m}(T) \times C_m(T)$ that remains in the neighborhood $B_{\delta}$ of the initial datum. In addition, the bounds
\begin{align*}
\Phi^{*}_{\tau,T} &:= \sup_{\tau \leq t \leq T}\;\,\|\Phi(t)\|_{\dot \cV^{3}_{m}} \leq \min\left\{ \|\Phi_0\|_{\dot \cH^{2}_{m}}(\bveps\tau)^{-\frac12} , \|\Phi_0\|_{\dot \cH^{3}_{m}}\right\} + \mathfrak{F}^{*} \qquad \text{and} \qquad \\
0 &\,< \bveps_{*} \leq \veps_i(t) \leq \bveps^{*} < +\infty
\end{align*}
hold on $[0,T]$ for any $\tau >0$ arbitrary. If $\Phi_0 \in \cH^{3}_m$ then the solution is unique.
\item For any $\tau \leq s \leq t \leq T,$ any solution from {\rm i)} obeys the estimates
\begin{align*}
\| \Phi(t) - \Phi(s) \|_{\cH^{1}_{m}}  &\leq \frC^{*}\left(\Phi^{*}_{\tau,T},\mathfrak{F}^*,\bveps^{*}\right)(t-s) \qquad \text{and} \qquad \\
\| \Phi(t) \|_{\cH^{3}_{m}} &\leq \frC^{*}\left(\Phi^{*}_{\tau,T},\mathfrak{F}^{*},\bveps_{*},\bveps^*\right)\left(1 + \min\left\{ \tau^{-\frac14}, \|\Phi_0\|_{\cH^{3}_{m}} \right\} \right)
\end{align*}
uniformly in time. Any such solution is H\"{o}lder continuous in space-time with
\begin{align*}
\|\Phi(t) - \Phi(s)\|_{\cH^{2+\vth}_m} &\leq \frC^{*}\left(\vth,\Phi^{*}_{\tau,T},\mathfrak{F}^*,\bveps_{*},\bveps^{*}\right)|t-s|^{\frac{1-\vth}{2}} \qquad \text{and} \qquad \\
\|\Phi(t) - \Phi(s)\|_{C^{2,\alpha}(\T)} &\leq \frC^{*}\left(\alpha,\Phi^{*}_{\tau,T},\mathfrak{F}^*,\bveps_{*}\bveps^{*}\right)|t-s|^{ \frac{1 - 2\alpha}{4} }
\end{align*}
for any $0 < \alpha \leq \frac12,$ so in particular $\Phi_{t} \in C^{\frac{1-2\alpha}{4}}([\tau,T];C^{\alpha}(\T))$ and the solution is classical.
\item If $|\dot \phi_i(x,0)| \equiv 1$ for some $i \in [m]$ then $|\dot \phi_i(x,t)| \equiv 1$ for as long as the solution exists.
\end{enumerate}
\end{theorem}
\begin{proof}
We have shown {\rm i)} and the first part of {\rm ii)} already. Appealing first to Sobolev embedding and then to interpolation on the $\cH^{s}_{m}$ norm gives the bound
$$
\| \Phi(t) - \Phi(s) \|_{C^{2,\vth-1/2}(\T)} \leq \frc_{\alpha} \| \Phi(t) - \Phi(s) \|_{\cH^{2+\vth}(\T)} \leq  \frC^{*}\left(\alpha,\Phi^{*}_{\tau,T},\mathfrak{F}^*,\bveps_{*},\bveps^{*}\right) |t-s|^{\frac{1-\vth}{2}},
$$
and as a consequence, the claimed continuity in space-time. The claimed regularity of $\Phi_t$ then follows from the differential equation itself. Finally, {\rm iii)} follows from the fact that
$$
z_t = \veps \left( z_{xx} + a(2z-1) + \po\big(a\big)z_x \right) \qquad \text{for} \qquad z(x,t) := \frac12 |\dot \phi(x,t)|^2,
$$
holds in the $L^{2}(\T)$ sense by a straightforward differentiation. So, if $z(x,0) \equiv 1/2$ then $z(x,t) = 1/2$ for all $0 \leq t \leq T$ by uniqueness of the linear initial value problem.
\end{proof}

We conclude this section with a few remarks. First, continuation is a straightforward matter. In the usual way, we may continue the trajectory from $[0,T]$ to some larger time interval $[0,T+\delta T]$ as long as $(\Phi(T),\bveps(T),\cG)$ form a compatible triple. In other words, if we cannot continue the solution then a loss of $\cH^{2}$ regularity, a loss of embeddedness, or the occurrence of intersections between distinct interfaces must occur. Of course, the precise conditions that allow for continuation depend on the choice of non-local kernels. For example, if we consider only regular kernels then only a loss of $\cH^{2}$ regularity would prevent continuation. (If the speed degenerates $\sigma_i(T) = 0$ for some $i \in [m]$ but $\|\Phi(T)\|_{\cH^{2}_m}$ is finite, we simply remove any trivial, length-zero interface from the system before continuation). Second, at this level of generality we cannot conclude anything beyond local well-posedness of the dynamic. Both loss of embeddedness and finite-time blow-up can occur, even when working with smooth kernels.

\section{Embeddedness, Kernel Distortions and Distortion Measures}\label{sec:Embeded}
With existence in hand, we turn our attention to the dynamics of geometric properties of the interfaces $\ga_i \in \Gamma$ under the evolution. We focus mainly on embeddedness, for which the distortion
$$
\delta_{\infty}(\ga) := \sup_{y\neq x} \;\, \frac{D_{\ga}(x,y)}{|\ga(x) - \ga(y)|}
$$
provides a quantitative measure. We have two motivations for this focus. Our primary motivation arises from a modeling concern; as the dynamic \eqref{eq:shapedynamic} provides a reduced description for the motion of a collection of $m$ coupled, mutually interacting \emph{interfaces}, model validity demands that each shape-scaling pair $(\phi_i(\cdot,t),\sigma_i(t))$ should represent a closed, embedded curve. Analytical concerns also drive our inquiry into embeddedness. Indeed, for singular interaction kernels $g_{ij}$ we require an embedded interface (c.f. proposition \ref{prop:nlest0}) in order to guarantee sufficient regularity of the nonlocal forcing. Moreover, non-embedded curves may form curvature singularities in finite time. For such initial data, the dynamic \eqref{eq:shapedynamic} is only locally well-posed even if the interfacial speeds $\sigma_i(t)$ never vanish. In this way, embeddedness and global well-posedness are intimately tied.

To set the stage we recall that embeddedness is well-understood for pure curve shortening flow. More specifically, the distortion-like quantity
\begin{align}\label{eq:pseudist}
\Delta^{2}(\psi) := \sup_{x, z \neq 0} \;\, \frac{2\, \sigma^{2}(\psi)\left(1 - \cos\left( \sigma^{-1}(\psi)\int^{x+z}_{x} |\dot \psi(u)| \, \rd u\right) \right)}{|\psi(x) - \psi(y)|^2}
\end{align}
which we refer to as the \emph{pseudo-distortion} for lack of a better terminology, decays monotonically in time. It is equivalent to the distortion in the sense that the bounds
$$
\Delta(\psi) \leq \delta_{\infty}(\psi) \leq \frac{\pi}{2} \Delta(\psi)
$$
hold for all closed curves, so a planar motion by mean curvature preserves the embeddedness of the initial condition. To formulate our task in the general case, recall that the shape-scaling dynamic \eqref{eq:shapedynamic} furnishes a family of $m$ unit speed curves $\phi_i(\cdot,t) \in C^{2,\frac12}(\T)$ at each instant in time. By scale and parametrization invariance of \eqref{eq:pseudist}, we therefore wish to understand how the supremum
$$
\Delta^2(\phi) := \sup_{x, z \neq 0}\;\, K\left(|\phi(x) - \phi(x+z)|^2, 2(1-\cos z) \right) \qquad \text{for} \qquad K(u,v) = v u^{-1}
$$
evolves in time. We may also interpret the curvature bound of \cite{AB11} in a similar way. To over-simplify the argument somewhat, we may take a kernel of the form
\begin{align}\label{eq:generickernel}
K(u,v) = g\left(v, \frac{u}{v} \right)
\end{align}
and then show that the analogous supremum
\begin{align}\label{eq:Kdist0}
\Delta_{K}(\phi) := \;\, \sup_{x, z \neq 0}\;\, K\left(|\phi(x) - \phi(x+z)|^2, 2(1-\cos z) \right)
\end{align}
remains bounded in time under a planar curve shortening flow. If the function $g(\ell,r)$ behaves like $(1-r)/\ell$ near $r=1$ then a simple Taylor expansion shows
\begin{align}
K\left(|\phi(x) - \phi(x+z)|^2, 2(1-\cos z) \right) \approx \frac1{12}\left( \kappa^{2}_{\phi}(x) - 1 \right)_{+} \qquad \text{if} \qquad z \approx 0
\end{align}
and so the global-in-time curvature bound $\kappa^{2}_{\phi}(x,t) \leq 1 + 12\,\Delta_{K}(\phi(\cdot,0))$ immediately follows. In the presence of either normal growth $c_i \neq 0$ or a non-zero interaction $g_{ij} \neq 0$ between interfaces, the distortion-like quantities (\ref{eq:pseudist},\ref{eq:Kdist0}) need not remain finite. We therefore want some machinery to track the temporal \emph{evolution} of such quantities, so that we may make the mechanisms by which a loss of embeddedness occurs as transparent as possible. Our interest mainly lies in the pseudo-distortion \eqref{eq:pseudist}, but we shall treat the general case (\ref{eq:generickernel},\ref{eq:Kdist0}) since it will not require too much additional effort.

To motivate the argument, consider a simple case where we have some smooth function $u(x,t)$ of $x \in \T, t > 0$ and a differentiable, one-parameter family $x(t), t > 0$ of points that realize its supremum. We may then differentiate to find an evolution for the maximum,
$$
\frac{\rd}{\rd t} \; \max_{x \in \T} \;\, u(x,t) = u_{t}(x(t),t) + u_{x}(x(t),t)x^{\prime}(t) = u_{t}(x(t),t) = \int_{\T} u_t(x,t) \, \rd \pi_{t}
$$
where the probability measure $\pi_t = \delta(x - x(t))$ has support on the set of maximizers. The integral version
$$
\max_{x \in \T} \;\, u(x,T) = \max_{x \in \T} \;\, u(x,0) + \int^{T}_{0} \left( \int_{\T} u_t(x,t) \, \rd \pi_{t} \right) \, \rd t
$$
of this identity holds in the general case, and still provides a suitable way of formulating how the maximum evolves in time. We shall therefore work with the analogous identity
\begin{align}\label{eq:repr}
\Delta_{K}\left( \phi(T) \right) &\;= \Delta_{K}\left( \phi(0) \right) + 2 \int^{T}_{0} \left( \int_{\T} K_{u}\left( |\delta \phi(x,z,t)|^2 , 2(1-\cos z) \right)\langle \delta \phi(x,z,t),\delta \phi_t(x,z,t)\rangle \, \rd \pi_{t} \right) \, \rd t \\
\delta f(x,z,t) &:= f(x+z,t) - f(x,t),\nonumber
\end{align}
for \eqref{eq:Kdist0}, where the probability measures $\pi_t$ have support on the collection of $\Delta_{K}$-realizing pairs $(x,z)$ that attain the supremum. We then develop a ``calculus'' for such pairs $(x,z)$ that allow us to estimate integrals of the form
$$
\int_{\T} K_{u}\left( |\delta \phi(x,z,t)|^2 , 2(1-\cos z) \right)\langle \delta \phi(x,z,t),\delta \phi_t(x,z,t)\rangle \, \rd \pi_{t},
$$
and in this way obtain a means to track distortion-like quantities in time. 

Most of this section involves proving the representation \eqref{eq:repr} for a reasonably wide class of kernels. We conclude this section with a few examples to illustrate the utility of the the approach. In particular, it reveals the precise situations that lead to a loss of embededdness, how finite-time singularities then occur, and how non-local forcing can prevent a loss of embeddedness.

\subsection*{Distortion Kernels:} Fix some immersion $\phi \in C^{2,\alpha}(\T)$ that has unit speed. We choose to define a distortion kernel
$$
K(u,v) = g\left( v , \frac{u}{v} \right)
$$
in terms of some generic function $g(\ell,r)$ of the (squared) chord-length $v(z) = 2(1 - \cos z)$ along the standard circle and of the ratio
$$
r_{\phi}(x,z) = \frac{u_{\phi}(x,x+z)}{v(z)} := \frac{|\phi(x) - \phi(x+z)|^2}{2(1-\cos z)}
$$
between chord-length along $\phi$ and chord-length along the circle. On occasion we will parametrize the kernel in terms of the change of variables
$$
q(\ell,\alpha) := g\left( \ell , 1 + \alpha \ell \right) \qquad\qquad \left(\alpha > -\ell^{-1}\right)
$$
when we wish to examine the behavior of $K$ near the diagonal $u=v$ in the original variables. Given any such kernel we refer to the supremum
$$
\Delta_{K}(\phi) := \sup_{x \in \T, z \in \T \setminus\{0\} } \; f_{K,\phi}(x,z) \qquad \text{for} \qquad f_{K,\phi}(x,z) := K\left( u_{\phi}(x,x+z) , v(z) \right)
$$
as the kernel distortion or \emph{$K$-distortion} of the embedding. For general immersions $\ga \in C^{2,\alpha}(\T)$ we compute the $K$-distortion $\Delta_{K}(\ga)$ simply by using 
$$
u_{\ga}(x,x+z) = \frac{ |\ga(x) - \ga(x+z) |^{2} }{\sigma^{2}(\ga)} \qquad \text{and} \qquad v_{\ga}(x,x+z) := 2\left(1 - \cos\left( \frac1{\sigma(\ga)} \int^{x+z}_{x} |\dot \ga(u)| \, \rd u\right) \right)
$$
and then taking the supremum. If we view an embedded curve $\phi$ as an invertible function between $\mathbb{S}^{1}$ and its image $\mathrm{im}\big(\phi\big)$ then a uniform lower bound $r_{\phi}(x,x+z) \geq r_{*}(\phi) > 0$ holds precisely when $\phi$ has a Lipschitz inverse. We therefore assume that $g(\ell,r) \to \infty$ as $r \to 0$ for each $\ell$ fixed so that a bound on $\Delta_{K}(\phi)$ guarantees that $\phi \in H^{2}(\T)$ has finite distortion. In this way the kernel distortion $\Delta_{K}(\phi)$ furnishes a family of scale-invariant and parametrization-invariant quantifications of embeddedness. If, in addition, the function $g(\ell,r)$ behaves like $(1-r)/\ell$ near $r=1$  then a bound on $\Delta_{K}(\phi)$ also induces a bound on the bending energy.

We shall provide a calculus for computing with these distortion-like quantities for a relatively broad set of kernels. We simply ask that the kernel $K$ obeys appropriate regularity, monotonicity and convexity properties. For regularity, we assume the function $g(\ell,r)$ obeys the properties 
\begin{enumerate}[\indent\indent \rm R1)]
\item If $\ell,r>0$ then $g(\ell,r)$ is continuously differentiable.
\item The one-sided limits
$$
q_{0}(\alpha) := \lim_{ \ell \downarrow 0 } \;\,q(\ell,\alpha)\qquad \text{and} \qquad q_{1}(\alpha) := \lim_{\ell \downarrow 0} \;\, \partial_\ell q(\ell,\alpha)
$$
exist uniformly on any compact set.
\end{enumerate}
For monotonicity and convexity we assume that the function $g(\ell,r)$ obeys the properties
\begin{enumerate}[\indent\indent \rm M1)]
\item For each $0 < r \leq 1$ fixed, the function $f_{r}(\ell) := g(\ell,r)$ is non-increasing. For each $\ell > 0$ fixed, the function $f_{\ell}(r) := g(\ell,r)$ satisfies $f_{\ell}(0) = \infty, f^{\prime}_{\ell}(r)<0$ and is convex.
\item If $\alpha < 0$ then $q_1(\alpha)$ is strictly positive.
\end{enumerate}
\noindent Some simple examples of distortion kernels
\begin{align*}
&(\text{Pseudo-Distortion}) \quad &&K(u,v) = uv^{-1} && g(\ell,r) = r^{-1} && q(\ell,\alpha) = (1 + \alpha \ell)^{-1} \\
&(\text{M\"{o}bius}) \quad &&K(u,v) = u^{-1} - v^{-1} && g(\ell,r) = \ell^{-1}\left( r^{-1} - 1 \right) && q(\ell,\alpha) = -\alpha(1+\alpha \ell)^{-1}  \\
&(\text{KL Divergence}) \quad &&K(u,v) = v^{-1}\log\left( vu^{-1} \right)&& g(\ell,r) = -\ell^{-1}\log r && q(\ell,\alpha) = -\ell^{-1}\log(1 + \alpha \ell)
\end{align*}
help to elucidate the definition. We may easily verify the identities 
\begin{align*}
&q_0(\alpha) = 1,  &&q_0(\alpha) = -\alpha, \qquad\qquad \text{and} \qquad &&q_0(\alpha) = -\alpha,\\ 
&q_{1}(\alpha) = -\alpha, &&q_{1}(\alpha) = \alpha^{2}, \qquad\qquad\,\, \text{and} \qquad &&q_{1}(\alpha) = \alpha^{2}/2,
\end{align*}
respectively, for each of our three exemplars of distortion kernels. More generally, consider any separated function of the form $g(\ell,r) = \vrho(r)/\ell$ for some function $\vrho$ with $\vrho(1)=0$ and $\vrho(r)$ decreasing and convex. Then $q_0(\alpha) = \vrho^{\prime}(1)\alpha$ while $q_{1}(\alpha) = \vrho^{\prime\prime}(1) \alpha^2/2,$ so if $\vrho^{\prime\prime}(1) \neq 0$ then $g$ will satisfy both monotonicity properties; the property M2) is redundant in this case. The analytical reason for the definition is that the property R2) allows us to define the kernel $K$ and its derivatives along the diagonal. For example, if $\phi \in C^{2,\alpha}(\T)$ has unit speed then we may Taylor expand $f_{x}(z) := |\phi(x+z) - \phi(x)|^2$ near $z = 0$ to obtain the error estimates
\begin{align*}
\left| |\phi(x+z) - \phi(x)|^{2} - z^{2}\left(1 - \frac{z^2}{12}\kappa^{2}_{\phi}(x) \right) \right| &\leq \frac{ \|\phi\|^{2}_{C^{2,\alpha}(\T)} }{6} |z|^{4+\alpha} \\
\left| r_{\phi}(x,z) - 1 + \frac{z^2}{12}\left(\kappa^{2}_{\phi}(x)-1\right) \right| &\leq \frac{ 3\pi^{2} \|\phi\|^{2}_{C^{2,\alpha}(\T)} }{24} |z|^{2+\alpha} + \frac{z^4}{120}
\end{align*}
that hold for $z \in \T$ uniformly. We may therefore conclude the derivative identities
\begin{align}\label{eq:diaglimits}
\lim_{z\rightarrow 0}f_{K,\phi}(x,z)  &= q_{0}\left( \frac1{12}\left( 1 - \kappa^2_\phi(x)\right) \right) \nonumber \\
\lim_{z\rightarrow 0}\frac{ \partial_{z} f_{K,\phi}(x,z) }{2 z} &= q_{1}\left( \frac1{12}\left( 1 - \kappa^2_\phi(x)\right) \right)
\end{align}
and, in addition, that the argument $f_{K,\phi}$ of the $K$-distortion $\Delta_{K}(\phi)$ varies continuously along the diagonal. In a similar vein, if $\phi_{j} \to \phi$ in $C^{2,\alpha}(\T)$ then the uniform convergence
\begin{align}\label{eq:unifcv}
 \left\| f_{K,\phi_j} - f_{K,\phi}\right\|_{C(\T \times \T)} < \veps \qquad \mbox{whenever}\qquad \| \phi_j - \phi \|_{C^{2,\alpha}(\T)} < \delta\left( \veps, \|\phi\|_{C^{2,\alpha}(\T)} \right)
\end{align}
holds and so the $K$-distortion $\Delta_{K}(\phi)$ varies continuously with respect to the $C^{2,\alpha}(\T)$ topology. As the ratio $r_{\phi}(x,z)$ remains bounded away from the origin for embedded curves, we may use the regularity hypothesis R1) to conclude that $f_{K,\phi}$ varies continuously on its domain. In particular, the set
\begin{align*}
\mathcal{D}_{K}\left(\phi\right)  := \left\{ (x,z) \in \T \times \T : f_{K,\phi}(x,z) = \Delta_{K}(\phi) \right\}
\end{align*}
of \emph{$\Delta_{K}$-realizing pairs} is non-empty. 

\subsection*{Calculus of Distortion Pairs} The regularity and monotonicity properties of distortion kernels will allow us to establish some additional structure on the set $\mathcal{D}_{K}\left(\phi\right)$ of realizing pairs. To accomplish this, we begin by showing that the $K$-distortion $\Delta_{K}$ attains its unique minimum at the unit circle. Recall from \cite{Wirtinger} that if $\phi \in C^{0,1}(\T)$ has unit speed then the Poincar\'{e}-type inequality
$$
\fint_{\T} \frac{ |\phi(x+z) - \phi(x)|^{2} }{2(1-\cos z)} \, \rd x = 1 - h_{\phi}(z) \qquad h_{\phi}(z) \geq 0
$$
holds, and $h_{\phi} \equiv 0$ if and only if $\phi$ parametrizes the unit circle. For each $z \in \T, \,z \neq 0$ fixed the monotonicity property M1) gives
\begin{align*}
g\left( 2(1-\cos z),1 \right) &\leq g\left( 2(1-\cos z) , \fint_{\T} \frac{ |\phi(x+z) - \phi(x)|^{2} }{2(1-\cos z)} \, \rd x \right) \\
& \leq \fint_{\T} g\left( 2(1-\cos z) , \frac{ |\phi(x+z) - \phi(x)|^{2} }{2(1-\cos z)} \right)\, \rd x \leq \Delta_{K}(\phi)
\end{align*}
by convexity, Jensen's inequality and the definition of $\Delta_{K}(\phi)$ as a supremum. The monotonicity property M1) and the regularity property R2) show that function $q(\ell,0)$ is non-increasing and has $q_0(0)$ as its supremum. All together, we may conclude that
$$
q_{0}(0) = \Delta_{K}\left( \phi_{ {\rm circ}} \right) = \sup_{z \neq 0} \;\, q\left( 2(1-\cos z),0\right) = \sup_{z \neq 0} \; \, g\left( 2(1-\cos z),1 \right) \leq \Delta_{K}(\phi)
$$
holds for the $K$-distortion. 

Now assume that  $\phi \in C^{2,\alpha}(\T)$ is non-circular. By the monotonicity assumption M1) the function $q_{0}(\alpha)$ is non-increasing. Then $\kappa^{2}_{\phi}(x) > 1$ for at least one $x \in \T$ and so both inequalities
\begin{align*}
q_{0}\left( \frac1{12}\left( 1 - \kappa^{2}_{\phi}(x)\right) \right) \geq q_{0}(0) \quad \text{and} \quad q_{1}\left( \frac1{12}\left( 1 - \kappa^{2}_{\phi}(x)\right) \right) > 0
\end{align*}
hold by the second monotonicity assumption. By the limits \eqref{eq:diaglimits}, for any $\veps > 0$ there exists $\delta > 0$ so that
$$
\frac{ \partial_{z} f_{K,\phi}(x,u) }{2u} \geq (1-\veps)q_{1}\left( \frac1{12}\left( 1 - \kappa^{2}_{\phi}(x)\right) \right) 
$$
on $0 \leq u \leq \delta,$ so we must have
\begin{align}\label{eq:circunique}
f_{K,\phi}(x,\delta) &= f_{K,\phi}(x,0) + \int^{\delta}_{0} \frac{ \partial_{z} f_{K,\phi}(x,u) }{2u} 2u \, \rd u \\
&\geq q_{0}\left( \frac1{12}\left( 1 - \kappa^{2}_{\phi}(x)\right) \right) + (1-\veps)q_{1}\left( \frac1{12}\left( 1 - \kappa^{2}_{\phi}(x)\right) \right)\delta^{2} > q_{0}(0) \nonumber
\end{align}
and therefore the inequality $\Delta_{K}(\phi) > \Delta_{K}(\phi_{ {\rm circ}} )$ follows. In other words, $\Delta_{K}$ attains its minimum uniquely at the unit circle. All together, we have
\begin{lemma}\label{lem:distmin}
Assume that the kernel $K(u,v)$ satisfies the regularity properties {\rm R1), R2), M1), M2)} and that $\phi \in C^{2,\alpha}(\T)$ has unit speed. Then the inequality
$$
\Delta_{K}\left(\phi\right) > \Delta_{K}\left( \phi_{ {\rm circ}} \right) = q_0(0)
$$
holds unless $\phi$ parametrizes the unit circle.
\end{lemma}

We now turn our attention to the structure of the set
$$
\mathcal{D}_{K}\left(\phi\right)  := \left\{ (x,z) \in \T \times \T : f_{K,\phi}(x,z) = \Delta_{K}(\phi) \right\}
$$
of $\Delta_{K}$-realizing pairs in the interesting case where $\phi$ is non-circular. We first show that no point of the form $(x,0)$ realizes the supremum, and that the ratio
$$
1 \geq r_{\phi}(x_0,z_0) := \frac{|\phi(x) - \phi(x+z)|^2}{2(1-\cos z)}
$$
does not exceed unity at any such pair. To see this, let $(x_{0},z_{0}) \in \mathcal{D}_{K}(\phi)$ and assume that $z_{0} = 0$ for the sake of contradiction. Then at a point $x_{*} \in \T$ of maximal squared curvature we must have $\kappa^{2}_{\phi}(x_*) > 1$ since $\phi$ is non-circular, and moreover that
$$
\Delta_{K}(\phi) = f_{K,\phi}(x_0,0) = q_{0}\left( \frac1{12}\left(1 - \kappa^{2}_{\phi}(x_0) \right) \right) \leq q_{0}\left( \frac1{12}\left(1 - \kappa^{2}_{\phi}(x_*)\right) \right) = f_{K,\phi}(x_{*},0)
$$
since $q_{0}$ is a non-increasing function. The fact that $\kappa^{2}_{\phi}(x_*) > 1$ and the inequality \eqref{eq:circunique} combine to show that
$$
f_{K,\phi}(x_{*},\delta) > f_{K,\phi}(x_{*},0) \geq f_{K,\phi}(x_{0},0) = \Delta_{K}(\phi)
$$
must hold for all $\delta > 0$ small enough. This contradiction shows $z_0 \neq 0$ whenever $(x_0,z_0)$ form a $\Delta_{K}$-realizing pair. Similarly, if $r_{\phi}(x_0,z_0) > 1$ then M1) would imply the contradiction
$$
\Delta_{K}(\phi) = g\left( 2(1-\cos z_0), r_{\phi}(x_0,z_0) \right) \leq g\left( 2(1-\cos z_0) , 1 \right) \leq q_{0}(0) = \Delta_{K}(\phi_{ {\rm circ}} ),
$$
and so $r_{\phi}(x_0,z_0) \leq 1$ must also hold at any realizing pair. In particular, by R1) the kernel $K(u,v)$ is continuously differentiable at any realizing pair.

With these facts in hand, we may establish a few derivative identities that hold at distortion points. These identities, while straightforward, prove essential when we turn our attention to estimating the evolution of distortion-like quantities in time. So, fix $(x,z) \in \mathcal{D}_{K}(\phi)$ any $\Delta_{K}$-realizing pair. As $z \neq 0$ the function $f_{K,\phi}(x,z)$ is differentiable in a neighborhood of any such pair. If we make the change of variables
$$
u(x,y) = |\phi(x) - \phi(y)|^{2}, \quad v(x,y) = 2\left( 1 - \cos(y-x) \right) \quad \text{and} \quad f(x,y) := f_{K,\phi}(x,y-x)
$$
then the derivative identities
\begin{align}\label{eq:deronce}
f_x = K_u u_{x} + K_v v_x \qquad \text{and} \qquad f_y = K_u u_y + K_v v_y
\end{align}
follow immediately from the chain rule. In particular, we have
\begin{align}\label{eq:distpointder1}
u_x = -\frac{K_v}{K_u} v_x \qquad \text{and} \qquad u_y = -\frac{K_v}{K_u} v_y
\end{align}
at any $\Delta_{K}$-realizing pair. Let $\vv(x,y) \propto \phi(y) - \phi(x)$ denote the unit vector that defines the chord between a realizing pair. For planar curves $\phi$ we may use  \eqref{eq:distpointder1} to write
$$
\left.
\begin{aligned}
\dot \phi(x) = c \vv(x,y) + s \vv^{\perp}(x,y) \\
\dot \phi(y) = c \vv(x,y) \pm s \vv^{\perp}(x,y)
\end{aligned}
\quad \right\}
\qquad c^{2} + s^{2} = 1 \qquad \text{and} \qquad 2c^{2} = \frac{K^2_v}{K^2_u}\frac{v(4-v)}{2u}
$$
for $c,s \in \R$ appropriate constants; we must take the minus branch at a supremum. To see this, assume $\dot \phi(x) = \dot \phi(y)$ for the sake of contradiction. As $\phi$ forms a closed, embedded curve, we may find a point $u \in \T$ so that $\phi(u)$ intersects the open line segment
$$
L = \left\{ \phi(x) + t \vv(x,y) : 0 < t < 1 \right\}
$$
by applying the mean value theorem applied to the signed-distance function. Upon relabelling coordinates and using periodicity,  we may assume $x < u < y$ and that $0 < y-x \leq \pi$ without loss of generality. After defining $d_{a,b} := |\phi(a) - \phi(b)|, \ell_{a,b} := (b-a)$ and $r_{a,b} := d^{2}_{a,b}/2(1-\cos \ell_{a,b})$ for $a < b$ we conclude that the identities
$$
d_{x,y} = d_{x,u} + d_{u,y} \qquad \text{and} \qquad \ell_{x,y} = \ell_{x,u} + \ell_{u,y} 
$$
must hold. But then
$$
\Delta_{K}(\phi) = g\left( 2(1 - \cos \ell_{x,y}), r_{x,y} \right) \geq g\left( 2(1-\cos \ell_{x,u}) , r_{x,u} \right)
$$
by definition of the $K$-distortion. As $r_{x,y} \leq 1$ and $0 \leq \ell_{x,y} \leq \pi$ the monotonicity property M1) then implies
$$
g\left( 2(1-\cos \ell_{x,u}) , r_{x,u} \right) \leq g\left( 2(1 - \cos \ell_{x,y}), r_{x,y} \right) \leq g\left( 2(1 - \cos \ell_{x,u}), r_{x,y} \right),
$$
and so $r_{x,u} \geq r_{x,y}$ by appealing to the strict monotonicity hypothesis M1) once again. This happens only if
\begin{align*}
&d_{x,u} \sin\left( \frac{\ell_{x,u} + \ell_{u,y}}{2} \right) \geq \left( d_{x,u} + d_{u,y} \right) \sin \left( \frac{\ell_{x,u}}{2} \right) \quad \longrightarrow \\
&d_{x,u} \sin\left( \frac{\ell_{x,u}}{2} \right)\left( \cos \left( \frac{\ell_{u,y}}{2} \right) - 1\right)  + d_{x,u}  \sin\left( \frac{\ell_{u,y}}{2} \right)\cos\left( \frac{\ell_{x,u}}{2}\right) \geq d_{u,y} \sin \left( \frac{\ell_{x,u}}{2} \right) \quad \longrightarrow \\
&d_{x,u} \sin \left( \frac{\ell_{u,y}}{2} \right) > d_{u,y} \sin \left( \frac{\ell_{x,u}}{2} \right)
\end{align*}
and so $r_{x,u} > r_{u,y}$ must hold. A similar argument shows $r_{u,y} > r_{x,u}$ must hold as well, a contradiction. Thus the set of equalities
$$
\left.
\begin{aligned}
\dot \phi(x) = c \vv(x,y) + s \vv^{\perp}(x,y) \\
\dot \phi(y) = c \vv(x,y) - s \vv^{\perp}(x,y)
\end{aligned}
\quad \right\}
\qquad c^{2} + s^{2} = 1 \qquad \text{and} \qquad 2c^{2} = \frac{K^2_v}{K^2_u}\frac{v(4-v)}{2u}
$$
must hold at any $\Delta_{K}$-realizing pair. Finally, if we further assume $g(\ell,r)$ is $C^{2}$ for $\ell,r>0$ then differentiate \eqref{eq:deronce} once more and combine the trigonometric identities
\begin{align*}
v_{x} = -v_{y}, v_{xx} = v_{yy} = -v_{xy}, \quad  v_{xx} = 2 - v \quad \text{and} \quad v^2_{x} = v(4-v)
\end{align*}
with the critical point relation \eqref{eq:distpointder1} to find that the second-derivative identities
\begin{align}\label{eq:2ndderiv}
f_{xx} &= \left( \frac{K_{uu} K^2_v}{K^2_u} - 2 \frac{K_{uv}K_v}{K_u} + K_{vv}\right)v(4-v) + K_{v}(2-v) + 2 K_u\left( 1 + \langle \phi(x) - \phi(y),\ddot \phi(x)\rangle\right) \nonumber \\
f_{yy} &= \left( \frac{K_{uu} K^2_v}{K^2_u} - 2 \frac{K_{uv}K_v}{K_u} + K_{vv}\right)v(4-v)  +  K_{v}(2-v)+ 2 K_u\left( 1 + \langle \phi(y) - \phi(x),\ddot \phi(y)\rangle\right) \\
f_{xy} &= \left( \frac{K_{uu} K^2_v}{K^2_u} - 2 \frac{K_{uv}K_v}{K_u} + K_{vv}\right)v(v-4) + K_{v}(v-2) + K_u \left(2 - \frac{K^2_v}{K^2_u}\frac{v(4-v)}{u} \right) \nonumber
\end{align}
must hold at a $\Delta_{K}$-realizing maximum. In particular, the inequality
\begin{align}\label{eq:2ndineq}
f_{xx}(x,y) + f_{yy}(x,y) \leq - \left(  (f_{xx}-f_{yy})^2 + 4 f_{xy}^2  \right)^{\frac12}(x,y) \leq - 2|f_{xy}|(x,y)
\end{align}
must hold by the $2^{{\rm nd}}$ derivative test. The following lemma summarizes these conclusions.

\begin{lemma}[Distortion Point Calculus]\label{lem:distcalc} Assume that $K$ satisfies {\rm R1), R2), M1), M2)} and that $\phi \in C^{2,\alpha}(\T)$ has unit speed. 
If $q_0(0) < \Delta_{K}(\phi) < \infty$ and $(x,z) \in \mathcal{D}_{K}(\phi)$ then the following hold ---
\begin{enumerate}[{\rm i)}]
\item The ratio
$$
r = \frac{u}{v} =: \frac{ |\phi(y) - \phi(x)|^2 }{2\left(1 - \cos z\right)}
$$
obeys the bounds $0 < r \leq 1,$ and arc-length distance $|z| = |y-x|$ between $\phi(x)$ and $\phi(y) = \phi(x+z)$ does not vanish. 
\item The tangent-angle and normal-angle relations
\begin{align*}
\langle \vv(x,y), \tang(y) \rangle &= \langle \vv(x,y), \tang(y) \rangle = -\frac{K_v}{K_u}\frac{\sin z}{\sqrt{u}} && \vv(x,y) = \frac{\phi(y) - \phi(x)}{|\phi(y) - \phi(x)|}\\
-\langle \vv(x,y) , \nrml(y) \rangle &= \langle \vv(x,y), \nrml(x) \rangle = \pm\sqrt{1 - \frac{K^2_v}{K^2_u}\frac{v(4-v)}{4u}}
\end{align*}
hold for a planar curve.
\item If, in addition, $g(\ell,r)$ is twice differentiable for $\ell,r>0$ then second derivative identities (\ref{eq:2ndderiv},\ref{eq:2ndineq}) hold.
\end{enumerate}
\end{lemma}

\subsection*{Distortion Measures:} To finish up, we shall prove the identity \eqref{eq:repr} for the dynamics of kernel distortions along a one-parameter family $\Phi := \{ \phi(t) : t \in [a,b] \}$ of unit speed embeddings. We may freely assume that no $\phi(t)$ is circular, for we know the value $\Delta_{K} = q_0(0)$ of the kernel distortion exactly for all such circular points. Given a distortion kernel $K$ and a unit speed embedding $\phi \in C^{2,\alpha}(\T)$ we may let
$$
\mathfrak{Z}_{*}\left( \phi \right) := \min \left\{ |z| : (x,z) \in \mathcal{D}_{K}(\phi) \right\}
$$
denote the smallest distance along $\T$ between any distortion realizing pair. As $f_{K,\phi}(x,z)$ varies continuously when $\phi \in C^{2,\alpha}(\T)$ this minimum is well-defined. Moreover, the uniform convergence \eqref{eq:unifcv} suffices to show lower semi-continuity
$$
\| \phi_{j} - \phi \|_{C^{2,\alpha}(\T)} \to 0 \qquad \longrightarrow \qquad \liminf_{j \to \infty} \; \, \mathfrak{Z}_{*}\left( \phi_j \right) \geq \mathfrak{Z}_{*}\left( \phi \right)
$$
of this distance with respect to the $C^{2,\alpha}(\T)$ topology. As a consequence, if the one-parameter family $\Phi$ of unit speed embeddings varies continuously in $C^{2,\alpha}(\T)$ then the minimum
$$
\mathfrak{Z}_{*}\left( \Phi \right) := \min_{t \in [a,b]} \;\, \mathfrak{Z}_{*}\left( \phi(t) \right) 
$$
is attained at some $a \leq t_{*} \leq b,$ and is non-vanishing since $\phi(t_*)$ is non-circular. Moreover, as $\Delta_{K}(\phi(t))$ varies continuously on $[a,b]$ the minimum
$$
\mathfrak{E}_{*}\left( \Phi \right) := \min_{t \in [a,b]} \;\, \Delta_{K}\left( \phi(t) \right)
$$
is also attained and therefore finite.

Let us now turn our attention toward computing the evolution of the $K$-distortions $\Delta_{K}(\phi(t))$ in time, which involves the integration formula
\begin{align}\label{eq:measuretwo}
\frac{ \Delta_{K}\left( \phi(b) \right) - \Delta_{K}\left( \phi(a) \right)}{2} &\;=  \int^{b}_{a} \left( \int_{\T} K_{u}\left( |\delta \phi(x,z,t)|^2 , 2(1-\cos z) \right)\langle \delta \phi(x,z,t),\delta \phi_t(x,z,t)\rangle \, \rd \pi_{t} \right) \, \rd t \nonumber \\
\delta f(x,z,t) &:= f(x+z,t) - f(x,t)
\end{align}
against a corresponding one-parameter family $\pi_{t} : [a,b] \mapsto \mathcal{P}( \T \times \T)$ of probability measures. By such a \emph{distortion measure} for $\phi$ we mean a Borel probability measure $\pi$ on $\T \times \T$ that has support $\mathrm{supp}(\pi) \subset \mathcal{D}_{K}\left(\phi\right)$ on $\Delta_{K}$-realizing pairs. Under the further assumption that the time derivatives
$$
\Phi^{\prime} = \left\{ \phi_{t} : t \in [a,b] \right\}
$$
vary continuously, we may proceed with a relatively straightforward construction of just such a one-parameter family $\pi_t$ of distortion measures.
\begin{theorem}[Existence of Distortion Measures]\label{thm:dist}
Assume that $K$ satisfies {\rm R1), R2)} and that $\Phi \in C\left( [a,b] ; C^{2,\alpha}(\T) \right)$ defines a one-parameter family of unit speed embeddings. Assume the time derivatives $\Phi^{\prime} \in C\left([a,b];C(\T)\right)$ define a one-parameter family of continuous functions, and that $\mathfrak{Z}_{*}(\Phi)$ does not vanish. Then there exists a family $\pi_{t} : [a,b] \mapsto \mathcal{P}( \T \times \T)$ of distortion measures for which \eqref{eq:measuretwo} holds.
\end{theorem}
\begin{proof}
The regularity and embeddedness assumptions combine with \eqref{eq:unifcv} to show that the function
$$
f(x,z,t) = f_{K,\phi(t)}(x,z) := K\left( |\phi(x+z,t) - \phi(x,t)|^{2} , 2(1 - \cos z) \right)
$$
is jointly continuous in all of its arguments. In particular, $\mathfrak{Z}_{*}(\Phi)$ is well-defined and positive by assumption. Similarly, the kernel distortions
$$
\Delta_{K}\left( \phi(t) \right) \geq \mathfrak{E}_{*}(\Phi) > -\infty
$$
are uniformly bounded below. So, take any fixed $\lambda < \mathfrak{E}_{*}(\Phi)$ and define the $L^{2p}$-regularization
$$
E_{p}(x,z,t) := \left( f_{K,\phi(t)}(x,z)  - \lambda \right)^{2p}_{+}\chi^{2p}(z) \quad \text{and} \quad E_{p}(t) := \left[ \int_{\T \times \T} E_{p}(x,z,t) \, \rd x \rd z \right]^{\frac1{2p}} 
$$
of the kernel distortion, with $\chi(z)$ some continuous cut-off function with 
$$
\begin{cases}
\chi(z) = 1 & \text{on} \qquad \mathfrak{Z}_{*}\left( \Phi \right) \leq |z|,\\
\chi(z) = 0 & \text{on} \qquad \mathfrak{Z}_{*}\left( \Phi \right) \geq 2|z|
\end{cases}
$$
and $0 < \chi(z) < 1$ otherwise. By construction of $E_p(x,z,t)$ and the definitions of $\mathfrak{Z}_{*}\left( \Phi \right),\chi(z)$ it follows that
\begin{align*}
E_{p}(t)  \quad \overset{ p \to \infty }{\longrightarrow}  \quad
\sup_{ (x,z) \in \T \times \T } \;\, (f_{K,\phi(t)}(x,z) - \lambda)_{+}\chi(z) = \Delta_{K}\left(\phi(t)\right) - \lambda =: E_{\infty}(t)
\end{align*}
since $\mathfrak{Z}_{*}\left( \Phi \right)$ does not vanish. Finally, note that for $p > 1$ the integral identity
$$
\int^{b}_{a} \int_{\T \times \T} \zeta(x,z,t) \, \rd  \nu_{p} := \fint^{b}_{a} \int_{\T \times \T} \frac{ \left( f_{K,\phi(t)}(x,z)  - \lambda \right)^{2p-2}_{+}\chi^{2p-2}(z) }{\left[ E_{p-1}(t)\right]^{2p-2}}\zeta(x,z,t) \, \rd x \rd z \rd t,
$$
on $\zeta \in C(\T \times \T \times [a,b])$ defines a sequence $\nu_{p}  \in \mathcal{P}(\T \times \T \times [a,b] )$ of Borel probability measures.

Now as $\phi_{t} \in C(\T \times [a,b])$ is jointly continuous by assumption, the differentiation formula
$$
\partial_{t} E_{p}(x,z,t) = 4p\left( f_{K,\phi(t)}(x,z)  - \lambda \right)^{2p-1}_{+}K_{u}\left(  |\delta \phi(x,z,t)|^{2},v(z) \right) \langle \delta \phi(x,z,t) , \delta \phi_{t}(x,z,t) \rangle \chi^{2p}(z)
$$
is valid since, by R1), the composition $K\left(  |\delta \phi(x,z,t)|^{2},v(z) \right)$ is differentiable in regions where $z$ does not vanish. As a consequence, the fundamental theorem of calculus gives
\begin{align}\label{eq:FTCdist}
\frac{ E_{p}(b) - E_{p}(a) }{2(b-a)} &\;= \int_{\T \times \T \times [a,b]} R_{p}(t)h(x,z,t)  \, \rd \nu_{p} \qquad \qquad R_{p}(t) := \frac{ E^{2p-2}_{p-1}(t) }{E^{2p-1}_{p}(t)}  \\
\qquad h(x,z,t) &:= \left( f_{K,\phi(t)}(x,z)  - \lambda \right)_{+}K_{u}\left(  |\delta \phi(x,z,t)|^{2},v(z) \right) \langle \delta \phi(x,z,t) , \delta \phi_{t}(x,z,t) \rangle \chi^{2}(z) \nonumber
\end{align}
for the evolution of the regularized kernel distortion. As $\phi_{t} \in C(\T \times [a,b])$ and $z$ remains bounded away from the origin unless $h$ vanishes, the regularity assumption R1) and the embeddedness assumption ensure that the function $h(x,z,t)$ is jointly continuous in all of its variables and uniformly bounded. An application of H\"{o}lder's inequality shows that
$$
\frac{ E^2_{p}(t) }{E^{2}_{\infty}(t)  } \leq E_{p}(t)R_{p}(t) \leq (2\pi)^{\frac1{p}},
$$
and so $R_{p}(t)E_{p}(t) \to 1$ pointwise. Recall that the function
$$
f(x,z,t) = \left( f_{K,\phi(t)}(x,z) - \lambda \right)^{2}_{+} \chi^{2}(z)
$$
is continuous in all of its arguments. As such, there exists some uniform $\delta>0$ so that the implication
$$
|x-y|^2 + |z - u|^2 < \delta^2 \qquad \longrightarrow \qquad |f(x,z,t) - f(y,u,t)| < \frac{ \left[ \mathfrak{E}_{*}(\Phi) - \lambda \right]^{2}}{2} \leq \frac{ E^{2}_{\infty}(t) }{2}
$$
holds. Applying this at a point $(y,u,t)$ realizing the maximum $f(y,u,t) = E^{2}_{\infty}(t)$ gives the uniform lower bound
\begin{align*}
E^{2p}_{p}(t) &= \int_{\T \times \T} f^{p}(x,z,t) \, \rd x \rd z \geq \int_{ B_{\delta}(y,u) } f^{2p}(x,z,t) \, \rd x \rd z \geq \frac{ E^{2p}_{\infty}(t) }{2^{2p}} \pi \delta^{2} \\
E_{p}(t) & \geq \frac{ \mathfrak{E}_{*}(\Phi) - \lambda}{2}\left( \pi \delta^{2} \right),
\end{align*}
and so $R_{p}(t)$ is uniformly bounded in time. From the definition of $\nu_{p}$ the limit
$$
\left| \int_{\T \times \T \times [a,b]} \left(R_{p}(t) - E^{-1}_{\infty}(t) \right)h(x,z,t)  \, \rd \nu_{p}\right| \leq \frac{ \|h\|_{L^{\infty}(\T \times \T \times [a,b])} }{b-a} \int^{b}_{a} \left|R_{p}(t) - E^{-1}_{\infty}(t)\right| \, \rd t \to 0
$$
therefore follows from the dominated convergence theorem. As the measures $\nu_{p}$ defined on Borel subsets of $\T \times \T \times [a,b]$ are clearly tight probability measures, there exists a subsequence of $p \in \N$ and a limit measure $\nu \in \mathcal{P}\left( \T \times \T \times [a,b] \right)$ defined on the Borel sets so that the weak-$*$ convergence
$$
\int_{\T \times \T \times [0,T] } \phi(x,y,t) \, \rd \nu_{p} \to \int_{\T \times \T \times [0,T] } \phi(x,y,t) \, \rd \nu  \quad \text{for all} \quad \phi \in C(\T \times \T \times [0,T])
$$
holds. Passing to the limit in \eqref{eq:FTCdist} then gives the evolution identity
$$
\frac{\Delta_{K}\left( \phi(b) \right) - \Delta_{K}\left( \phi(a) \right)}{2(b-a)} = \int_{\T \times \T \times [a,b] } E^{-1}_{\infty}(t)h(x,z,t) \, \rd \nu
$$
for the kernel distortion.

By the disintegration theorem (c.f. \cite{AFP} Thm 2.28 or \cite{AGS} Thm 5.3.1), there exists a probability measure $\mu$ on $[a,b]$ and a $\mu$-a.e. unique family of Borel probability measures $\pi_{t} : [a,b] \mapsto \mathcal{P}(\T \times \T)$ so that if $\phi \in L^{1}( \rd \nu)$ is Borel then the identity
$$
\int_{\T \times \T \times [a,b] } \phi(x,y,t) \, \rd \nu =  \int^{b}_{a} \left( \int_{\T \times \T} \phi(x,y,t) \, \rd \pi_t \right) \, \rd \mu
$$
holds. Take $\phi(x,y,t) = \zeta(t)$ for $\zeta$ any continuous function. Then by the definition of $\nu_{p}$ and weak-$*$ convergence
\begin{align*}
\fint^{b}_{a} \zeta(t) \, \rd t = \int_{\T \times \T \times [a,b] } \phi(x,y,t) \, \rd \nu_{p} \quad \to \quad  \int_{\T \times \T \times [a,b] } \phi(x,y,t) \, \rd \nu = \int^{b}_{a} \zeta(t) \left( \int_{\T \times \T} \rd \pi_{t}\right) \, \rd \mu 
\end{align*}
and so $\mu$ agrees with the normalized Lebesgue measure. Define the subset
$$
A:= \left\{ (x,z,t) \in \T \times \T \times [a,b] : f_{K,\phi(t)}(x,z) = \Delta_{K}\left( \phi(t) \right) \right\}
$$
and note that $A$ is closed since $f_{K,\phi(t)}(x,z)$ is jointly continuous in all of its variables. So for any point $(x_0,z_0,t_0) \in A^{c}$ there exists some $\delta > 0$ small enough so that
\begin{align*}
B &:= \left\{ (x,z,t) \in \T \times \T \times [a,b] :  |x - x_0| + |z - z_0| + |t-t_0| \leq \delta \right\} \subset A^{c} \\
r &:= \max_{B} \;\, \frac{ f_{K,\phi(t)}(x,z) - \lambda }{\Delta_{K}(\phi(t)) - \lambda} < 1,
\end{align*}
which upon applying H\"{o}lder's inequality and passing to the limit gives
\begin{align*}
\nu\left( \mathrm{int}(B) \right) &\leq \liminf_{p \to \infty} \nu_{p}\left( B \right) = \liminf_{p \to \infty}\;\, \frac1{b-a} \int_{B} \frac{ \left( f_{K,\phi(t)}(x,z)  - \lambda \right)^{2p-2}\chi^{2p-2}(z) }{\left[ E_{2p-2}(t)\right]^{2p-2}} \, \rd x \rd z \rd t \\
& \leq r^{q} \liminf_{p \to \infty} \;\, (2\pi)^{\frac{q}{2p-2}}\fint^{b}_{a} \frac{ E^{q}_{\infty}(t) }{E^{q}_{p-1}(t)} \, \rd t = r^{q}
\end{align*}
for any $q>1$ fixed. Sending $q \to \infty$ shows that $\nu\left( \mathrm{int}(B) \right) = 0,$ so $(x_0,y_0,t_0) \notin \mathrm{supp}(\nu)$ and the identity
\begin{align*}
\int_{\T \times \T \times [a,b] } \phi(x,z,t) \, \rd \nu &= \int_{\T \times \T \times [a,b] } \phi(x,z,t)\mathbf{1}_{A}(x,z,t) \, \rd \nu \\
&= \int^{b}_{a} \left( \int_{\T \times \T} \phi(x,z,t) \mathbf{1}_{A}(x,z,t) \, \rd \pi_t \right) \, \rd \mu
\end{align*}
therefore holds. For almost every $a \leq t \leq b$ the restricted measures
$$
\tilde{\pi}_{t} := \pi_{t} \mathbf{1}_{A}
$$
on $\T \times \T$ are therefore distortion measures. All-together, after relabelling $\tilde \pi_t$ to $\pi_t$ this gives a one-parameter family $ \pi_t : [a,b] \mapsto \mathcal{P}(\T \times \T)$ of distortion measures so that the claimed evolution identity
$$
\Delta_{K}\left( \phi(b) \right) - \Delta_{L}\left( \phi(a) \right) = \int^{b}_{a} \left( \int_{\T \times \T} 2 K_{u}\left(  |\delta \phi(x,z,t)|^{2},v(z) \right) \langle \delta \phi(x,z,t) , \delta \phi_{t}(x,z,t) \rangle \, \rd \pi_t \right) \, \rd t
$$
holds for the kernel distortion.
\end{proof}
\begin{remark}
Our formulation of theorem \ref{thm:dist} isolates a reasonably generic set of structural properties on $K$ and $\Phi$ whose imposition will guarantee the existence of distortion measures. In many cases we may discharge these assumptions easily; we simply impose {\rm R1), R2), M1), M2)} on the kernel together with the geometric assertions that all $\phi(t)$ are embedded and non-circular. Nevertheless, in certain circumstances we might guarantee $\mathfrak{Z}_{*}(\Phi) > 0$ ``by hand,'' weaken the monotonicity properties {\rm M1), M2)} and yet still have an evolution formula for the corresponding $K$ distortion.
\end{remark}
\subsection*{Applications}
We may now proceed to leverage this result to study embeddedness. We will first give a quick proof of the Huisken monotonicity formula using distortion measures, which will serve as a point of contrast. We will then examine how the addition of a normal driving force can, even in the absence of non-locality, induce a finite-time lack of embeddedness.

To begin, recall that we have the evolution
$$
\phi_{t} = \veps(t)\left( \ddot \phi + \po\big( a \big)\dot \phi + \mu_{a} \phi \right) \qquad a = |\ddot \phi|^2
$$
for a pure motion by mean curvature. On an interval $[0,T]$ with $\phi$ embedded and non-circular, we may take the pseudo-distortion kernel $K(u,v) = vu^{-1}$ and apply theorem \ref{thm:dist} to conclude
$$
\frac{\Delta_{K}\left( \phi(T) \right) - \Delta_{K}\left( \phi(0) \right)}{2} = - \int^{T}_{0} \Delta_{K}\left( \phi(t) \right) \left( \int_{\T \times \T} \frac{ \langle \delta \phi(x,z,t) , \delta \phi_{t}(x,z,t) \rangle }{|\delta \phi(x,z,t)|^{2}} \, \rd \pi_{t} \right) \rd t
$$
for $\pi_t$ some family of distortion measures. By the distortion point calculus (c.f. lemma \ref{lem:distcalc}), we must have
\begin{align*}
\frac{\langle \delta \phi(x,z,t), \dot \phi(x+z,t) \rangle}{|\delta \phi(x,x+z,t)|^2} = \frac{\langle \delta \phi(x,z,t), \dot \phi(x,t) \rangle}{|\delta \phi(x,x+z,t)|^2} = \frac{\sin z}{2(1-\cos z)} \quad\;\text{and} \quad\;
-\frac{ \langle \delta \phi(x,z,t), \ddot \delta \phi(x,z,t) \rangle }{|\delta \phi(x,z,t)|^{2}} \leq 1
\end{align*}
on the set of distortion points. We therefore have the lower bound
$$
\frac1{\veps(t)}\frac{ \langle \delta \phi(x,z,t) , \delta \phi_{t}(x,z,t) \rangle }{|\delta \phi(x,z,t)|^{2}} \geq \frac{\sin z}{2(1-\cos z)} \int^{x+z}_{x} |\ddot \phi(u,t)|^{2} \, \rd u + \mu_{a(t)}\left( 1 - \frac{z \sin z}{2(1- \cos z)} \right) - 1
$$
on the support of any distortion measure. Now as $\Delta_{K}(\phi(t)) \geq 1 = q_0(0)$ we know
\begin{align*}
\cos\left( \vth(x+z,t) - \vth(x,t) \right) &= \langle \dot \phi(x+z,t), \dot \phi(x,t) \rangle = \frac{1 + \cos(z)}{\Delta_{K}\left( \phi(t) \right)} - 1 \\
|z| \leq \cos^{-1}\left( \frac{1 + \cos(z)}{\Delta_{K}\left( \phi(t) \right)} - 1 \right) &= \left| \int^{x+z}_{x} \dot \vth(u,t) \, \rd u \right| \leq |z|^{\frac12} \left| \int^{x+z}_{x} |\ddot \phi(u,t)|^{2} \, \rd u \right|^{\frac12}
\end{align*}
and so we may conclude that the lower bound
$$
\frac1{\veps(t)}\frac{ \langle \delta \phi(x,z,t) , \delta \phi_{t}(x,z,t) \rangle }{|\delta \phi(x,z,t)|^{2}} \geq \left( \mu_{a(t)} - 1 \right)\left( 1 - \frac{z \sin z}{2(1- \cos z)} \right)
$$
holds on the support of any distortion measure. In addition, we have the Poincar\'{e} inequality
$$
1 = \fint_{\T} |\dot \phi(u,t)|^{2} \, \rd u \leq \fint_{\T} |\ddot \phi(u,t)|^{2} \, \rd u = \mu_{a(t)}
$$
and we know that for $0 \leq z \leq \pi$ the function
$$
h(z) := \left( 1 - \frac{z \sin z}{2(1- \cos z)} \right) 
$$
is non-negative and increasing. Finally, for $(x,z) \in \mathcal{D}_{K}(\phi(t))$ we may use the expansion
\begin{align*}
\langle \phi(x+z) - \phi(x),\dot \phi(x)\rangle = z - R(x,z) \qquad \qquad R(x,z) &= \int^{x+z}_{x} \int^{u}_{x}(x+z - u)\langle \ddot \phi(v), \ddot \phi(u) \rangle \, \rd v \rd u \\
|R(x,z)| & \leq \frac{ z^2}{\sqrt{12}} \int^{x+z}_{x} |\ddot \phi(u,t)|^{2} \, \rd u
\end{align*}
together with lemma \ref{lem:distcalc}, ii) to show that the separation $|z|$ between any realizing pair obeys the lower bound
\begin{align}\label{eq:distlwr}
|z| \geq \sqrt{12}\left(1 - \frac{1}{\Delta_{K}(\phi(t))} \right)\left( \int^{x+z}_{x} |\ddot \phi(u,t)|^{2} \, \rd u \right)^{-1} \geq \frac{\left(1 - \frac{1}{\Delta_{K}(\phi(t))} \right)}{2 \mu_{a(t)}},
\end{align}
and so all together we obtain the inequality
$$
\Delta_{K}\left( \phi(T) \right) \leq \Delta_{K}\left( \phi(0) \right) - \int^{T}_{0} \veps(t)\left( \mu_{a(t)} - 1 \right) \Delta_{K}\left( \phi(t) \right) h\left( \frac{\left(1 - \frac{1}{\Delta_{K}(\phi(t))} \right) }{2 \mu_{a(t)}} \right) \, \rd t
$$
for the pseudo-distortion. In particular, the pseudo-distortion decreases monotonically as long as $\phi(t)$ remains non-circular.

\begin{figure}[t]
\centering
\includegraphics[height=2in]{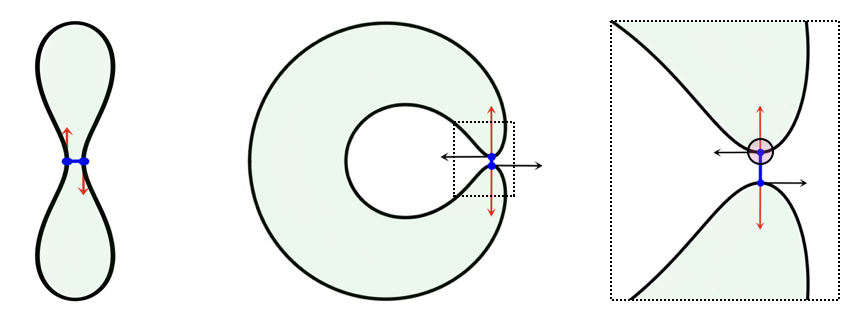}
\caption{Examples of distortion realizing pairs. At left, the interior chord leads defines a negatively ($-$) oriented $o(x,z) = -1$ pair, while at center the exterior chord leads to a positively ($+$) oriented $o(x,z) = 1$ pair. the orientation. The magnification at right shows a typical length-scale at which a nonlocal self-repulsion acts. For the configuration at left, the nonlocal self-interaction further drives a thinning of the neck and leads to a loss of embeddedness. For the configuration at center, a strong enough repulsion (relative to the size of the normal growth) prevents self-intersections and preserves the embeddedness of the interface.}\label{fig2}
\end{figure}

The situation changes rather markedly once we include either normal growth or non-local forcing. We illustrate this by investigating embeddedness for the evolution
\begin{align*}
\phi_{t} &= \veps(t)\left( \ddot \phi + \po\big( a \big)\dot \phi + F\dot \phi^{\perp} + \mu_{a} \phi \right) \qquad &&\;\,\qquad a = |\ddot \phi|^2 + \langle \ddot \phi,\dot \phi^{\perp} \rangle F \\
\rd_{t} \veps &= 2 \veps^{2} \mu_{a} \qquad \qquad 
&&F(x,t) = \veps^{-\frac12}(t)\left( c + f(x,t) \right)
\end{align*}
describing a motion by mean curvature together with an external normal forcing. The parameter $c<0$ reflects a constant normal growth force, while the generic function $f(x,t)$ stands in for the contribution of a non-local forcing. We select the orientation of $\dot \phi^{\perp}$ so that it represents the inward normal $\nrml$ to the region enclosed by the embedding. As before, we may write the pseudo-distortion evolution
$$
\Delta_{K}\left( \phi(T) \right) = \Delta_{K}\left( \phi(0) \right) - \int^{T}_{0} \Delta_{K}\left( \phi(t) \right)\left( \int_{\T \times \T} \frac{ \langle \delta \phi(x,z,t) , \delta \phi_{t}(x,z,t) \rangle }{|\delta \phi(x,z,t)|^{2}} \, \rd \pi_{t} \right) \rd t
$$
for such an interface in terms of some appropriate family $\pi_{t}$ of distortion measures. Applying the distortion point calculus once again leads to the identities
\begin{align*}
\frac{ \langle \delta \phi(x,z,t), \delta \ddot \phi(x,z,t)\rangle}{|\delta \phi(x,z,t)|^2} &= o(x,z) \sqrt{1 - \frac{1 + \cos z}{2 \Delta_{K}(\phi(t))}}\left( \frac{ \kappa(x+z,t) +\kappa(x,t) }{|\delta \phi(x,z,t)|} \right) \geq - 1
\\
\frac{ \langle \delta \phi(x,z,t), \delta P_{0}\big(a\big)\dot \phi(x,z,t) \rangle}{|\delta \phi(x,z,t)|^{2}} &= \frac{ \sin z }{2(1 - \cos z)}\left( \int^{x+z}_{x} a(u,t) \, \rd u \right) - \frac{ z \sin z}{2(1-\cos z)} \mu_{a(t)}\\
\veps^{\frac12}(t)\frac{\langle \delta \phi(x,z,t),\delta F \dot \phi^{\perp}(x,z,t)  \rangle}{|\delta \phi(x,z,t)|^{2}} &= o(x,z) \sqrt{1 - \frac{1 + \cos z}{2 \Delta_{K}(\phi(t))}}\left( \frac{ f(x+z,t) + f(x,t) + 2c }{|\delta \phi(x,z,t)|} \right),
\end{align*}
where $\kappa(x,t) = \dot \vth(x,t)$ denotes the normalized curvature of the interface and the choice of orientation $o(x,z) \in \{-1,1\}$ depends upon whether the chord $\vv(x,z,t) \propto \phi(x+z,t) - \phi(x,t)$ between a realizing pair $(x,z) \in \mathcal{D}_{K}(\phi(t))$ lies \emph{interior} ($-$) or  \emph{exterior} ($+$) to the region enclosed by the interface. If we let $k := \sqrt{\veps}\kappa$ denote the physical (i.e. unnormalized) curvature of the interface and
$$
v(x,t) = k(x,t) + c + f(x,t)
$$
the total normal velocity then we obtain
\begin{align}\label{eq:odekernel}
\frac{\rd }{\rd t} \Delta_{K}( t ) = &-\veps^{\frac12}(t)\Delta_{K}(t) \int_{\T \times \T} o(x,z)\sqrt{1 - \frac{1 + \cos z}{2 \Delta_{K}(t)}}\left(\frac{ v(x,t) + v(x+z,t) }{|\delta \phi(x,z,t)|}\right) \, \rd \pi_t \nonumber \\
& - \veps(t) \Delta_{K}( t ) \int_{\T \times \T}\left( \frac{\sin z \int^{x+z}_{x} a(u,t) \, \rd u}{2(1 - \cos z)} + h(z) \mu_{a(t)}\right)  \, \rd \pi_t
\end{align}
for the evolution of the pseudo-distortion. (Of course, we really mean that \eqref{eq:odekernel} holds in an integral sense.)

\begin{figure}[t]
\centering
\includegraphics[height=2in]{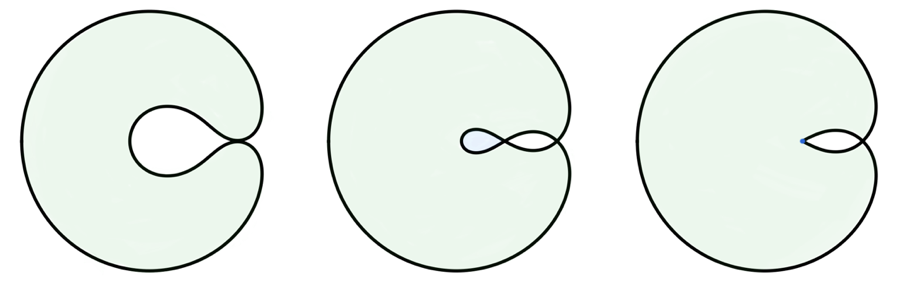}
\caption{Time slices of a numerical simulation for a boundary evolving solely under curvature and normal growth. A smooth, embedded initial condition eventually forms a self-intersection, loses embeddedness, then forms a singularity in finite time.}\label{fig3}
\end{figure}
The evolution \eqref{eq:odekernel} reveals a mechanism by which pure normal growth (i.e. $c<0$ and $f=0$) may cause a finite-time loss of embeddedness. By choosing initial conditions appropriately (as in fig xxx), on $[0,T]$ for $T>0$ small enough we may assume that the set of realizing pairs takes the form
$$
\mathcal{D}_{K}\left( \phi(t) \right) = \left\{ (x_0(t),z_0(t)) , (x_0 (t)+ z_0(t), - z_0(t) ) \right\},
$$
or in other words, that the supremum occurs at a unique pair $\phi(x_0(t)+z_0(t),t),\phi(x_0(t),t)$ of points along the interface. By selecting the initial condition appropriately, we may also arrange that the chord $\phi(x_0(t)+z_0(t),t) - \phi(x_0(t),t)$ between these points lies exterior to the region enclosed by the interface, and so $o(x,z) = 1$ on the set of realizing pairs. The physical curvature $k = \sqrt{\veps(t)}\kappa$ of the interface evolves according to
\begin{align*}
k_{t} = \veps(t)\left( \ddot k + \po\big( a \big) \dot k \right) + k^{2}v
\end{align*}
for as long as the solution exists, so in particular, the maximal curvature $k_{ {\rm max}}$ cannot increase as long as the velocity $v = k+c$ remains negative. If we take $c<0$ small enough to have a negative initial velocity $v_0$ then the inequalities
$$v(x,t) \leq v(x,0) \leq \max_{x \in \T} \; \, v(x,0) := v^{*}_{0} < 0 \qquad \text{and} \qquad \max_{x \in \T} \;\, k(x,t) \leq \max_{x \in \T}\;\, k(x,0) := k^{*}_{0} < |c|$$
will hold for all such times. Finally, by taking $c<0$ small enough we may assume that on $[0,T]$ the length of the interface increases. In other words, the interfacial growth condition 
$$
\mu_{a(t)} \leq 0 \qquad \qquad \longleftrightarrow \qquad \qquad \fint_{\T} \kappa^{2}(u,t)\, \rd u \leq |c|\sigma(t) \fint_{\T} \kappa(u,t) \, \rd u =  |c| \sigma(t)
$$
holds on $[0,T]$ for $T>0$ small enough. Appealing to the length evolution
$$
\sigma(t)\sigma'(t) = -\fint_{\T} \kappa^{2}(u,t)\, \rd u + |c|\sigma(t) \fint_{\T} \kappa(u,t) \, \rd u \leq \sigma(t)|c|
$$
then shows that the speed $\sigma(t) \leq \sigma_0 + |c|t$ of the interface grows no worse than linearly in time. Applying the lower bound \eqref{eq:distlwr} then shows
$$
\int_{\T} \left( \frac{\sin z \int^{x+z}_{x} a(u,t) \, \rd u}{2(1 - \cos z)} \right) \, \rd \pi_{t} \geq -\mathfrak{c}_{*}\veps^{-\frac12}(t) \qquad \text{on} \qquad [0,T]
$$
for $\mathfrak{c}_{*} > 0$ some fixed constant. We therefore obtain an overall lower bound
\begin{align*}
\frac{\rd}{\rd t} \Delta_{K}(t) &\geq \veps^{\frac12}(t)\Delta_{K}(t)\left[ \Delta^{\frac12}_{K}(t)\int_{\T} \frac{ |v^{*}_{0}|}{\left|\sin \frac{z}{2}\right|}\sqrt{1 - \frac{1+\cos z}{2\Delta_K(t)}}\, \rd \pi_{t} - \mathfrak{c}_{*}\right] \\
&\geq \veps^{\frac12}(t)\left( \frac{|v^{*}_{0}|}{2} \Delta^{\frac32}_{K}(t) - \mathfrak{c}_{*} \Delta_{K}(t) \right)
\end{align*}
provided we take $\Delta_{K}(0)$ large enough. Quantifying these choices carefully shows that $\Delta_{K}(t)$ blows up in finite time, hence embeddedness is lost. 

Figure \ref{fig3} shows the general dynamic behind this mechanism. The chord formed by the unique realizing pair $(x(t),z(t))$ eventually merge, producing a pair of self-intersections. As the loop formed by one of these self-intersections has uniformly negative inward velocity, it must then collapse and form a curvature singularity in finite time. In particular, \eqref{eq:shapedynamic} is not globally well-posed in general.

\begin{figure}[t]\label{fig4}
\centering
\includegraphics[height=0.9in]{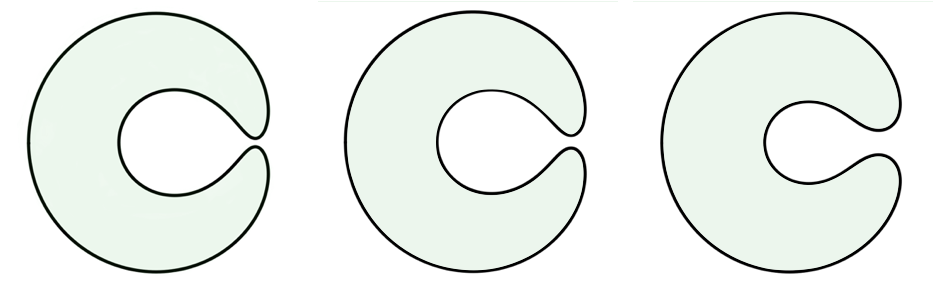}
\includegraphics[height=0.9in]{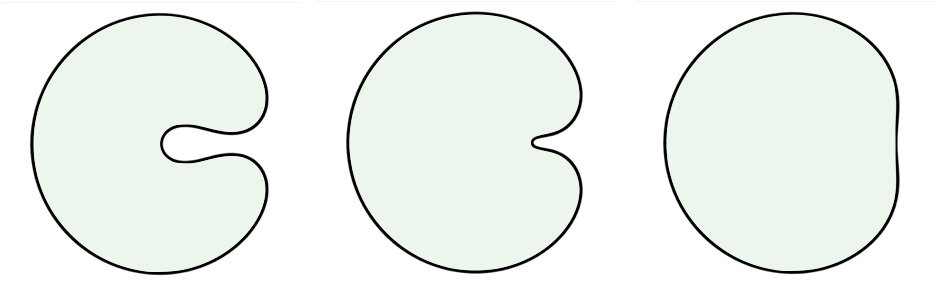}
\caption{For the same initial condition as in figure~\ref{fig3}, including a nonlocal self-repulsion prevents self-intersection. Nonlocality thus restores embeddedness and global existence for this particular initial condition.}\label{fig4}
\end{figure}

Of course, this mechanism persists under the inclusion $F = (c+f)/\sqrt{\veps}$ of a small but non-trivial nonlocal term $f$ into the overall normal forcing. However, in certain regimes nonlocality permits both interfacial growth and can prevent a loss of embeddedness. As an illustration, we assume a similar setup where $o(x,z) = 1$, or in other words, that the set of distortion points $\mathcal{D}_{K}(\phi(t))$ contains only exterior chords. Following [PRE], we consider a force
$$
f(x,t) = \int_{\T} g\left( \frac{|\psi(x,t) - \psi(x+z,t)|^{2}}{2} \right) |\dot \psi(y,t)| \, \rd y \qquad g(s) = g_{\alpha,\beta}(s) = \frac{ \beta }{\sqrt{2\pi \alpha^{2}}} \mathrm{exp}\left( -\frac{s}{2\alpha^2 } \right)
$$
on the interface $\psi = \sigma \phi$ driven by a Gaussian kernel, and assume that the model parameters $(c,\alpha,\beta)$ obey the conditions
$$
\beta  < |c| < 2\beta \qquad \text{and} \qquad c<0
$$
for \emph{arrested front formation}. Generically, the lower bound $\beta < |c|$ allows (but does not necessarily guarantee) that the length of the interface grows in time. As before, we therefore assume that the interfacial growth condition $\mu_{a(t)} \leq 0$ holds. The upper bound $|c| < 2\beta$ ensures, in certain parameter regimes, that the strength of the nonlocal force can prevent distinct points from colliding. To see this, fix a pair $(x,z) \in \mathcal{D}_{K}(\phi(t))$ with $z>0$ (the case $z<0$ follows from identical considerations) and some $0 < \vth < 1,$ then decompose the nonlocal force $f(x,t) = \mathrm{I} + \mathrm{II} + \mathrm{III}$ into three pieces, the first two of which
\begin{align*}
\mathrm{I} &:= \int^{\vth z}_{-\vth z} g\left( \frac12|\psi(x,t) - \psi(x+u,t)|^2 \right) |\dot \psi(u,t)| \, \rd u \\
\mathrm{II} &:= \int^{(2-\vth)z}_{\vth z} g\left( \frac12|\psi(x,t) - \psi(x+u,t)|^2 \right) |\dot \psi(u,t)| \, \rd u
\end{align*}
encode self-repulsive forces centered around the distortion points. The third piece $\mathrm{III} \geq 0$ encodes some non-negative remainder including both the remaining self-repulsive forces and any additional nonlocal repulsive forces due to the presence of other interfaces in the system. For the first term, the inequality $|\psi(x,t) - \psi(x+u,t)| \leq \sigma(t)|u|$ induces a corresponding lower bound
$$
\mathrm{I} \geq \frac{\beta \sigma(t)}{\sqrt{2\pi \alpha^{2}} } \int^{\vth z}_{-\vth z} \mathrm{exp}\left( -\frac{u^2}{2 \alpha^2 \veps(t)} \right) \, \rd u \geq \beta \;\mathrm{erf}\left( \frac{ \sigma(t) \vth z }{\sqrt{2} \alpha} \right)
$$
for the first integral. Similarly, if we let $d := |\phi(x+z,t) - \phi(x,t)|$ denote the distance between a realizing pair and use the bound $|\psi(x,t) - \psi(x+z,t)| \leq \sigma(t)(d + |u-z|)$ then we obtain an analogous lower bound
\begin{align*}
\mathrm{II} \geq \beta  \left[ \mathrm{erf}\left( \frac{ \sigma(t)\left( d + (1-\vth)z \right)}{\sqrt{2}\alpha} \right) - \mathrm{erf}\left( \frac{ \sigma(t) d }{\sqrt{2}\alpha} \right) \right]
\end{align*}
for the second integral. Combining these estimates and optimizing in $0 < \vth < 1$ yields an overall lower bound
$$
f(x,t) \geq 2 \beta \left[ \mathrm{erf}\left( \frac{ \sigma(t)(d+z)}{2\sqrt{2} \alpha} \right) - \frac12 \mathrm{erf}\left( \frac{ \sigma(t) d }{\sqrt{2}\alpha} \right) \right]
$$
for the nonlocal forcing. By the lower bound \eqref{eq:distlwr}, if the dimensionless ratio 
$$
\frac{\alpha}{\sigma(t)\kappa^2_{2}(\phi)}
$$
is small enough relative to the parameter gap $2\beta - |c|$ then the nonlocal force $f(x,t) > |c|$ exceeds the normal driving force $|c|$ whenever the pseudo-distortion $\Delta_{K}(t)$ is large enough. For such a range of parameter values, we therefore have $f(x+z,t) + f(x,t) + 2c > 0$ on $\mathrm{supp}(\pi_t)$ for $\Delta_{K}(t)$ large enough. Appealing once again to the distortion point calculus gives the inequality
$$
-\frac{ \langle \delta \phi(x,z,t), \delta \ddot \phi(x,z,t) \rangle }{|\delta \phi(x,z,t)|^2} \leq 1,
$$
and so overall we have
$$
-\veps^{\frac12}(t)\Delta_{K}(t) \int_{\T \times \T} o(x,z)\sqrt{1 - \frac{1 + \cos z}{2 \Delta_{K}(t)}}\left(\frac{ v(x,t) + v(x+z,t) }{|\delta \phi(x,z,t)|}\right) \, \rd \pi_t \leq \veps(t)\Delta_{K}(t)
$$
whenever the pseudo-distortion $\Delta_K(t)$ gets large. Now, assume that the inequality
\begin{align}\label{eq:badcase}
\int^{x+z}_{x} a(u,t) \, \rd u \leq 0 \qquad \longleftrightarrow \qquad \int^{x+z}_{x} \kappa(u,t)\left[ \kappa(u,t) + F(u,t) \right] \, \rd u \leq 0
\end{align}
holds for $(x,z) \in \mathcal{D}_{K}(\phi(t))$ our fixed realizing pair. Then the lower bound
$$
z \leq \int^{x+z}_{x} \kappa^{2}(u,t) \, \rd u
$$
still holds by the distortion point calculus (c.f. the argument for a pure motion by mean curvature), while the upper bound
$$
\int^{x+z}_{x} \kappa^{2}(u,t)\, \rd u \leq z\|F(t)\|^2_{L^{\infty}(\T)}
$$
follows from \eqref{eq:badcase} and Cauchy-Schwarz. For such a pair $(x,z) \in \mathcal{D}_{K}(\phi(t))$ this upper bound gives
\begin{align*}
-\frac{ \sin z}{2(1 - \cos z)}\int^{x+z}_{x} a(u,t) \, \rd u \leq \|F(t)\|^2_{L^{\infty}(\T)},
\end{align*}
while if \eqref{eq:badcase} fails this bound holds trivially. As $h(z) \leq 1$ on $\T$ and $\mu_{a(t)} \leq 0$ by assumption, we therefore have the upper bound
$$
\frac{\rd}{\rd t}\log\left( \sqrt{\veps(t)} \Delta_{K}(t) \right) \leq \veps(t)\left( 1 + \|F(t)\|^{2}_{L^{\infty}(\T)} \right)
$$
on the evolution of the pseudo-distortion when $\Delta_K(t)$ is large enough. In contrast to the case of a constant normal forcing $F = c \sigma(t)$ with $c<0,$ embededness is preserved as long as the normal force $F$ remains bounded.

Figure \ref{fig4} shows the general dynamic behind this mechanism. We use the same initial condition as before, but include a nonlocal self-repulsive force. The chord formed by the unique realizing pair $(x(t),z(t))$ eventually merge, producing a pair of self-intersections. As the loop formed by one of these self-intersections has uniformly negative inward velocity, it must then collapse and form a curvature singularity in finite time. In particular, \eqref{eq:shapedynamic} is not globally well-posed in general.

We conclude this section with a few observations regarding the overall picture. For geometries with purely exterior realizing pairs (e.g. figure \ref{fig2} at center), the addition of a seemingly benign constant normal growth $F = c \nrml$ may cause a loss of embeddedness and subsequent formation of curvature singularities. Incorporating a repulsive nonlocal force $F = (c + f)\nrml$ can, but need not, restore embeddedness. The underlying mechanism is that of arrested front formation. The situation reverses for geometries with purely interior realizing pairs (e.g. figure \ref{fig2} at left), in that a repulsive nonlocal force $F = f\nrml$ will lead to loss of embeddedness as the width of the ``bottleneck'' collapses. This illustrates a mechanism akin to cell-death in a bacteria colony. Incorporating a constant normal growth $F = (c + f)\nrml,$ i.e. cell-birth, can then restore embeddedness and prevent collapse. The same set of model parameters may lead to either scenario. As these examples illustrate, embeddedness is heavily dependent upon both model parameters and the geometry of the initial datum. 

\section{Numerical Approximation}\label{sec:Numeric} 

We conclude with documentation for the numerical scheme that we have implemented and made publicly available. We use the tangent angle formulation of the dynamic
\begin{align}\label{eq:evonum}
\partial_{t} \vth_i &= \veps_i \left( \ddot \vth_i + \po\big( a_i \big) \dot \vth_i - \dot F_i \right) 
\qquad\qquad\quad \dot \veps_i = 2 \veps^2_i \mu_{a_i}  &&\;\dot \frc_i = \veps^{\frac12}_i \mu_{\vv_i} \\
a_i &= \dot \vth^2_i - F_i\dot \vth_i &&\vv_i =  \po\big(a_i\big) \tang_i  + F_i \nrml_i \nonumber
\end{align}
for simulation purposes, which we supplement with boundary conditions
$$
\vth_i(x,t) = \eta_i(x,t) + k_i(x+\pi)
$$
determined by enforcing periodicity of the $\eta_i$ and their derivatives. The initial conditions $\vth_i(x,0)$ determine the constant winding numbers $k_i \in \Z$ of each interface. Finally, the definitions
$$
\psi_i(x,t) := \frc_i(t) + \frac1{\sqrt{\veps_i(t)}} \pc \tang_i(x,t) \qquad \text{and} \qquad \tang_i(x,t) := \begin{bmatrix} - \sin \vth_i(x,t) \\ \;\;\;\cos \vth_i(x,t) \end{bmatrix}
$$
allow us to compute the nonlocal forcing
\begin{align}\label{eq:nonlocnum}
F_i(x,t) = \frac1{\sqrt{ \veps_i(t)}}\left( c_i + \sum^{m}_{j=1} \frac1{\sqrt{\veps_j(t)}} \int_{\T} g_{ij}\left( \frac12|\psi_i(x,t) - \psi_j(x,t)|^{2} \right)\,\rd y + F^{ {\rm ext} }_{i}(x,t) \right)
\end{align}
from knowledge of interfacial lengths $\veps_i$, their centers of mass $\frc_i$ together with angular variables $\vth_i$ that represent the unit tangent. The functions $F^{ {\rm ext}}_i \in H^{1}(\T),\,i \in [m]$ represent some generic external normal forcing applied to each interface. This formulation has two main advantages. First, in contrast to \eqref{eq:basic} it exhibits a linear diffusive term rather than a quasilinear one. Second, in contrast to \eqref{eq:shapedynamic} it exhibits lower-order rather than leading-order nonlinearities. For these reasons we may avoid a diffusive time-step restriction relatively easily, while still retaining an efficient and easy-to-implement procedure.

The main computational cost comes from evaluating the nonlocal forcing \eqref{eq:nonlocnum}, so we prefer to compute them only once per time-step. We therefore approximate the total nonlinear forcing
$$
 G_i := \po\big( a_i \big) \dot \vth_i  - \dot F_i 
$$
in an explicit, multi-step manner. Specifically, assume we have a discrete collection of times $\{t_0,\ldots,t_n\}$ with corresponding time-steps $\rd t_{k} := t_{k+1} - t_{k}$ between consecutive times. We also assume a history of approximations
$$
\vth^{k}_{i}(x) \simeq \vth_i(x,t_n) \qquad a^{k}_{i}(x) \simeq a_i(x,t_k) \qquad \text{and} \qquad F^{k}_{i}(x) \simeq F_{i}(x,t_k)
$$
to the angular variables, transport coefficients and nonlocal forcings, respectively. Any function $h(x,t)$ for which we have a history allows us to construct a second-order extrapolation
\begin{align*}
h(x,t) \simeq \mathrm{H}^{n}(x,t) := P_{1}\left(t-t_n;h^{n}(x),h^{n-1}(x),\rd t_{n-1}\right) \qquad \text{on} \qquad [t_n,t_{n+1}]
\end{align*}
of this function for use in the next time step, where the straightforward linear interpolation
\begin{align*}
P_1\left(s;v_{0},v_{-1},\tau_{-1}\right) &= v_0 \left(1 + \frac{s}{\tau_{-1}} \right) - \frac{s}{ \tau_{-1} }v_{-1}
\end{align*}
suffices for the extrapolation. Given a new time $t_{n+1} = t_{n} + \rd t_{n}$ and these approximations, we may then try to use an approximate form of the evolution (\ref{eq:evonum},\ref{eq:nonlocnum}) to advance the solution variables 
$$\left(\vth^n_i,\veps^n_i,\frc^n_i\right) \quad \to \quad \left(\vth^{n+1}_i,\veps^{n+1}_i,\frc^{n+1}_i\right)$$
over the current time interval.

We accomplish this using a relatively simple update strategy. Fix $i \in [m]$ and let $\vth = \vth_i $ denote the corresponding angular variable. Recall that if we define a function
$$
\frw(t) = t_n + \int^{t}_{t_n} \veps(s) \, \rd s
$$
as well as its inverse map $\fra \circ \frw(t) = t$ then the temporal change of variables $\vth(x,t) = u\left(x,\frw(t) \right)$ converts the angular evolution
$$
u_{t}(x,t) = u_{xx}(x,t) + G\big(x,\fra(t) \big)
$$
into a generic forced heat equation. Now consider the task of advancing 
$$u(x,t_n) \quad \to \quad u\big(x,\frw(t_n + \rd t_n) \big) =: u\big(x , t_n + \rd \frw_n \big)$$
over the current time interval. The second-order approximations 
\begin{align*}
a(x,t) \simeq \mathrm{A}^{n}(x,t) &:= P_{1}\left(t-t_n;a^{n}(x),a^{n-1}(x),\rd t_{n-1}\right) \\
\vth(x,t) \simeq \Theta^{n}(x,t) &:= P_{1}\left(t-t_n; \vth^{n}(x), \vth^{n-1}(x),\rd t_{n-1}\right)\\
F(x,t) \simeq \mathrm{F}^{n}(x,t) &:= P_{1}\left(t-t_n; F^{n}(x), F^{n-1}(x),\rd t_{n-1}\right)
\end{align*}
allow us to define an approximate forcing 
$$v^{n}(x,t) = \mathrm{G}^{n}\big(x,\fra(t)\big) \qquad \text{with} \qquad \mathrm{G}^{n}(x,t) := \po\left( \mathrm{A}^{n}(x,t) \right) \dot{\Theta}^{n}(x,t) - \dot{\mathrm{F}}^{n}(x,t)$$ 
over $[t_n,t_n + \rd w_n]$ that we may compute at any instant in time. The decomposition $\vth(x,t) = \eta(x,t) + k(x+\pi)\,, k \in \Z$ with $\eta$ periodic induces an analogous decomposition $u(x,t) = \xi(x,t) + k(x+\pi)$ with $\xi$ periodic. We may therefore write the solution driven by the approximate forcing
$$
\hat \xi_{\ell}(t_n + \rd \frw_n) = \re^{-\ell^2 \rd \frw_n} \hat \xi_{\ell}(t_n) + \int^{t_n + \rd \frw_n}_{t_n} \re^{-\ell^2(t_n + \rd \frw_n - s)} \hat v_{\ell}(s) \, \rd s
$$
explicitly. A simple quadrature approximation over $[t_n,t_n + \rd \frw_n]$ gives an estimate of the contribution
$$
\frac1{\rd \frw_n} \int^{t_n + \rd \frw_n}_{t_n} \re^{-\ell^2(t_n + \rd \frw_n - s)} \hat v_{\ell}(s) \, \rd s
\simeq E_{1}\left(-\ell^2 \rd \frw_n \right) \hat v_\ell(t_n) + E_{2}\left(-\ell^2 \rd \frw_n \right)\left(  \hat v_{\ell}(t_n + \rd \frw_n) - \hat v_{\ell}(t_n) \right)
$$
$$
E_0(x) = \re^{x} \qquad \qquad  x E_{j+1}(x) := E_j(x) - E_j(0)
$$
from the non-homogeneous forcing. Now as $t_n = \frw(t_n)$ and $t_n + \rd \frw_n = \frw(t + \rd t_n)$ we have the identities
$$
v(x,t_n) = \mathrm{G}^{n}(x,t_n) \qquad \text{and} \qquad v(x,t_{n} + \rd \frw_n) = \mathrm{G}^{n}(x,t_n + \rd t_n),
$$
and so we may affect the update
\begin{align*}
\hat \xi_{\ell}(t_{n} + \rd \frw_n) &\gets E_0\left( - \ell^{2} \rd \frw_n\right) \hat \eta^{n}_{\ell} + \rd \frw_n\left[  E_1\left( - \ell^{2} \rd \frw_n\right)\widehat{\mathrm{G}}_\ell(t_n) + E_2\left( - \ell^{2} \rd \frw_n\right)   \left(  \widehat{\mathrm{G}}_{\ell}(t_{n+1}) - \widehat{\mathrm{G}}_{\ell}(t_{n}) \right) \right]
\end{align*}
if we know the increment $\rd \frw_{n}$ at each step. Performing an inverse transform then gives a means to update
$$
\vth^{n}(x) \quad \longrightarrow \quad \vth^{n+1}\left(x\right) \simeq \xi(x,t_n + \rd \frw_n) + k(x+\pi)
$$
 the corresponding approximation for the angular variables, as well as the tangents and normals
 $$
 \tang^{n+1}(x) = \begin{bmatrix} -\sin \vth^{n+1}(x) \\ \;\;\;\cos \vth^{n+1}(x) \end{bmatrix} \qquad \text{and} \qquad  \nrml^{n+1}(x) = \begin{bmatrix} \cos \vth^{n+1}(x) \\ \sin \vth^{n+1}(x) \end{bmatrix}
 $$
 to the interface.

It therefore suffices to determine the increment $\rd \frw_n$ to close the evolution across the current time interval. We simply approximate
$$
\mu_{a}(t) \simeq \mu_{A^{n}}(t)
$$
and analytically solve the corresponding approximate evolution $\sigma \dot \sigma = - \mu_{A^n}$ to update the speed
\begin{align*}
\sigma^{n+1} = \left( \left(\sigma^{n}\right)^{2} -  \int^{t_{n+1}}_{t_n} \mu_{A^n}(s)\, \rd s \right)^{\frac12} = \left( \left(\sigma^{n}\right)^{2} - \rd t_n\left[2 \mu_{a^{n}} + \frac{\rd t_n}{\rd t_{n-1}}\big(\mu_{a^n}-\mu_{a^{n-1}}\big)\right]\right)^{\frac12}
\end{align*}
of the interface. The time step restriction
$$
2\,\rd t_{n} \leq \frac{ \left( \sigma^n\right)^{2} }{\left| \mu_{a^{n}} \right| + \sqrt{\left| \mu_{a^{n}} \right| + \left|\Delta \mu_{a^{n}} \right| \frac{ \left( \sigma^n\right)^{2} }{2} } } \qquad \text{with} \qquad \Delta \mu_{a^{n}} := \frac{ \mu_{a^{n}} - \mu_{a^{n-1}} }{\rd t_{n-1}}
$$
guarantees that the updated speed $\sigma^{n+1}$ remains positive. We then use the definition
$$
\veps^{n+1} =  \left(\sigma^{n+1}\right)^{-2}
$$
and use the approximation
$$
\rd \frw_n = \int^{t_{n+1}}_{t_n} \veps(s) \, \rd s \simeq \rd t_{n} \left( \frac{\veps^n + \veps^{n+1} }{2} \right)
$$
to find the increment. Finally, we may update the center of mass
$$
\frac{\frc^{n+1} - \frc^{n}}{\rd t_n} = \frac{ \sqrt{\veps^{n+1}} \mu_{\vv^{n+1}} + \sqrt{\veps^n} \mu_{\vv^n} }{2} \qquad \text{with} \qquad \vv^{k} = \po\big(A^{n}(t_k)\big) \tang^{k} + \mathrm{F}^{n}(t_k) \nrml^{k}
$$
to compute the new forcing coefficients then move to the next time step.  The computational results displaying the difficulties in establishing global well-posedness from Section~\ref{sec:Embeded} were produced using this numerical scheme.

\section{Appendix: Proofs of Technical Lemmas}\label{sec:Appendix}

\begin{lemma}[Lemma \ref{lem:chofvar}]\label{lem:chofvarapp}
Assume that $f \in H^{1}(\T)$ and that $\ga_i \in W^{1,1}(\T), i \in \{0,1\}$ are immersions. Let
$$
\eta_i(x) := -\pi + \frac{2\pi}{\ell(\ga_i)} \int^{x}_{-\pi} |\dot \ga_i(z)| \, \rd z \qquad \text{and} \qquad \xi_i \circ \eta_i(x) = x
$$
denote their transition maps to constant speed coordinates. Then the estimate
$$
\|f - f \circ \eta_1 \circ \xi_0\|_{L^{2}(\T)} \leq \mathfrak{C}\|\dot f\|_{L^{2}(\T)}\frac{\| \dot \ga_1 - \dot \ga_0\|_{L^{1}(\T)}}{\ell(\ga_0) \vee \ell(\ga_1)}
$$ 
holds for $\mathfrak{C}>0$ a universal constant.
\end{lemma}
\begin{proof}
Define the function
$$
h(z) := \mathbf{1}_{[-3\pi,3\pi]}|\dot f(z)| \qquad \text{with} \qquad \|h\|_{L^{2}(\R)} = 3\|\dot f\|_{L^{2}(\T)}
$$
by extending $|\dot f(z)|$ from $I_{\pi}$ to all of $\R$ using periodicity. Let
$$
h^{*}(x) := \sup_{r > 0} \, \frac1{2r}\int^{x+r}_{x-r} h(z) \, \rd z
$$ 
denote its maximal function, and recall that
$$
\|h^{*}\|_{L^{2}(\R)} \leq \mathfrak{C} \|h\|_{L^{2}(\R)} \leq \mathfrak{C} \|\dot f\|_{L^{2}(\T)}
$$
since on $L^{2}(\R)$ the maximal function defines a bounded operator. Now let $\veps(x) := |\eta_1 \circ \xi_0(x) - x|$ denote the deviation of the composite map from the identity. Taylor's theorem gives the pointwise estimate
$$
|f(x) - f \circ \eta_1 \circ \xi_0(x)| \leq \int^{x + \veps(x)}_{x - \veps(x)} |\dot f(z)| \, \rd z,
$$
and since $|\veps(x)| \leq 2\pi$ for $x \in I_{\pi}$ the pointwise estimate
$$
|f(x) - f \circ \eta_1 \circ \xi_0(x)| \leq \int^{x + \veps(x)}_{x - \veps(x)} h(z)\, \rd z \leq 2\veps(x)h^{*}(x) 
$$
holds for $x \in I_{\pi}$. Integrating over $I_{\pi}$ then shows that
$$
\|f - f \circ \eta_1 \circ \xi_0\|_{L^{2}(\T)} \leq 2 \|\veps\|_{L^{\infty}(\T)}\|h^{*}\|_{L^{2}(\T)} \leq \mathfrak{C} \|\veps\|_{L^{\infty}(\T)}\|\dot f\|_{L^{2}(\T)}
$$
for $\mathfrak{C}>0$ some universal constant. Now $\|\veps\|_{L^{\infty}(\T)} = \|\veps \circ \eta_0\|_{L^{\infty}(\T)} = \|\eta_1 - \eta_0\|_{L^{\infty}(\T)},$ and since $v(x) := \eta_1(x) - \eta_0(x)$ vanishes at $x = -\pi$ the Sobolev embedding $\|v\|_{L^{\infty}(\T)} \leq \|\dot v\|_{L^{1}(\T)}$ applies. The simple estimate
\begin{align*}
\|\dot v\|_{L^{1}(\T)} = \int_{\T} \left| \frac{2\pi}{\ell(\ga_0)}|\dot \ga_0(x)| -  \frac{2\pi}{\ell(\ga_1)}|\dot \ga_1(x)| \right| \, \rd x \leq \frac{2\pi}{\ell(\ga_0)}\left( \|\dot \ga_1 - \dot \ga_0\|_{L^{1}(\T)} + |\ell(\ga_1)-\ell(\ga_0)| \right)
\end{align*}
then proves the lemma.
\end{proof}

\begin{lemma}[Lemma \ref{lem:chofvar}]\label{lem:equivboundapp}
Assume $\ga_0,\ga_1 \in H^{2}(\T)$ are immersions and let $\psi_i := \ga_i \circ \xi_i$ denote their constant speed representations. Then the map $\ga_i \mapsto \psi_i$ is locally Lipschitz with respect to the $H^{1}(\T)$ topology, in the sense that the bound
$$
\|\psi_1 - \psi_0\|_{H^{1}(\T)} \leq \mathfrak{C}^{*}\left( \frac{\ell(\ga_0)\vee\ell(\ga_1)}{\frs_{*}(\ga_0)\wedge \frs_{*}(\ga_1)},\frac{\|\ddot \ga_0\|_{L^{2}(\T)}\vee\|\ddot \ga_1\|_{L^{2}(\T)}}{\frs_{*}(\ga_0)\wedge \frs_{*}(\ga_1)}\right)\|\ga_1 - \ga_0\|_{H^{1}(\T)}
$$
holds for $\mathfrak{C}^{*} : \R^{2} \mapsto \R$ some continuous, increasing function of its arguments.
\end{lemma}
\begin{proof}
Use the identity $\ga_i = \psi_i \circ \eta_i$ and the triangle inequality to find
\begin{align*}
\|\psi_0 - \psi_1\|_{L^{\infty}(\T)} &= \|\ga_0 - \psi_1 \circ \eta_0\|_{L^{\infty}(\T)} \leq \|\ga_0 - \ga_1\|_{L^{\infty}(\T)} + \|\ga_1 - \psi_1 \circ \eta_0\|_{L^{\infty}(\T)} \\
&= \|\ga_1 - \ga_0\|_{L^{\infty}(\T)} + \|\psi_1 - \psi_1 \circ \eta_0 \circ \xi_1\|_{L^{\infty}(\T)},
\end{align*}
then use the fact that $\psi_1$ is $\ell(\ga_1)/2\pi$-Lipschitz and lemma \ref{lem:chofvarapp} to obtain an estimate
$$
\|\psi_0 - \psi_1\|_{L^{\infty}(\T)} \leq \|\ga_1 - \ga_0\|_{L^{\infty}(\T)} + \|\dot \ga_1 - \dot \ga_0\|_{L^{1}(\T)} + |\ell(\ga_1) - \ell(\ga_0)|
$$
that suffices for the $L^{\infty}(\T)$ norm. It therefore suffices to estimate the $L^{2}(\T)$ norm of the derivatives. 

Consider first the case that $\ga_1,\ga_0$ are \emph{compatibly oriented}, i.e. that the inequality $\langle \dot \ga_1(x),\dot \ga_0(x) \rangle \geq 0$ holds. Then $\ga(x,\vth) := \vth \ga_1(x) + (1-\vth)\ga_0(x)$ defines a one-parameter family of immersions, and moreover the lower bound
\begin{align}\label{eq:lowerbndsapp}
|\dot \ga(x,\vth)| \geq \frac{ \frs_{*}(\ga_0) \wedge \frs_{*}(\ga_1) }{\sqrt{2}}
\end{align}
holds. An immediate consequence is the uniform bound
\begin{align}\label{eq:spdbndapp}
|\dot \xi(x,\vth)| \leq \frac{ \ell(\ga_1) \vee \ell(\ga_2)}{\frs_{*}(\ga_1) \wedge \frs_{*}(\ga_2)}
\end{align}
for the transition to constant speed coordinates. Let $\psi(x,\vth) = \ga( \xi(x,\vth), \vth)$ denote the constant speed representation of this family, and so in particular the identities $\psi(x,1) = \psi_1(x),$ $\psi(x,0) = \psi_0(x)$
$$
\dot \psi_1(x) - \dot \psi_0(x) = \int^{1}_{0} \dot \psi_{\vth}(x,\vth) \, \rd \vth
$$
all hold. Define $\Delta \ga(x) := \ga_1(x) - \ga_0(x)$ and apply \eqref{eq:cspdder} to obtain
$$
\psi_{\vth} = \vv - \po\big( \langle \dot \vv , \tang_{\psi} \rangle ) \tang_{\psi}  \qquad \vv(x,\vth) = \Delta \ga\left( \xi(x,\vth) \right),
$$
where $\tang_{\psi}(x,\vth) \propto \dot \psi(x,\vth)$ denotes the unit tangent of the interpolant. Differentiate and apply the triangle inequality to obtain the bound
$$
\|\dot \psi_{\vth}\|_{L^{2}(\T)} \leq \|\dot \vv - \big( \langle \dot \vv, \tang_{\psi} \rangle - \mu_{\langle \dot \vv , \tang_{\psi}\rangle}  \big)\tang_{\psi}\|_{L^{2}(\T)} + \|\po\left( \langle \dot \vv , \tang_{\psi} \rangle\right) \dot \tang_{\psi}\|_{L^{2}(\T)},
$$
then use Jensen's inequality to establish the upper bound
$$
\|\dot \vv - \big( \langle \dot \vv, \tang_{\psi} \rangle - \mu_{\langle \dot \vv , \tang_{\psi}\rangle}  \big)\tang_{\psi}\|_{L^{2}(\T)} \leq \|\dot \vv\|_{L^{2}(\T)}
$$
for the first term. Apply a change of variables and \eqref{eq:spdbndapp} to obtain the final estimate
$$
\|\dot \vv\|_{L^{2}(\T)} \leq \left( \frac{\ell(\ga_1)\vee\ell(\ga_0)}{\frs_{*}(\ga_1)\wedge \frs_{*}(\ga_0)} \right)^{\frac12}\|\dot \ga_1 - \dot \ga_0\|_{L^{2}(\T)}
$$
for the first term. Use the simple bound
\begin{align*}
\|\po\left( \langle \dot \vv , \tang_{\psi} \rangle\right) \dot \tang_{\psi}\|_{L^{2}(\T)} &\leq \|\po\big(\langle \dot \vv,\tang_{\psi}\rangle \big)\|_{L^{\infty}(\T)}\|\dot \tang_{\psi}\|_{L^{2}(\T)}\\
&\leq 2\|\dot \vv\|_{L^{1}(\T)}\|\dot \tang_{\psi}\|_{L^{2}(\T)} = 2 \|\dot \ga_1 - \dot \ga_0\|_{L^{1}(\T)} \|\dot \tang_{\psi}\|_{L^{2}(\T)}
\end{align*}
for the second term, then use \eqref{eq:lowerbndsapp} and the fact that
$$
\|\dot \tang_{\psi}\|^{2}_{L^{2}(\T)} \leq \frac{\ell(\vth)}{2\pi}\int_{\T} \frac{|\ddot \ga(x,\vth)|^{2}}{|\dot \ga(x,\vth)|^{3}} \, \rd x \leq \left( \frac{ \ell(\ga_0)\vee \ell(\ga_1)}{\frs_{*}(\ga_0)\wedge \frs_{*}(\ga_1)}\right)\left( \frac{\|\ddot \ga_0\|_{L^{2}(\T)} \vee \|\ddot \ga_1\|_{L^{2}(\T)} }{\frs_{*}(\ga_0)\wedge \frs_{*}(\ga_1)} \right)^{2}
$$
to prove the lemma in this case.

It remains to consider the case that $\ga_0,\ga_1$ are not compatibly oriented. Then there exists some $x$ for which $\langle \dot \ga_1(x),\dot \ga_0(x)\rangle <0,$ and so necessarily the lower bound
\begin{align}\label{eq:bignormapp}
\|\dot \ga_1 - \dot \ga_0\|^2_{L^{\infty}(\T)} \geq |\dot \ga_1(x) - \dot \ga_0(x)|^2 = |\dot \ga_1(x)|^2 + |\dot \ga_0(x)|^2  - 2 \langle \dot \ga_1(x) ,\dot \ga_0(x) \rangle \geq 2\left( \frs_{*}(\ga_0) \wedge \frs_{*}(\ga_0) \right)^2
\end{align}
must hold. Use the embedding \eqref{eq:gns} to obtain the upper bound
\begin{align*}
\|\dot \ga_1 - \dot \ga_0\|_{L^{\infty}(\T)} &\leq \|\dot \ga_1 - \dot \ga_0\|^{\frac12}_{L^{2}(\T)}\|\ddot \ga_1 - \ddot \ga_0\|^{\frac12}_{L^{2}(\T)} \leq \sqrt{2}\|\dot \ga_1 - \dot \ga_0\|^{\frac12}_{L^{2}(\T)}\left( \|\ddot \ga_0\|_{L^{2}(\T)} \vee  \|\ddot \ga_1\|_{L^{2}(\T)} \right)^{\frac12}
\end{align*}
and then combine it with \eqref{eq:bignormapp} to conclude that the lower bound
$$
2 \leq \frac{\|\ddot \ga_0\|_{L^{2}(\T)} \vee \|\ddot \ga_1\|_{L^{2}(\T)}}{\frs_{*}(\ga_0)\wedge \frs_{*}(\ga_1)}\frac{ \|\dot \ga_1 - \dot \ga_0\|_{L^{2}(\T)} }{\frs_{*}(\ga_0) \wedge \frs_{*}(\ga_1)}
$$
must hold as well. Then use the trivial upper bound
$$
\|\dot \psi_1 - \dot \psi_0\|_{L^{2}(\T)} \leq \ell(\ga_0)\vee\ell(\ga_1) \leq\frac{\|\ddot \ga_0\|_{L^{2}(\T)} \vee \|\ddot \ga_1\|_{L^{2}(\T)}}{\frs_{*}(\ga_0)\wedge \frs_{*}(\ga_1)}\frac{\ell(\ga_0)\vee\ell(\ga_1)}{\frs_{*}(\ga_0) \wedge \frs_{*}(\ga_1)} \|\dot \ga_1 - \dot \ga_0\|_{L^{2}(\T)} 
$$
to prove the lemma.
\end{proof}

\begin{lemma}\label{lem:mapsapp}
For $i=1,2$ let $f_i : \R^{+} \mapsto \R$ denote continuous functions obeying the global bounds
$$
\max\left\{ \sup_{t \geq 0}\;\, |f_1(t)| , \sup_{t \geq 0}\;\, |f_2(t)| \right\}  \leq \mathfrak{f}^{*} \qquad \text{and} \qquad \sup_{t \geq 0}\;\, |f_1(t) - f_2(t)| \leq \frd \frf^{*},
$$
and let $\fra_i,\frb_i: \R^{+} \mapsto \R$ denote any corresponding $C^{1}(\R^{+})$ solutions to the systems
\begin{align}\label{eq:sysapp}
\dot \fra_i(t) &= \frb_i(t) \qquad \qquad \quad \quad \;\,\,\fra_i(0) = 0 \nonumber \\
\dot \frb_i(t) &= \frb_i(t) f_i( \fra_i(t) ) \qquad \quad \frb_i(0) = \frb_{*} >0
\end{align}
of ordinary differential equations. Then the following hold ---
\begin{enumerate}[{\rm i)}]
\item The $\fra_i(t)$ are strictly increasing and obey the bound
$$
\frac{ \frb_* }{\frf^{*} }\left( 1 - \re^{-\frf^{*}t} \right) \leq \fra_i(t) \leq \frac{\frb_*}{\frf^{*}}\left( \re^{ \frf^{*}t } - 1 \right)
$$
for all time.
\item The inverse maps $\frw_i\big( \fra_i(t) ) = t$ are well-defined on the common interval
$$
0 \leq t < \frac{\frb_*}{\frf^{*}} := T_{*},
$$
are $C^{2}\big( [0,T_{*}) \big)$ and obey the bound
$$
\frac1{\frf^{*}}\log\left( 1 + \frac{\frf_*}{\frb_*}t \right) \leq \frw_i(t) \leq \frac1{\frf^{*}}\log\left( \frac1{1 - \frac{\frf_*}{\frb_*}t} \right)
$$
for as long as they are well-defined.
\item The transition maps
$$
T_{12}(t) := \frw_1\big( \fra_2(t) ) \qquad \text{and} \qquad T_{21}(t) := \frw_2\big( \fra_1(t) )
$$
are well-defined on the common interval
$$
0 \leq t < \frac{\log 2}{\frf^*} := T_{**},
$$
are $C^{2}\big( [0,T_{**}) \big)$ with $T_{ij}(0) = 0, \dot T_{ij}(0) = 1$ and obey the bounds
\begin{align*}
&\left| \frac{T_{ij}(t) - t}{t^2} \right| \leq \frac{\frd \frf^{*} }{2} \re^{\frf^{*}(t+T_{ij}(t))} \leq \frac{\frd \frf^{*} }{2}\left( \frac{\re^{\frf^{*} t} }{2 - \re^{\frf^* t} } \right) \qquad \text{and} \\
&\left| \frac{\dot T_{ij}(t)-1}{t}\right| \leq  \frd \frf^{*} \re^{\frf^{*}(t+T_{ij}(t))}  \leq \frd \frf^{*} \frac{\re^{\frf^{*}t} }{2 - \re^{\frf^* t} }
\end{align*}
for as long as they are well-defined.
\item The maps $T_{+}(t) := \max\{ T_{12}(t),T_{21}(t) \}$ and $T_{-}(t) := \min\{ T_{12}(t),T_{21}(t) \}$ are locally Lipschitz on $[0,T_{**}),$ are continuously differentiable outside of a countable set in $[0,T_{**}),$ are ordered $T_{-}(t) \leq t \leq T_{+}(t)$ and obey the bounds in $\mathrm{iii})$ for as long as they are well-defined.
\end{enumerate}
\end{lemma}
\begin{proof}
The positivity of the $\frb_i$ follows directly from integrating
$$
\frb_i(t) = \frb_i(0) \re^{\int^{t}_{0} f_i( \fra_i(s) ) \, \rd s } \qquad \text{where} \qquad  \frb_i(0) = \frb_{*} > 0,
$$
and so the first conclusion of the lemma follows immediately from the inequalities
$$
\frac{\frb_*}{\frf^{*}}\left( 1 - \re^{-\frf^* t} \right) = \frb_* \int^{t}_{0} \re^{-\frf^* s } \, \rd s \leq \int^{t}_{0} \frb_i(s) \, \rd s = \fra_i(t) \leq  \frb_* \int^{t}_{0} \re^{\frf^* s } \, \rd s = \frac{\frb_*}{\frf^*}\left( \re^{\frf^*t} - 1 \right).
$$
In particular, $[0,T_*) \subset \mathrm{range}( \fra_i )$ and thus $[0,T_{*}) \subset \mathrm{dom}( \frw_i )$ as claimed in the second part of the lemma. The claimed bounds on $\frw_i$ then follow by inverting the inequalities
$$
\frac{\frb_*}{\frf^*}\left( 1 - \re^{-\frf^* \frw_i(t)} \right) \leq t \leq \frac{\frb_*}{\frf^*}\left( \re^{\frf^* \frw_i(t)} - 1 \right)
$$
for the inverse maps. For the third statement, let $T(t) = T_{21}(t)$ denote one of the transition maps; the argument for $T_{12}(t)$ is identical. That $T(0) = 0$ follows immediately; the identity
$$
\dot T(t) = \dot \frw_2 \big( \fra_1(t) \big) \dot \fra_1(t) = \frac{ \frb_1(t) }{\frb_2( T(t) )}
$$
also follows immediately from the fact that $\dot \frw_2( \fra_2(t) )\frb_2(t)  = 1,$ the definition of $T(t)$ and the chain rule. Thus
$$
\dot T(0) = \frac{ \frb_1(0) }{\frb_2( 0 )} = \frac{\frb_*}{\frb_*} = 1
$$
as claimed. Moreover, if $t < T_{**}$ then $\fra_1(t) < T_{*}$ by the first part of the lemma, and so $T(t)$ is well-defined.

To prove the bound on the transition maps, put $g_i(t) := f_i( \fra_1(t) )$ and note that the chain rule gives
$$
\ddot T(t) = \frac{ \dot \frb_1(t) \frb_2( T(t) ) - \frb_1(t) \dot \frb_2(T(t))\dot T(t) }{\frb^2_2( T(t) )} = \frac{ \frb_1(t) }{\frb_2(T(t))}\left( f_1 \circ \fra_1(t)  - f_2 \circ \fra_2 \circ T(t) \dot T(t)  \right) = \dot T(t)\left( g_1(t) - g_2(t)\dot T(t) \right)
$$
since $\fra_2 \circ T(t) = \fra_2 \circ \frw_2 \circ \fra_1(t) = \fra_1(t)$ by definition of the transition map. The identity
\begin{equation}\label{eq:tdfirst}
\mathrm{exp}\left( \int^{t}_{0} g_2(s) \dot T(s) \, \rd s \right) \dot T(t) = \mathrm{exp}\left( \int^{t}_{0} g_1(s) \, \rd s \right) 
\end{equation}
then follows by integrating. Set $g_{12}(s) := g_1(s) - g_2(s),$ multiplying by $g_2(t)$ and integrate once again to deduce the chain of equalities
\begin{align}\label{eq:tdsecnd}
\mathrm{exp}\left( \int^{t}_{0} g_2(s) \dot T(s) \, \rd s \right) &= 1 + \int^{t}_{0} g_2(s)\mathrm{exp}\left( \int^{s}_{0} g_1(z) \, \rd z \right) \, \rd s \nonumber \\
&= 1 + \int^{t}_{0} \mathrm{exp}\left( \int^{s}_{0} g_{12}(z) \, \rd z \right) \frac{\rd}{\rd s} \left[ \mathrm{exp}\left( \int^{s}_{0} g_{2}(z) \, \rd z \right)\right] \, \rd z \nonumber \\
&= \mathrm{exp}\left( \int^{t}_{0} g_1(s) \, \rd s \right) - \int^{t}_{0} g_{12}(s) \mathrm{exp}\left(\int^{s}_{0} g_1(z) \, \rd z \right) \, \rd s
\end{align}
due to an integration by parts. A change of variables $s = T(z)$ gives the relation
$$
\frb_2( T(t) ) = \frb_2(0)\mathrm{exp}\left( \int^{T(t)}_{0} f_2\big( \fra_2(s) \big) \, \rd s \right) = \frb_1(0) \mathrm{exp}\left( \int^{t}_{0} g_2(s)\dot T(s) \, \rd s \right),
$$
which combines with (\ref{eq:tdfirst},\ref{eq:tdsecnd}) to yield the equality
$$
\dot T(t) = \frac{1}{1 - r(t)} \qquad \text{where} \qquad r(t) := \mathrm{exp}\left( -\int^{t}_{0} g_1(s) \, \rd s \right)\int^{t}_{0} g_{12}(s) \mathrm{exp}\left(\int^{s}_{0} g_1(z) \, \rd z \right) \, \rd s
$$
for the transition map. Thus
$$
\frac{ \frb_1(t) }{\frb_2(T(t))} = \dot T(t) = \frac1{1 - r(t)},
$$
and so in particular the identity
\begin{align}\label{eq:dotTbound}
\dot T(t) = 1 + r(t)\frac{ \frb_1(t) }{\frb_2(T(t))} &= 1 + \mathrm{exp}\left( - \int^{t}_{0} g_2(s) \dot T(s) \, \rd s \right)\int^{t}_{0} g_{12}(s) \mathrm{exp}\left(\int^{s}_{0} g_1(z) \, \rd z \right) \, \rd s \nonumber \\
&:= 1 +  \mathrm{exp}\left( - \int^{t}_{0} g_2(s) \dot T(s) \, \rd s \right)h_{12}(t)
\end{align}
necessarily holds. To conclude, note first that $T = \frw_2 \circ \fra_1$ is a composition of increasing functions, so $\dot T > 0$ and the trivial bound
$$
T(t) = \frw_2\big( \fra_1(t) \big) \leq \frw_2\left( \frac{\frb_*}{\frf^*}\left( \re^{\frf^* t} - 1 \right) \right) \leq \frac1{\frf^*}\log\left( \frac1{2 - \re^{\frf^* t} } \right)
$$
holds by the first two parts of the lemma. In particular, the bound
\begin{align*}
\mathrm{exp}\left( -\int^{t}_{0} g_2(s) \dot T(s) \, \rd s \right) \leq \re^{ \frf^{*} T(t) } \leq \frac1{2 - \re^{\frf^* t} }
\end{align*}
as well as the simple estimate
$$
|h_{12}(t)| \leq \frd \frf^{*} \, t\,\re^{\frf^{*}t}
$$
must hold on $[0,T_{**}),$ which when combined with \eqref{eq:dotTbound} show that the differential inequalities
\begin{align*}
1 - \frd \frf^{*} t \,\re^{\frf^{*}(t+T(t))}  &\leq \dot T(t) \leq 1 + \frd \frf^{*} t \,\re^{\frf^{*}(t+T(t))}\\
1 - \frd \frf^{*} \frac{t \,\re^{\frf^{*}t} }{2 - \re^{\frf^* t} } &\leq \dot T(t) \leq 1 + \frd \frf^{*} \frac{t \,\re^{\frf^{*}t} }{2 - \re^{\frf^* t} }
\end{align*}
must hold on $[0,T_{**})$ as well. The claimed bound
$$
\left| \frac{T(t) - t}{t^2} \right| \leq  \frac{\frd \frf^{*} }{2} \re^{\frf^{*} (t+T(t))} \leq \frac{\frd \frf^{*} }{2}\left( \frac{\re^{\frf^{*} t} }{2 - \re^{\frf^* t} } \right)
$$
then follows by integrating in time. The fourth statement follows my checking that the piecewise maps 
\begin{align*}
T_{+}(t) := \max\left\{ T_{12}(t) , T_{21}(t) \right\} \qquad \text{and} \qquad T_{-}(t) := \min\left\{ T_{12}(t) , T_{21}(t) \right\}
\end{align*}
have the desired properties. Both are continuously differentiable outside of the set
$$
\mathcal{N} := \left\{ t \geq 0 : \fra_1(t) = \fra_2(t) , \dot \fra_1(t) \neq \dot \fra_2(t) \right\},
$$
which consists of purely isolated points and is therefore countable. The remaining properties then follow trivially from the earlier statements in the proposition.
\end{proof}

\begin{proposition}[Proposition \ref{prop:linearA}]\label{prop:linearAapp}
For $i=1,2$ let $f_i\in C([0,\infty);L^{1}(\T))$ and $g_i \in L^{\infty}([0,\infty);W^{1,1}(\T))$ obey the global bounds
\begin{align*}
\sup_{t \geq 0}\;\, |\mu_{f_i}(t)| &\leq \frf^{*}, \qquad\, \text{and} \qquad \quad \,\sup_{t \geq 0}\;\,\left| \mu_{f_1}(t) - \mu_{f_2}(t)\right| \leq \frd \frf^{*} \\
\sup_{t \geq 0}\;\, \|g_i(\cdot,t)\|_{W^{1,1}(\T)} &\leq \frg^{*}, \qquad \text{and} \qquad \|g_1(\cdot,t) - g_2(\cdot,t)\|_{L^{2}(\T)} \leq \frd \frg^*,
\end{align*}
and let $\vth_0 \in H^{2}(\T)$ and $\veps_0 > 0$ denote arbitrary initial data. Then the following hold ---
\begin{enumerate}[{\rm i)}]
\item For any time $T < T_{*}$ with
$$
T_{*} := \frac{1}{ \veps_0 \frf^{*}},
$$
there exist unique mild solutions $\vth_i \in C( [0,T];H^{2}(\T) ), \veps_{i} \in C^{1}([0,T];\R)$ to the initial value problems
\begin{align*}
\partial_{t} \vth_{i} &= \veps_i(t)\ddot \vth_i + \veps_i(t)g_i &&\vth(x,0) = \vth_0(x) \in H^{2}(\T)\\
\dot \veps_i(t) &= - \veps^2_i(t)\mu_{f_i}(t) && \;\,\,\, \veps_i(0) = \veps_0 > 0,
\end{align*}
and the solution $\vth_i$ exists for as long as $\veps_i$ remains finite.
\item There exists a continuous, increasing function $\Gamma_{\vth_0} : \R^{+} \mapsto \R^{+},$ depending only on the initial datum $\vth_0,$ so that the properties
$$
\Gamma_{\vth_0}(0) = 0 \qquad \text{and} \qquad \Gamma_{\vth_0}(t) \nearrow \Gamma_{\vth_0}(\infty) = \| \vth_0 \|_{\dot{H}^{2}(\T)}
$$
hold.
\item The solutions $\vth_i,\veps_i$ obey the bounds
\begin{align*}
|\veps_i(t) - \veps_i(s)| & \leq \veps_0\left(1 - \frac{t}{T_{*}}\right)^{-2}\frac{t-s}{T_{*}}\\
\left| \mu_{\vth_i}(t) - \mu_{\vth_i}(s)\right| &\leq \frg^{*}\left( \zeta_i(t) - \zeta_i(s) \right),\qquad \quad \zeta_i(t) := \int^{t}_{0} \veps_i(s) \, \rd s,\\
\| \vth_i(\cdot,t) - \vth_i(\cdot,s) \|_{\dot{H}^{2}(\T)} &\leq \Gamma_{\vth_0}\left( \zeta_i(t) - \zeta_i(s) \right) + \frg^{*}\left( \zeta_i(t) - \zeta_i(s) \right)^{\frac14}
\end{align*}
for all pairs of times $0 \leq s < t$ at which the solutions exist.
\item If $g_1=g_2=0$ then the corresponding homogeneous solutions $\vth^{h}_{i}$ obey the difference bound
$$
\| \vth^{h}_1(\cdot,t) - \vth^{h}_2(\cdot,t)\|_{H^{2}(\T)} \leq  \frac{ \frd\frf^{*} }{4\frf^{*}}\left(1 - \frac{t}{T_*}\right)^{-1}\log^2\left(1-\frac{t}{T^*}\right)\|\vth_0\|_{\dot H^{2}(\T)}
$$
for as long as both exist. 
\item For any $0 < s < 1/2$ the differences $\xi_i := \vth_{i} - \vth^{h}_{i}$ obey the estimates
$$
\|\xi_i(\cdot,t)\|_{H^{2+s}(\T)} \leq \frg^{*}\big( \frc_{s} + \zeta_i(t) \big)
$$
for $0 < \frc_s < \infty$ a finite constant depending only on the modulus $s$ of regularity. Moreover, the estimate
$$
\|\xi_i(\cdot,t) - \xi_i(\cdot,s)\|_{H^{2}(\T)} \leq \frg^{*}\left( \big(\zeta_i(t) - \zeta_i(s)\big)^{\frac14} + \zeta_i(t) - \zeta_i(s) \right)
$$
holds for all pairs of times $0 \leq s < t$ at which the solutions exist.
\item The non-homogeneous solution map $(f_i,g_i) \mapsto \xi_i$ is H\"older continuous with respect to the $H^{2}(\T)$ topology, in the sense that the difference $\delta \xi = \xi_1 - \xi_2$ obeys the bound
\begin{align*}
&\|\delta \xi(\cdot,t) \|_{H^{2}(\T)} \leq \mathfrak{C}^{*}\left( \frac{T_*}{\frf^*(T_*-t)},\frf^* \right)\left[ \frd \frg^{*} \left( 1 + \log^{+}\left( \frac{\frg^* }{\frd \frg^{*}}\zeta^{\frac14}_1(t) \vee \zeta^{\frac14}_2(t) \right) \right) + \frac{\frd \frf^*}{\frf^*} \vee \left( \frac{\frd \frf^*}{\frf^*} \right)^{\frac14}\right]
\end{align*}
on the common interval $0\leq t < T_*$  where both solutions exist.
\end{enumerate}
\end{proposition}

\begin{proof}
Fix $f = f_i, g = g_i, \veps = \veps_i , \vth = \vth_i, \zeta = \zeta_i$ for some $i \in \{1,2\}$. The Fourier coefficient
$$
\mu_{f}(t) = \fint_{\T} f(x,t) \, \rd x
$$
defines a continuous function of time by hypothesis, and trivially obeys the bound
$$
|\mu_{f}(t)| \leq \frf^{*}
$$
by hypothesis. The existence and $C^{1}$ regularity of a unique $\veps$ on $[0,T_{*})$ follows by explicitly solving the ODE
$$
\dot \veps(t) = -\veps^2(t) \mu_{f}(t) \qquad \longrightarrow \qquad \veps(t) = \frac{\veps_0}{1 + \veps_0 \int^{t}_{0} \mu_{f}(s) \, \rd s}
$$
and performing the simple estimate
$$
1 + \veps_0 \int^{t}_{0} \mu_f(s) \, \rd s \geq 1 - \veps_0 \frf^{*} t
$$
on the denominator from below. The temporal Lipschitz bound
$$
|\veps(t) - \veps(s)| \leq \veps_0\left(1 - \frac{t}{T_*}\right)^{-2}\frac{t-s}{T_{*}}
$$
follows by explicitly computing $\veps(t)-\veps(s)$ and estimating the denominator in a similar manner.

The Fourier coefficients $\{ \hat \vth_k(t) \}_{k \in \Z}$ of any mild solution must then satisfy
\begin{align*}
\hat \vth_{k}(t) &= \re^{-k^2 \zeta(t)} \hat \vth_{k}(0) + \int^{t}_{0} \re^{-k^2( \zeta(t) - \zeta(s) ) } \veps(s) \hat g_{k}(s) \, \rd s \qquad \text{where} \qquad
\zeta(t) := \int^{t}_{0} \veps(z) \, \rd z,
\end{align*}
and the fact that $\vth \in C([0,T_*);H^{2}(\T))$ will follow by proving the remaining statements of the proposition. For the second statement, simply define $\Gamma_{\vth_0}$ as the function
$$
\Gamma_{\vth_0}(t) := \left( \sum_{k \in \Z_0} \left( 1 - \re^{-k^2t} \right)^{2} k^4 |\hat \vth_k(0)|^2 \right)^{\frac12},
$$
and note that since $\vth_0 \in H^{2}(\T)$ it has the requisite properties by the dominated convergence theorem. For $0 \leq s < t$ the trivial estimate
$$
|\hat \vth_{0}(t) - \hat \vth_{0}(s)| = \left| \int^{t}_{s} \veps(s) \hat g_0(s) \, \rd s \right| \leq \frac{\frg^{*}}{2\pi}\left( \zeta(t) - \zeta(s) \right)
$$
gives the first conclusion in the third statement. To see the second conclusion in the third statement, note the identity
\begin{align*}
\hat \vth_{k}(t) - \hat \vth_{k}(s) &= \re^{-k^2 \zeta(s)}\left( \re^{-k^2( \zeta(t) - \zeta(s) )} - 1 \right)\hat \vth_k(0) + \int^{t}_{s} \re^{-k^2( \zeta(t) - \zeta(z) ) } \veps(z) \hat g_{k}(z) \, \rd z \\
& + \left( \re^{-k^2(\zeta(t) - \zeta(s))} - 1 \right) \int^{s}_{0} \re^{-k^2( \zeta(s) - \zeta(z) ) } \veps(z) \hat g_{k}(z) \, \rd z := \mathrm{I}_{k} + \mathrm{II}_{k} + \mathrm{III}_{k}
\end{align*}
holds, and therefore so does the estimate
$$
\| \vth(\cdot,t) - \vth(\cdot,s)\|_{\dot H^{2}(\T)} \leq \left( \sum_{k \in \Z_0} k^4 |\mathrm{I}_k|^2 \right)^{\frac12} + \left( \sum_{k \in \Z_0} k^4 |\mathrm{II}_k|^2 \right)^{\frac12} + \left( \sum_{k \in \Z_0} k^4 |\mathrm{III}_k|^2 \right)^{\frac12} 
$$
by the triangle inequality. The desired inequality then follows by estimating each sum in turn. The first sum is easily estimated from above by
$$
\left( \sum_{k \in \Z} k^4 |\mathrm{I}_k|^2 \right)^{\frac12} \leq \Gamma_{\vth_0}\left( \zeta(t) - \zeta(s) \right),
$$
while for $k \ne0 $ the basic inequalities $2\pi |k||\hat g_k(z)| \leq \frg^*$ and
$$
\left| \int^{t}_{s} \re^{-k^2( \zeta(t) - \zeta(z) )} \veps(z) \hat g_k(z) \, \rd z \right| \leq \frac{\frg^{*}}{2\pi} \int^{t}_{s} \re^{-k^2( \zeta(t) - \zeta(z) )} \veps(z) \, \rd z = \frac{\frg^{*}}{2\pi} |k|^{-3} \left( 1 - \re^{-k^2( \zeta(t) - \zeta(s))} \right)
$$
gives the corresponding estimates
\begin{align*}
\left( \sum_{k \in \Z} k^4 |\mathrm{II}_k|^2 \right)^{\frac12} &\leq \frac{\frg^{*}}{2\pi} \left( \sum_{k \in \Z_0} \left(\frac{1 - \re^{-k^2(\zeta(t) - \zeta(s))}}{k}\right)^{2} \right)^{\frac12} \qquad \text{and} \\
\left( \sum_{k \in \Z} k^4 |\mathrm{III}_k|^2 \right)^{\frac12} &\leq \frac{\frg^{*}}{2\pi} \left( \sum_{k \in \Z_0} \left(\frac{1 - \re^{-k^2(\zeta(t) - \zeta(s))}}{k}\right)^{2}\left( 1 - \re^{-k^2\zeta(s)} \right)^{2} \right)^{\frac12}
\end{align*}
for the second and third sums, respectively. For any $T > 0$ the real-valued function $f(u) := u^{-1}\left( 1 - \re^{-uT}\right)$ is decreasing, and so the estimate
\begin{align}\label{eq:seriesest}
\sum_{k \geq 1} k^{-2}\left(1 - \re^{-k^2T}\right)^2 &\leq \sum_{k \geq 1} k^{-2}\left( 1 - \re^{-k^2 T} \right) \leq \int^{\infty}_{0} u^{-2}\left( 1 - \re^{-u^2 T} \right) \, \rd u = \sqrt{\pi T}
\end{align}
holds by the integral test. Combining the estimates for all three terms therefore gives the inequality
$$
\|\vth(\cdot,t) - \vth(\cdot,s)\|_{\dot H^{2}(\T)} \leq \Gamma_{\vth_0}\left( \zeta(t)-\zeta(s) \right) + \frg^{*}\left( \zeta(t) - \zeta(s) \right)^{\frac14}
$$
claimed in the third statement.

To prove the fourth statement of the proposition, let $\fra_i,\frb_i$ denote any $C^{1}(\R^{+})$ solutions to the systems
\begin{align*}
\dot \fra_i(t) &= \frb_i(t) \qquad \qquad \quad \quad \;\;\;\,\;\fra_i(0) = 0 \nonumber \\
\dot \frb_i(t) &= \frb_i(t) \mu_{f_i}\big( \fra_i(t) \big) \qquad \quad \frb_i(0) = \frac1{\veps_0}
\end{align*}
of ordinary differential equations. Let $\frw_i$ denote the corresponding inverses of $\fra_i$ and $T_{ij} := \frw_i \circ \fra_j$ the correponding transition maps. By lemma \ref{lem:mapsapp}, as long as
$$
t < T_{**} := \frac{\log 2}{\frf^{*}}
$$
both transition maps are well-defined and obey 
$$
\zeta_i \big( \fra_i(t) \big) = t
$$
on $[0,T_{**})$ by the argument following \eqref{eq:timemaps}, so $\zeta_i = \frw_i$ must hold on $[0,T_{*})$ as well. Pick $0 < z < T_{*}$ arbitrary, so that $z \in \mathrm{dom}(\frw_i)$ and thus $z = \fra_1(t)$ for $t = \frw_1(z),$ and assume $\fra_1(t) \geq \fra_2(t)$ without loss of generality. Then since $\zeta_1(z) = \frw_1 \circ \fra_1(t) = t$ and $\zeta_2(z) = \frw_2 \circ \fra_1(t) = T_{+}(t)$ the identities
\begin{align*}
\hat \vth^{h}_{1,k}(z) = \re^{-k^2 t} \hat \vth_{k}(0) \qquad \text{and} \qquad 
\hat \vth^{h}_{2,k}(z) = \re^{-k^2 T_{+}(t)} \hat \vth_{k}(0)
\end{align*}
must hold for the homogeneous solutions. Let $T_{+}(t) = t + D_{+}(t)$ for $D_{+}(t) \geq 0$ and note that the relation
$$
\|\vth^{h}_2(\cdot,z) - \vth^{h}_1(\cdot,z)\|_{\dot H^{2}(\T) } = \left( \sum_{k \in \Z_0} k^4 \re^{-2 k^2 t }\left(1 - \re^{-k^2 D_{+}(t)}\right)^{2}|\hat \vth_k(0)|^2 \right)^{\frac12}
$$
holds for the $\dot H^{2}(\T)$ semi-norm. The fact that $1-e^{-x} \leq x$ and the fact that $D_{+}(t) \geq 0$ combine to justify the simple estimate
\begin{align*}
k^4 \re^{-2 k^2 t }\left(1 - \re^{-k^2 D_{+}(t)}\right)^{2}|\hat \vth_k(0)|^2 &= \left(\frac{ D^{2}_{+}(t) }{t^2}\right) \big(k^4 t^2 \big) \re^{-2 k^2 t }\left(\frac{1 - \re^{-k^2 D_{+}(t)}}{k^{2}D_{+}(t)}\right)^{2} k^4|\hat \vth_k(0)|^2\\
& \leq \left(\frac{ D^{2}_{+}(t) }{t^2}\right) \big(k^4 t^2 \big) \re^{-2 k^2 t }  k^4|\hat \vth_k(0)|^2,
\end{align*}
while the fact that $x^2\re^{-2x} \leq 1/4$ for $x>0$ shows that in fact
$$
k^4 \re^{-2 k^2 t }\left(1 - \re^{-k^2 D_{+}(t)}\right)^{2}|\hat \vth_k(0)|^2 \leq \left( \frac{D_{+}(t)}{2t}\right)^{2} k^4|\hat \vth_k(0)|^{2}
$$
must hold as well. Appealing to lemma \ref{lem:mapsapp} part iv) and summing gives the overall estimate
$$
\|\vth^{h}_{1}(\cdot,z) - \vth^{h}_{2}(\cdot,z)\|_{\dot{H}^{1}(\T)} \leq \frac{ \frd\frf^{*}\,t }{4} \mathrm{exp}\left( \frf^{*}(t + T_{+}(t))\right)\|\vth_0\|_{\dot H^{2}(\T)}
$$
for the seminorm. Now $z = \fra_1(t),$ $t = \frw_1(z)$ and $T_{+}(t) = \frw_2(z)$ and so lemma \ref{lem:mapsapp} gives the temporal estimates,
\begin{align*}
&z \geq T_{*}\left( 1 - \re^{- \frf^* t } \right) \qquad  \qquad  &&\re^{\frf^* t} \leq \log\left(\frac1{1 - z/T^{*}}\right) \\
&\re^{ \frf^{*} T_{+}(t) } = \re^{ \frf^{*} \frw_2(z)} \leq \frac1{1 - z/T_{*}} \qquad \qquad &&t = \frw_1(z) \leq \frac1{\frf^{*}} \log\left(\frac1{1 - z/T^{*}}\right)
\end{align*}
which in turn imply the overall bound
$$
\|\vth^{h}_{1}(\cdot,z) - \vth^{h}_{2}(\cdot,z)\|_{\dot{H}^{2}(\T)} \leq \frac{ \frd\frf^{*} }{4\frf^{*}}\frac1{1 - z/T_*}\log^2\left(\frac1{1-z/T^*}\right)\|\vth_0\|_{\dot H^{2}(\T)}
$$
for the seminorm. But as the means of $\vth_1$ and $\vth_2$ coincide for all time, the full $H^{1}(\T)$ norm of the difference is, in fact, bounded as claimed.

For the fifth statement, let $\xi = \xi_i$ denote one of the differences and $g = g_i$ the corresponding forcing. Then
$$
\hat \xi_{k}(t) = \int^{t}_{0} \re^{-k^2(\zeta(t) - \zeta(s))} \veps(s) \hat g_{k}(s) \, \rd s,
$$
and so if $k \neq 0$ then the estimate
$$
|\hat \xi_{k}(t)| \leq \frac{\frg^{*}}{2\pi}\left( \frac{1 - \re^{-k^{2}\zeta(t)}}{|k|^3}\right)
$$
holds as before. The $\dot H^{2+s}(\T)$ semi-norm therefore satisfies
$$
\|\xi(\cdot,t)\|_{\dot H^{2+s}(\T)} \leq \frac{\frg^{*}}{2\pi} \left( \sum_{k \in \Z_0} k^{4+2s}\frac{\big(1 - \re^{-k^2\zeta(t)}\big)^{2}}{k^{6}} \right)^{\frac12} \leq \frac{\frg^{*}}{\sqrt{2}\pi} \left( \sum_{k \geq 1} k^{2s-2} \right)^{\frac12} := \frc^{*}_{s} \frg^{*},
$$
for some $\frc_{s},$ as $s < 1/2$ so the series is summable; the mean clearly obeys
$$
|\mu_{\xi}(t)| \leq \frac{\frg*}{2\pi}\zeta(t),
$$
and so the claimed bound in $H^{2+s}$ holds.  Similarly, the final bound follows by bounding
\begin{align*}
&|\hat \xi_k(t) - \hat \xi_k(s)| \leq \frac{\frg^{*}}{2\pi} \int^{t}_{s} \re^{-k^{2}(\zeta(t) - \zeta(s))}\veps(s) \, \rd s, \\
 &\left(\sum_{k \in \Z_0} k^4|\hat \xi_k(t) - \hat \xi_k(s)|^{2} \right)^{\frac12} \leq \frac{\frg^{*}}{2\pi}\left( \sum_{k \in \Z_0} \left(\frac{1 - \re^{-k^2(\zeta(t)-\zeta(s))}}{k}\right)^{2}\right)^{\frac12} \leq \frg^{*}\big( \zeta(t) - \zeta(s) \big)^{\frac14}
\end{align*}
and summing as before. 

For the sixth and final statement, let $\delta \xi := \xi_1 - \xi_2$ denote the between non-homogeneous solutions and recall that
\begin{align*}
\widehat{\delta \xi}_{k}(t) &= \int^{t}_{0} \re^{-k^{2}(\zeta_1(t) - \zeta_1(s))}\veps_{1}(s) \hat g_{1,k}(s) \, \rd s - \int^{t}_{0} \re^{-k^2(\zeta_2(t) - \zeta_2(s))}\veps_2(s)\hat g_{2,k}(s) \, \rd s
\end{align*}
by definition of mild solution. First, set $\delta g = g_1 - g_2$ and note that the difference in means obeys
$$
\left|\widehat{\delta \xi}_{0}(t)\right| = \left| \int^{t}_{0} \veps_1(s) \hat g_{1,0}(s) - \int^{t}_{0} \veps_2(s) \hat g_{2,0}(s) \right| = \left| \int^{z}_{0} \left( \widehat{ \delta g}_k \left( \fra_1(u) \right) + (1 - \dot T_{21}(u)) \hat g_{2,0}\big( \fra_1(u) \big) \right)\, \rd u \right|
$$
after applying the change of variables $z = \zeta_1(t),u = \zeta_1(s)$ to obtain the second equality. Apply the triangle inequality and lemma \ref{lem:mapsapp} to obtain the upper bound
$$
\left|\widehat{\delta \xi}_{0}(t)\right| \leq \frac{ \frd \frg^{*}(t) }{\sqrt{2\pi}} z + \frac{ \frg^{*} \frd \frf^{*}}{2\pi}\frac{z^2}{2} \re^{ \frf^{*}(z + T_{21}(z)) },
$$
then use lemma \ref{lem:mapsapp} to obtain the bounds and overall estimate
\begin{align}\label{eq:meanHolder}
z &\leq \frac1{\frf^{*}}\log\left(\frac1{1-\frac{t}{T_*}}\right) \qquad \text{and} \qquad \re^{\frf^{*}(z + T_{\pm}(z))} \leq \frac1{1 - \frac{t}{T_{*}}}\log\left(\frac1{1-\frac{t}{T_*}}\right), \nonumber \\
\left|\widehat{\delta \xi}_{0}(t)\right| &\leq \frac{\frd \frg^{*}(t)}{\sqrt{2\pi}\frf^{*}} \log\left( \frac{T_*}{T_*-t} \right) + \frac{\frg^*}{4\pi}\frac{\frd \frf^{*}}{\frf^{*}}\log^3\left( \frac{T_*}{T_*-t} \right)\left( \frac{T_*}{\frf^*(T_*-t)}\right)
\end{align}
for the difference in means. Next, for $k \neq 0$ perform the decomposition
\begin{align*}
\widehat{\delta \xi}_{k}(t) &= \int^{t}_{0} \re^{-k^2(\zeta_1(t) - \zeta_1(s))}\veps_1(s)\widehat {\delta g}_{k}(s) \, \rd s + \\
&\int^{t}_{0} \left(\re^{-k^2(\zeta_1(t) - \zeta_1(s))}\veps_1(s) - \re^{-k^2(\zeta_2(t) - \zeta_2(s))}\veps_2(s) \right)\hat g_{2,k}(s) \, \rd s  := \mathrm{I}_{k} + \mathrm{II}_{k}
\end{align*}
of the difference. For the first term, let $z := \zeta_1(t)$ and use the change of variables $u = \zeta_1(s)$ to uncover
$$
\mathrm{I}_k = \int^{z-\veps}_{0} \re^{-k^2(z-u)} \widehat{ \delta g}_k \left( \fra_1(u) \right) \, \rd u +  \int^{z}_{z-\veps} \re^{-k^2(z-u)} \widehat{ \delta g}_k \left( \fra_1(u) \right) \, \rd u =: \mathrm{I}^{\veps}_k + \overline{ \mathrm{I}^{\veps}_k }
$$
for $0 \leq \veps \leq z$ an arbitrary parameter. Apply Minkowski's inequality to obtain the upper bound
\begin{align*}
\left\| \mathrm{I}^{\veps}_{k} \right\|_{\dot H^{2}(\T)} &= \sqrt{2\pi} \left( \sum_{k \in \Z_0} \left| \int^{z-\veps}_{0} k^2 \re^{-k^2(z-u)} \widehat{ \delta g}_k\left( \fra_1(u) \right) \, \rd u \right|^2 \right)^{\frac12} \\
& \leq \sqrt{2\pi} \int^{z-\veps}_{0} \left( \sum_{k \in \Z_0} k^4 \re^{ -2k^2(z-u) } \left| \widehat{ \delta g}_{k}\left( \fra_1(u) \right) \right|^2 \right)^{\frac12} \, \rd u
\end{align*}
and then apply the trivial inequality
$$
k^4 \re^{-2k^2(z-u)} \leq \left( \frac1{\re(z-u)} \right)^2
$$
together with an explicit integration to obtain the overall bound
\begin{align*}
\left\| \mathrm{I}^{\veps}_{k} \right\|_{\dot H^{2}(\T)} &\leq \frac1{\re} \int^{z-\veps}_{0} \frac1{z-u} \left\| \delta g\left( \cdot , \fra_1(u) \right) \right\|_{L^{2}(\T)} \, \rd u \leq \frac{\frd \frg^{*}(t)}{2}\log \frac{z}{\veps}
\end{align*}
for the first sub-term. Note that for $k \neq 0$ the uniform bound
$$
\left| \overline{ \mathrm{I}^{\veps}_k } \right| \leq \frac{\frg^{*}}{\pi |k|^3}\left( 1 - \re^{-k^2 \veps} \right)
$$
holds by direct integration, and so the overall estimate
$$
\left\| \overline{ \mathrm{I}^{\veps}_k }  \right\|_{\dot H^{2}(\T) } \leq \frac{2\,\frg^{*} }{\sqrt{\pi} }\left( \sum_{k \geq 1} k^{-2}\left( 1 - \re^{-k^2 \veps} \right) \right)^{\frac12} \leq  2\,\frg^{*} \veps^{\frac14}
$$
holds for the second sub-term. All-together, the bound
$$
\left\| \mathrm{I}_{k} \right\|_{\dot H^{2}(\T)} \leq \frac{\frd \frg^{*}(t)}{2}\log \frac{z}{\veps} + 2\,\frg^{*} \veps^{\frac14}
$$
holds for $0 \leq \veps \leq z$ arbitrary. Now in the case that the choice
$$
\veps_* := \left( \frac{ \frd \frg^*(t)}{\frg^* } \right)^{4}
$$
is valid, i.e. $\veps_* \leq z$, then the overall upper bound
$$
\left\| \mathrm{I}_{k} \right\|_{\dot H^{2}(\T)} \leq 2\, \frd \frg^{*}(t) \left( 1 + \log^{+}\left( \frac{\zeta^{\frac14}_1(t) \frg^* }{\frd \frg^{*}(t)} \right) \right)
$$
holds. Otherwise, take $\veps=z$ and use $z \leq \veps_*$ to see that the overall upper bound
\begin{align}\label{eq:partialHolder}
\left\| \mathrm{I}_{k} \right\|_{\dot H^{2}(\T)} \leq 2\, \frd \frg^{*}(t) \leq 2\, \frd \frg^{*}(t) \left( 1 + \log^{+}\left( \frac{\zeta^{\frac14}_1(t) \frg^* }{\frd \frg^{*}(t)} \right) \right)
\end{align}
holds in this case as well. For the second term, if $t < T_{*}$ then $t \in \mathrm{dom}(\frw_i)$ and so $t = \fra_1(z), z = \frw_1(t)$ and $\fra_1(z) \geq \fra_2(z)$ without loss of generality. As $\zeta_i = \frw_i$ on $[0,T_{*}),$ the equality 
$$
\left| \mathrm{II}_{k} \right|= \left|\int^{\fra_1(z)}_0\left(\re^{-k^2(z - \frw_1(s))}\veps_1(s) - \re^{-k^2(T_{+}(z) - \frw_2(s))}\veps_2(s) \right)\hat g_{2,k}(s) \, \rd s \right|
$$
therefore holds for the second term. Define the functions $\frm(u) = \min\{ \fra_1(u),\fra_2(u) \}$ and $\frn(s) := \max\{\frw_1(s),\frw_2(s)\},$ so that $\frn \circ \frm(u) = u$ and $\dot \frn \big( \frm(u) \big) \dot \frm(u) = 1$ holds Lebesgue almost everywhere. Note that if there exists $z \leq u \leq T_{+}(z)$ with $\fra_1(u) < \fra_2(u)$ then $T_{+}(z) = \frw_2 \circ \fra_1(z) \leq \frw_2 \circ \fra_1(u) < u \leq z$ would hold since $\frw_2 \circ \fra_1$ is increasing, which is a contradiction. In other words, if $z \leq u \leq T_{+}(z)$ then $\frm(u) = \fra_2(u)$ and $T_{-}(u) = \frw_1 \circ \fra_2(u)$ must hold. Now perform the change of variables $s = \frm(u)$ to obtain the further decomposition
\begin{align*}
\left| \mathrm{II}_{k}\right| &= \left|\int^{z}_0\left(\re^{-k^2(z - \frw_1 \circ \frm(u))}\veps_1 \circ \frm(u) - \re^{-k^2(T_{+}(z) - \frw_2\circ \frm(u))}\veps_2 \circ \frm(u) \right)\hat g_{2,k}\big( \frm(u) \big)\dot \frm(u) \, \rd u \right. \\
& + \left. \int^{T_{+}(z)}_{z} \left( \re^{-k^2(z - T_{-}(u))} \dot T_{-}(u) - \re^{-k^2(T_{+}(z) - u)}\right)\hat g_{2,k}\big( \fra_2(u) \big) \, \rd u \right| := \left|\mathrm{II}^{{\rm i}}_{k} + \mathrm{II}^{{\rm ii}}_{k}\right|
\end{align*}
of the second term. Define $z + D_{+}(z) = T_{+}(z)$ and $u + D_{-}(u) = T_{-}(u),$ so that the bound on $|\dot D_{-}(u)|$ from lemma \ref{lem:mapsapp} yields the overall estimate
\begin{align*}
\left|\mathrm{II}^{{\rm i}}_{k}\right| &= \left| \int^{z}_{0}\re^{-k^2(z-u)}\left(1 - \re^{-k^2(D_{+}(z) - D_{-}(u))}\right)\left[ \1_{\fra_1\leq\fra_2} + \re^{-k^2 D_{-}(u)} \1_{\fra_1>\fra_2}\right]\hat g_{2,k}\big(\frm(u)\big) \, \rd u \right.\\
&+\left. \int^{z}_{0} \re^{-k^2(z-u)}\left[ \re^{-k^2(D_{+}(z) - D_{-}(u))}\1_{\fra_1\leq \fra_2} - \re^{-k^2 D_{-}(u)}\1_{\fra_1>\fra_2} \right]\dot D_{-}(u) \hat g_{2,k}\big( \frm(u) \big) \, \rd u\right| \\
& \leq \frac{\frg^{*}}{2\pi|k|^3} \left[ \left(1 - \re^{-k^2 D_{+}(z)}\right)\left(1 - \re^{-k^2 z}\right) + \frd \frf^{*} \re^{\frf^{*}(z+T_{-}(z))}z \right]
\end{align*}
for the first term in this decomposition. As $T_{-}\circ T_{+}(z) = z,$ a similar argument yields the estimate
\begin{align*}
\left| \mathrm{II}^{{\rm ii}}_{k}\right| &:= \left|\int^{T_{+}(z)}_{z} \left( \re^{-k^2(z - T_{-}(u))} \dot T_{-}(u) - \re^{-k^2(T_{+}(z) - u)}\right)\hat g_{2,k}\big( \fra_2(u) \big) \, \rd u\right| \\
&\leq \frac{\frg^{*}}{2\pi|k|^3}\left( 2 - \re^{-k^2D_{-}(z)} - \re^{-k^2 D_{+}(z)} \right)
\end{align*}
for the second term in the decomposition, and as a consequence, the overall upper bound
$$
\mathrm{II}_{k} \leq \frac{\frg^{*}}{\pi |k|^3}\left(1 - \re^{-k^2 D_{+}(z)} \right) + \frac{\frg^{*}}{2\pi |k|^3}\left(1 - \re^{-k^2 D_{-}(z)} +  \frd \frf^{*} \re^{\frf^{*}(z+T_{-}(z))}z \right)
$$
holds for the second term. The upper bound
$$
\left( \sum_{k \in \Z_0} |k|^4|\mathrm{II}_{k}|^2 \right)^{\frac12} \leq \frg^*\left( D^{\frac14}_{+}(z) + D^{\frac14}_{-}(z) + \frd \frf^{*}z \re^{\frf^{*}(z+T_{-}(z))} \right)
$$
then follows by summing each series as before. Now recall once again that lemma \ref{lem:mapsapp} gives the bounds
$$
D_{\pm}(z) \leq \frd \frf^{*}\frac{z^2}{2} \re^{\frf^{*}(z + T_{+}(z)) }, \qquad z \leq \frac1{\frf^{*}}\log\left(\frac1{1-\frac{t}{T_*}}\right) \qquad \text{and} \qquad \re^{\frf^{*}(z + T_{\pm}(z))} \leq \frac1{1 - \frac{t}{T_{*}}}\log\left(\frac1{1-\frac{t}{T_*}}\right),
$$
and so the overall upper bound
$$
\left( \sum_{k \in \Z_0} |k|^4|\mathrm{II}^{{\rm i}}_{k}|^2 \right)^{\frac12} \leq \frg^{*}\left( 2 \left(\frac{ \frd \frf^{*} }{\frf^{*}}\right)^{\frac14}\log^{\frac34}\left(\frac{T_{*}}{T_{*}-t}\right)\left(\frac{T_*}{\frf^{*}(T_* - t)}\right)^{\frac14} + \frac{\frd \frf^{*}}{\frf^{*}}\log^2\left(1 - \frac{t}{T_*}\right)\left( \frac{T_*}{T_*-t} \right) \right)
$$
holds. Combining this estimate with (\ref{eq:meanHolder},\ref{eq:partialHolder}) yields the claim.
\end{proof}

\bibliographystyle{plain}
\bibliography{InteractingCurvesBib}

\begin{thebibliography}{10}

\bibitem{SevenColony}
Multimedia gallery - colonies of bacteria do battle (image 4) | nsf - national
  science foundation.
\newblock \url{http://www.nsf.gov/news/mmg/mmg_disp.jsp?med_id=67192&from=mmg}.
\newblock Accessed: 2016-05-09.

\bibitem{Wirtinger}
Aaron Abrams, Jason Cantarella, Joseph H.~G. Fu, Mohammad Ghomi, and Ralph
  Howard.
\newblock Circles minimize most knot energies.
\newblock {\em Topology}, 42(2):381--394, 2003.

\bibitem{alfaro2008singular}
Matthieu Alfaro, Danielle Hilhorst, and Hiroshi Matano.
\newblock The singular limit of the allen--cahn equation and the
  fitzhugh--nagumo system.
\newblock {\em Journal of Differential Equations}, 245(2):505--565, 2008.

\bibitem{AFP}
Luigi Ambrosio, Nicola Fusco, and Diego Pallara.
\newblock {\em Functions of bounded variation and free discontinuity problems},
  volume 254.
\newblock Clarendon Press Oxford, 2000.

\bibitem{AGS}
Luigi Ambrosio, Nicola Gigli, and Giuseppe Savar{\'e}.
\newblock {\em Gradient flows: in metric spaces and in the space of probability
  measures}.
\newblock Springer Science \& Business Media, 2008.

\bibitem{AB11}
Ben Andrews and Paul Bryan.
\newblock Curvature bound for curve shortening flow via distance comparison and
  a direct proof of grayson's theorem.
\newblock {\em Journal f{\"u}r die reine und angewandte Mathematik},
  2011(653):179--187, 2011.

\bibitem{BeerPNAS09}
Avraham Be'Er, HP~Zhang, E-L Florin, Shelley~M Payne, Eshel Ben-Jacob, and
  Harry~L Swinney.
\newblock Deadly competition between sibling bacterial colonies.
\newblock {\em Proceedings of the National Academy of Sciences},
  106(2):428--433, 2009.

\bibitem{BeerPNAS10}
Avraham Be’er, Gil Ariel, Oren Kalisman, Yael Helman, Alexandra Sirota-Madi,
  HP~Zhang, E-L Florin, Shelley~M Payne, Eshel Ben-Jacob, and Harry~L Swinney.
\newblock Lethal protein produced in response to competition between sibling
  bacterial colonies.
\newblock {\em Proceedings of the National Academy of Sciences},
  107(14):6258--6263, 2010.

\bibitem{DvD}
Guy David.
\newblock {\em Wavelets and singular integrals on curves and surfaces}.
\newblock Springer, 2006.

\bibitem{semistrong}
Arjen Doelman, Tasso~J. Kaper, and Keith Promislow.
\newblock Nonlinear asymptotic stability of the semistrong pulse dynamics in a
  regularized gierer–meinhardt model.
\newblock {\em SIAM Journal on Mathematical Analysis}, 38(6):1760--1787, 2007.

\bibitem{GH}
Michael Gage, Richard~S Hamilton, et~al.
\newblock The heat equation shrinking convex plane curves.
\newblock {\em Journal of Differential Geometry}, 23(1):69--96, 1986.

\bibitem{goldstein1996interface}
Raymond~E Goldstein, David~J Muraki, and Dean~M Petrich.
\newblock Interface proliferation and the growth of labyrinths in a
  reaction-diffusion system.
\newblock {\em Physical Review E}, 53(4):3933, 1996.

\bibitem{Grayson}
Matthew~A Grayson et~al.
\newblock The heat equation shrinks embedded plane curves to round points.
\newblock {\em Journal of Differential geometry}, 26(2):285--314, 1987.

\bibitem{Huisken}
Gerhard Huisken.
\newblock A distance comparison principle for evolving curves.
\newblock {\em Asian Journal of Mathematics}, 2(1):127--133, 1998.

\bibitem{MvB16}
Scott~G. McCalla and James~H. von Brecht.
\newblock Fronts under arrest: Nonlocal boundary dynamics in biology.
\newblock {\em Phys. Rev. E}, 94(6):060401, 2016.

\bibitem{mccalla2018consistent}
Scott~G McCalla and James~H von Brecht.
\newblock Consistent dynamics of stripes formed by cell-type interfaces.
\newblock {\em SIAM Journal on Applied Dynamical Systems}, 17(4):2615--2633,
  2018.

\bibitem{pego1989front}
Robert~L Pego.
\newblock Front migration in the nonlinear cahn-hilliard equation.
\newblock In {\em Proceedings of the Royal Society of London A: Mathematical,
  Physical and Engineering Sciences}, volume 422, pages 261--278. The Royal
  Society, 1989.

\bibitem{petrich1994nonlocal}
Dean~M Petrich and Raymond~E Goldstein.
\newblock Nonlocal contour dynamics model for chemical front motion.
\newblock {\em Physical review letters}, 72(7):1120, 1994.

\bibitem{rubinstein1989fast}
Jacob Rubinstein, Peter Sternberg, and Joseph~B Keller.
\newblock Fast reaction, slow diffusion, and curve shortening.
\newblock {\em SIAM Journal on Applied Mathematics}, 49(1):116--133, 1989.

\bibitem{van2010front}
Peter van Heijster, Arjen Doelman, Tasso~J Kaper, and Keith Promislow.
\newblock Front interactions in a three-component system.
\newblock {\em SIAM Journal on Applied Dynamical Systems}, 9(2):292--332, 2010.

\bibitem{knot}
James~H von Brecht and Ryan Blair.
\newblock Dynamics of embedded curves by doubly- nonlocal reaction–diffusion
  systems.
\newblock {\em Journal of Physics A: Mathematical and Theoretical}, 50, 2017.

\bibitem{SIMA}
James~H von Brecht and Scott~G McCalla.
\newblock Nonlinear stability through algebraically decaying point spectrum:
  Applications to nonlocal interaction equations.
\newblock {\em SIAM Journal on Mathematical Analysis}, 46(6):3727--3760, 2014.

\end{thebibliography}

\end{document}